\documentclass[10pt,reqno]{amsart}

\usepackage{amssymb}
\usepackage{amsmath}
\usepackage{amsfonts, dsfont}
\usepackage{kotex}
\counterwithin{figure}{section}
\usepackage[colorlinks=true, pdfstartview=FitV, linkcolor=black, citecolor=black, urlcolor=black]{hyperref} 

\usepackage{amsthm} 
\usepackage[utf8]{inputenc} 
\usepackage{enumitem} 

\usepackage{graphicx}

\usepackage[utf8]{inputenc}
\usepackage{anyfontsize}

\usepackage{tikz}
\usetikzlibrary{
	calc,intersections,arrows.meta,patterns,
	decorations.pathmorphing,
	decorations.fractals
}

\DeclareGraphicsExtensions{.pdf,.png,.jpg}

\usepackage{color}

\theoremstyle{plain}

\newtheorem{thm}{Theorem}[section]
\newtheorem*{thm*}{Theorem}
\newtheorem{corollary}[thm]{Corollary}

\newtheorem{lemma}[thm]{Lemma}
\newtheorem{prop}[thm]{Proposition}

\newtheorem*{statement*}{Statement}
\newtheorem{assumption}[thm]{Assumption}

\theoremstyle{definition}   
\newtheorem{defn}[thm]{Definition}

\theoremstyle{remark}  
\newtheorem{remark}[thm]{Remark}
\newtheorem{example}[thm]{Example}

\newtheorem{innercustomgeneric}{\customgenericname}
\providecommand{\customgenericname}{}

\newcommand{\newcustomtheorem}[2]{\newenvironment{#1}[1]
	{\renewcommand\customgenericname{#2}
		\renewcommand\theinnercustomgeneric{##1}\innercustomgeneric}{\endinnercustomgeneric}}

\newcustomtheorem{customthm}{Theorem}

\newcommand\trho{\widetilde{\rho}}

\newcommand\bL{\mathbb{L}}

\newcommand\bR{\mathbb{R}}

\newcommand\bH{\mathbb{H}}
\newcommand\bZ{\mathbb{Z}}
\newcommand\bW{\mathbb{W}}

\newcommand\bS{\mathbb{S}}

\newcommand\bN{\mathbb{N}}

\newcommand\cA{\mathcal{A}}

\newcommand\cD{\mathcal{D}}
\newcommand\cF{\mathcal{F}}

\newcommand\cH{\mathcal{H}}
\newcommand\cI{\mathcal{I}}

\newcommand\cL{\mathcal{L}}

\newcommand\cM{\mathcal{M}}

\newcommand\cO{\mathcal{O}}

\newcommand{\domain}{\mathcal{O}}

\newcommand\la{\langle}
\newcommand\ra{\rangle}
\newcommand\dd{\,\mathrm{d}}
\newcommand\ee{\mathrm{e}}

\newcommand{\mysection}[1]{\section{#1}
	\setcounter{equation}{0}}

\newcount\pgfdecorationorder
\pgfkeys{/pgf/decoration/.cd,
	Koch angle/.store in=\pgfkochangle, Koch angle=85,
	Koch order/.code={\global\pgfdecorationorder=#1}, Koch order=1
}
\pgfdeclaredecoration{Koch}{calculate}{
	\state{calculate}[width=0pt, next state=draw, persistent precomputation={
		\pgfmathparse{\pgfdecoratedinputsegmentlength/(2*(1+cos(\pgfkochangle)))}%
		\let\pgfkochsegmentlength=\pgfmathresult%
		\pgfmathparse{\pgfkochsegmentlength*sin(\pgfkochangle)}%
		\let\pgfkochy=\pgfmathresult%
		\pgfmathparse{\pgfkochsegmentlength*(1 + cos(\pgfkochangle))}%
		\let\pgfkochxa=\pgfmathresult%
		\pgfmathparse{\pgfkochsegmentlength*(1 + 2*cos(\pgfkochangle))}%
		\let\pgfkochxb=\pgfmathresult%
	}]{}
	\state{draw}[width=\pgfdecoratedinputsegmentlength]{
		\pgfpathmoveto{\pgfpointorigin}%
		\pgfpathlineto{\pgfqpoint{\pgfkochsegmentlength pt}{0pt}}%
		\pgfpathlineto{\pgfqpoint{\pgfkochxa pt}{\pgfkochy pt}}%
		\pgfpathlineto{\pgfqpoint{\pgfkochxb pt}{0pt}}%
		\pgfpathlineto{\pgfqpoint{\pgfdecoratedinputsegmentlength}{0pt}}%
	}
	\state{final}{
		\global\advance\pgfdecorationorder by -1\relax%
		\ifnum\pgfdecorationorder>0\relax%
		\pgfgetpath\decoratedpath%
		\pgfsetpath\empty%
		\begin{pgfdecoration}{{Koch}{\pgfdecoratedpathlength}}%
			\pgfsetpath\decoratedpath
		\end{pgfdecoration}%
		\fi
}}

\begin{document}

	\title[$L_p$-theory for PDEs in non-smooth domains]
	{Weighted Sobolev space theory for the heat equation and the time-fractional heat equation in non-smooth domains}

	\thanks{The author was supported by a KIAS Individual Grant (MG095802) at Korea Institute for Advanced Study.}

	\author[J. Seo]{Jinsol Seo}
	\address[J. Seo]{School of Mathematics, Korea Institute for Advanced Study, 85 Hoegiro Dongdaemun-gu, Seoul 02455, Republic of Korea}
	\email{seo9401@kias.re.kr}


	
	\begin{abstract}
		We present a general $L_p$-solvability framework for both the classical and time-fractional heat equations in non-smooth domains under the zero Dirichlet boundary condition.
		We consider domains $\Omega$ admitting the Hardy inequality: There exists a constant $N>0$ such that 
		$$
		\int_{\Omega}\Big|\frac{f(x)}{d(x,\partial\Omega)}\Big|^2\dd x\leq N\int_{\Omega}|\nabla f|^2 \dd x\quad\text{for any}\quad f\in C_c^{\infty}(\Omega)\,.
		$$
		To illustrate the boundary behavior of solutions in a general framework, we employ a weight system composed of a superharmonic function and a distance function to the boundary.
		Further, we investigate applications to various non-smooth domains, including convex domains, domains with exterior cone condition, totally vanishing exterior Reifenberg domains, and domains $\Omega\subset\bR^d$ for which the Aikawa dimension of $\Omega^c$ is less than $d-2$.
		By using superharmonic functions tailored to the geometric conditions of the domain, we derive weighted $L_p$-solvability results for various non-smooth domains, with specific weight ranges that differ for each domain condition.
		In addition, we provide an application to the H\"older continuity of solutions in domains with the volume density condition, as well as pointwise estimates for solutions in Lipschitz cones.
	\end{abstract}
	
	\maketitle

	\setcounter{tocdepth}{3}
	
	\let\oldtocsection=\tocsection
	
	\let\oldtocsubsection=\tocsubsection
	
	\let\oldtocsubsubsection=\tocsubsubsectio
	
	\renewcommand{\tocsection}[2]{\hspace{0em}\oldtocsection{#1}{#2}}
	\renewcommand{\tocsubsection}[2]{\hspace{1em}\oldtocsubsection{#1}{#2}}
	\renewcommand{\tocsubsubsection}[2]{\hspace{2em}\oldtocsubsubsection{#1}{#2}}
	
	\tableofcontents
	
	\mysection{Introduction}\label{sec:Introduction}
	The heat equation is one of the most classical partial differential equations.
	Historically, $L_p$-theory for this equation in $\bR^d$ and $C^2$-domains was developed in parallel with Schauder- and $L_2$-theories.
	Extensions have been extensively developed in various directions, including variable coefficients \cite{Krylov2007}, semigroups \cite{Pseudodiff, semigroup}, and non-smooth domains.
	
	In this paper, we develop a general $L_p$-theory for the heat equation in \textit{non-smooth} domains $\Omega$, under the zero-Dirichlet boundary condition:
	\begin{align}
		\partial_tu=\Delta u+f\quad\text{in}\,\,\,(0,\infty)\times \Omega\quad;\quad u(0,\cdot)=u_0\,\,,\,\,u|_{(0,\infty)\times \partial\Omega}\equiv 0\label{para}\,.
	\end{align}
	Unweighted and weighted $L_p$-theories for \eqref{para} have been developed on various types of domains, such as $C^1$-domains \cite{KK2004}, Reifenberg domains \cite{Rparabolic}, convex domains and Lipschitz domains \cite{IW}, smooth cones \cite{Kozlov,Naza,Sol}.
	
	In addition to the classical heat equation, we also consider the following time-fractional heat equation of order $0<\alpha<1$ within a unified framework:
	\begin{alignat}{3}
		\partial_t^{\alpha}u=\Delta u+f&\quad\text{in}\,\,\,(0,\infty)\times \Omega\quad&&;\quad u(0,\cdot)=u_0\,\,,\,\,u|_{(0,\infty)\times \partial\Omega}\equiv 0\label{parafrac}\,,
	\end{alignat}
	where $\partial_t^{\alpha}$ is the Caputo fractional derivative (see Definition \ref{240907807}).
	In physical terms, equation \eqref{parafrac} describes subdiffusion.
	While normal diffusion, modeled by the classical heat equation, describes particle systems where the mean squared displacement is proportional to time $t$, subdiffusion characterizes systems in which the mean-squared displacement grows proportionally to $t^\alpha$.
	This model has numerous important applications in probability and mechanics; for further discussion, we refer the reader to \cite{metzler2014anomalous,metzler2000random,metzler2004restaurant}.
	
	The $L_p$-theory for \eqref{parafrac} has been developed via various approaches, including a functional analytic and $H^\infty$ calculus methods \cite{pruss2012evolutionary, zacher2005maximal}, a kernel approach \cite{HAN20203515, HKP2021, KIM2017123}, and a kernel-free approach \cite{DONG2019289,DongKim2021}.
	Regarding the domain $\Omega$ appearing in \eqref{parafrac}, these studies mainly consider the whole space $\bR^d$, as well as $C^2$- and $C^1$-domains.
	Although many other relevant references are available, we omit a detailed discussion to focus on our primary interest.	
	
	In \cite{Seo202404}, the present author developed a general weighted $L_p$-theory for the Poisson equation in non-smooth domains.
	As a natural continuation of that work, we extend this framework and establish a weighted $L_p$-theory for \eqref{para} and \eqref{parafrac} in non-smooth domains.
	We further develop the framework of \cite{Seo202404}, obtaining nontrivial results for time-measurable coefficients (see Sections \ref{convex} and \ref{ERD}) and establishing pointwise estimates for homogeneous solutions on Lipschitz cones (see Section \ref{0074}).
	
	We consider domains $\Omega\subsetneq\bR^d$ admitting the Hardy inequality: There exists a constant $\mathrm{C}_0(\Omega)>0$ such that
	\begin{align}\label{hardy}
		\int_{\Omega}\Big|\frac{f(x)}{d(x,\partial\Omega)}\Big|^2\dd x\leq \mathrm{C}_0(\Omega)\int_{\Omega}|\nabla f(x)|^2 \dd x\quad\text{for all}\quad f\in C_c^{\infty}(\Omega)\,.
	\end{align}
	A sufficient geometric condition ensuring \eqref{hardy} is the volume density condition:
	\begin{align}\label{230212413}
		\inf_{\substack{p\in\partial\Omega\\r>0}}\frac{\big|\Omega^c \cap B_r(p)\big|}{\big|B_r(p)\big|}>0
	\end{align}
	(see Proposition \ref{240928313} and \eqref{Acondition}).
	We establish that for equations \eqref{para} and \eqref{parafrac} with a domain $\Omega$ admitting \eqref{hardy}, each \textit{superharmonic Harnack function} $\psi$ 
	immediately leads to a weighted $L_p$-solvability result associated with $\psi$, for general $p\in(1,\infty)$ (see \eqref{2401301249}).
	Here, a superharmonic Harnack function is a locally integrable function $\psi$ that satisfies the following conditions:
	\begin{align*}
		\begin{gathered}
			\text{$\Delta \psi\leq 0$ in the sense of distributions, and}\\
			\text{$\sup_{B(x,\rho(x)/2)}\psi\lesssim \inf_{B(x,\rho(x)/2)}\psi$ for all $x\in\Omega$,}
		\end{gathered}
	\end{align*}
	where $\rho(x):=d(x,\partial\Omega)$.
	In our results, $\psi$ describes the boundary behavior of solutions.
	We apply our result to various types of non-smooth domains through appropriate superharmonic functions.
	We discuss our main results and their applications in detail in Section \ref{0003}.
	
	\subsection{Historical remarks for the classical heat equation.}\label{230214201}
	In studying the theory of the heat equation in non-smooth domains, it is instructive to recall previous works on the Poisson equation with zero Dirichlet boundary conditions.
	Notable contributions include those by Jerison and Kenig \cite{kenig} in Lipschitz domains, Adolfsson \cite{convexAdo} and Fromm \cite{convexFromm} in convex domains, studies in smooth cones and polyhedrons presented in the monographs \cite{BK2006, MNP, MR}, and Byun and Wang \cite{Relliptic} in Reifenberg domains.
	We refer the reader to \cite{Seo202404} for further discussion of these works, and next focus on the heat equation \eqref{para}.
	
	The extension of the results for the Poisson equation to the heat equation in Lipschitz and convex domains was established by Wood \cite[Theorem 6.1]{IW}, which provides a universal range of $p$ for the unweighted $L_p$-theory.
	However, the impossibility of establishing a general unweighted $L_p$-theory for all $1<p<\infty$ for the Poisson equation, as shown in \cite[Theorem A]{kenig} (see also \cite[Theorem 7.1]{IW}), also extends to the heat equation.
	Given these limitations in unweighted Sobolev spaces, it is natural to turn to theories in weighted Sobolev spaces.
	
	There are several works on the $L_p$-theory for the heat equation in smooth cones, including those by Solonnikov \cite{Sol}, Nazarov \cite{Naza}, and Kozlov and Nazarov \cite{Kozlov}.
	Here, a smooth cone is defined by $\Omega:=\{r\sigma\,:\,r>0\,,\,\, \sigma\in\cM\}$, where $\cM$ is a smooth subdomain of $\bS^{d-1}$.
	For such domains, scholars have investigated the unique solvability of the heat equation in specific types of weighted $L_p$-Sobolev spaces for general $p\in(1,\infty)$.
	The weight system in these spaces is based on the distance function from the vertex.
	The admissible range of weights ensuring unique solvability is closely related to the \textit{eigenvalues} of the spherical Laplacian on $\cM$.
	
	These studies suggest that developing a unified framework for the $L_p$-solvability of the heat equation in general non-smooth domains requires adopting a weight system associated with the Laplace operator and the geometric features of each domain.
	
	In the context of the weighted $L_p$-theory, we focus on the localization argument developed by Krylov \cite{Krylov1999-1}, which was originally devised to establish the $L_p$-theory of the stochastic heat equation in smooth domains (see \cite{Kim2004,Krylov1994}).
	Krylov's work focuses on the heat equation in half-spaces and was later extended to $C^1$-domains by Kim and Krylov \cite{KK2004}.
	Although these works were originally confined to half-spaces and $C^1$ domains, the localization argument was later extended to non-smooth domains, as demonstrated by \cite{Kim2014, ConicPDE}.
	In the study of smooth cones \cite{ConicPDE}, the argument was adapted to a broader weight system and combined with heat kernel estimates.
	Additionally, Kim \cite{Kim2014} established a connection between Krylov's approach \cite{Krylov1999-1} and the classical Hardy inequality \eqref{hardy}.

	However, the approaches in \cite{Kim2014, ConicPDE} have some limitations.
	The method in \cite{ConicPDE} relies on sharp kernel estimates, and the heat kernels have only been investigated for a few classes of domains. 
	In \cite{Kim2014}, the range of weights requires further refinement to adequately capture the boundary behavior of solutions.
	In particular, the results in \cite{Kim2014} do not include the cases of the half-space \cite{Krylov1999-1} and of $C^1$ domains \cite{Kim2004}.

	\subsection{Main result and its applications to various domain conditions}\label{0003}

	In line with \cite{Kim2014}, we restrict our attention to the class of domains admitting the Hardy inequality.
	This choice is motivated by the observation that the Hardy inequality holds on various non-smooth domains (see \eqref{230212413}).
	A key distinguishing feature of the present paper, compared with earlier studies, is the utilization of superharmonic functions in conjunction with the Hardy inequality. 
	This combination enables us to accurately describe the boundary behavior of solutions.
	Our main result, Theorem \ref{22.02.18.6}, establishes the following estimate together with a related solvability result:
	\begin{itemize}
		\item[] Let $\Omega$ admit the Hardy inequality \eqref{hardy} and $\psi$ be a superharmonic Harnack function in $\Omega$.
		For any $1<p<\infty$ and $-\frac{1}{p}<\mu<1-\frac{1}{p}$, if $u\in C_c^\infty((0,\infty)\times \Omega)$ and $f_0,\,f_1,\,\ldots,\,f_d$ satisfy $\partial_t u=\Delta u+f_0+\sum_{i\geq 1}D_if_i$, then we have
		\begin{align}\label{2401301249}
			\begin{split}
				&\|\psi^{-\mu}\rho^{-2/p}u\|_p+\|\psi^{-\mu}\rho^{-2/p+1}D_xu\|_p\\
				\lesssim\,&\|\psi^{-\mu}\rho^{-2/p+2}f_0\|_p+\sum_{i\geq 1}\|\psi^{-\mu}\rho^{-2/p+1}f_i\|_p\,.
			\end{split}
		\end{align}
	\end{itemize}
	Here, $\rho(x):=d(x,\partial\Omega)$. 
	In \eqref{2401301249}, the superharmonic Harnack function $\psi$ represents the boundary behavior of solutions.
	Through the Sobolev-H\"older embedding theorem, we further obtain pointwise estimates for solutions (see \eqref{241031713} and Proposition \ref{2204160313}).
	
	Our main result does not specify the superharmonic Harnack function $\psi$.
	The flexibility in selecting $\psi$ is a crucial advantage of our theorem, which allows applications to a wide range of non-smooth domains $\Omega$.
	Throughout Section \ref{app.}, we investigate applications to various geometric domain conditions.
	Specifically, we apply superharmonic functions $\psi$ such that $\psi\simeq d(\cdot,\partial\Omega)^\lambda$ for some $\lambda\in\bR$, where the range of $\lambda$ is different for each domain condition.
	Additionally, in Lipschitz cones, we apply the weighted $L_p$-theory to obtain pointwise estimates for solutions.

	We next outline applications of our main result to various domain conditions.
	Consider a domain $\Omega\subset \bR^d$, $d\geq 2$, and denote $\rho(x):=d(x,\partial\Omega)$.
	For $p\in (1,\infty)$, $\theta\in\bR$, and $n\in \{0,1,2,\ldots\}$, we define
	\begin{alignat*}{2}
		&\|f\|_{\bW_{p,\theta}^n(\Omega,T)}&&:=\sum_{k=0}^n\|D_x^kf\|_{\bL_{p,\theta+kp}(\Omega,T)}:=\sum_{k=0}^n\bigg(\int_0^T\int_\Omega\big|D_x^kf(t,x)\big|^p\rho(x)^{\theta+kp}\dd x\dd t\bigg)^{1/p}\,,\\
		&\|f\|_{\bW_{p,\theta}^{-n}(\Omega,T)}&&:=\inf\bigg\{\sum_{|\beta|\leq n}\|f_\beta\|_{\bL_{p,\theta-|\beta|p}(\Omega,T)}\,:\,f=\sum_{|\beta|\leq n}D_x^{\beta}f_{\beta}\bigg\}\,.
	\end{alignat*}
	$\bW_{p,\theta}^{n}(\Omega,T)$ denotes the set of all $f$ such that $\|f\|_{\bW_{p,\theta}^{n}(\Omega,T)}<\infty$.
	
	\begin{remark}\label{230214208}
		For $n\in\bZ$, the spaces $\bW_{p,\theta}^{n}(\Omega,T)$ appear only in this subsection.
		However, these spaces coincide with $\bH_{p,\theta+d}^{n}(\Omega,T)$ (see \eqref{2409098051} and \eqref{241111154} with putting $\Psi\equiv 1$).
		Here, $\bH_{p,\theta+d}^{n}(\Omega,T)$ denotes the function space introduced in \eqref{221015645}.
	\end{remark}
	
	Consider the equation
	\begin{align}
		\partial_t^\alpha u=\cL u+f:=\sum_{i,j=1}^da^{ij}(t)D_{ij}u+f\quad\text{on }(0,T]\times \Omega\,,\label{2410161019}
	\end{align}
	where $0<\alpha\leq 1$ and $0<T<\infty$.
	Here, $\partial_t^1$ coincides with the classical derivative, and for $0<\alpha<1$, $\partial_t^\alpha$ is the Caputo derivative (see Definitions \ref{240907807} and \ref{240419512}).
	In addition, the coefficients $a^{ij}(\cdot)$ are measurable functions of $t$ and satisfy the uniform ellipticity condition: there exists $0<\nu\le1$ such that
	\begin{align}\label{2410161039}
		\nu |\xi|^2\leq \sum_{i,j=1}^da^{ij}(t)\xi_i\xi_j\leq \nu^{-1}|\xi|^2\quad\text{for all}\,\,\,\,0<t\leq T\,,\,\,\xi\in\bR^d\,.
	\end{align}
	
	We restrict our attention to both the zero initial data condition and the zero Dirichlet boundary condition,
	\begin{align}
		u(t,x)=0 \quad \text{if}\,\,\,\, t=0\,\,\,\,\text{or}\,\,\,\,x\in\partial\Omega\,.\label{2410161020}
	\end{align}
	The results in Sections \ref{mainresultsection} and \ref{app.} also cover the nonzero initial data problem, where the initial data lie in the space $B_{p,\theta}^\gamma(\Omega)$, as introduced in \eqref{2409098052} and \eqref{241111154} (with $\Psi\equiv1$).
	For the rigorous definition of equation \eqref{2410161019} with \eqref{2410161020} or nonzero initial data, see Definition \ref{240419512}.
	For simplicity of exposition, we omit the treatment of this case in the present subsection.
	
			\vspace{1mm}\noindent
	\textbf{1) Exterior cone condition.}
	For $\delta\in[0,\pi/2)$ and $R>0$, $\Omega$ is said to satisfy the exterior $(\delta,R)$-cone condition if for every $p\in\partial\Omega$, there exists a unit vector $e_p\in\bR^d$ such that
	$$
	\{x\in\bR^d\,:\,(x-p)\cdot e_p\geq |x-p|\cos\delta\,\,,\,\,|x-p|<R\}\subset \Omega^c\,.
	$$
	When $\delta=0$, this condition is often called the exterior $R$-line segment condition.
	An example of this condition is illustrated in Figure \ref{230212745}.
	Although our result is applicable to some unbounded domains, we restrict our attention here to bounded domains.

	For given $\delta\in(0,\frac{\pi}{2})$, let $\lambda_{\delta}$ be the constant defined in \eqref{241018612}.
	When $d=2$, we set $\lambda_{0}=1/2$.
	Note that $\lambda_\delta=\frac{\pi}{2(\pi-\delta)}$ if $d=2$; $\lambda_\delta=\frac{\delta}{\pi-\delta}$ if $d=4$; and $\lambda_\delta>0$ for all $d\ge2$ and $\delta>0$.

	\begin{thm}[see Corollary \ref{22.02.19.3} and Example \ref{240928314}.(3)]\label{2302141109}
		Let $\delta\in(0,\pi)$ if $d\geq 3$, and $\delta\in[0,\pi)$ if $d= 2$.
		Suppose that $\Omega\subset \bR^d$ is a bounded domain satisfying the $(\delta,R)$-exterior cone condition for some $0<R<\infty$.
		Then for any $p\in(1,\infty)$, $\theta\in\bR$, and $n\in\bZ$ satisfying
		\begin{align*}
			-2-(p-1)\lambda_\delta<\theta-d<-2+\lambda_\delta\,,
		\end{align*}
		if $f\in \bW_{p,\theta+2p}^n(\Omega,T)$, then equation \eqref{2410161019}--\eqref{2410161020} with $\cL:=\Delta$ has a unique solution $u$ in $\bW_{p,\theta}^{n+2}(\Omega,T)$.
		Moreover, we have
		\begin{align}\label{241018627}
			\|u\|_{\bW_{p,\theta}^{n+2}(\Omega,T)}\leq N\| f\|_{\bW_{p,\theta+2p}^{n}(\Omega,T)}\,,
		\end{align}
		where $N=N(d,p,\alpha,n,\theta,\lambda,M_{\lambda})$.
	\end{thm}
	
	For a comparison between the function spaces in Theorem \ref{2302141109} and those in \cite{IW}, see \cite[Remark 1.4]{Seo202404}.
	
			\vspace{1mm}\noindent
	\textbf{2) Convex domains.} A convex domain $\Omega$ is an open set such that $(1-t)x+ty\in \Omega$ whenever $x,\,y\in\Omega$ and $t\in [0,1]$.
	
	\begin{thm}[see Corollary \ref{2208131026}]\label{221005646}
		Let $\Omega$ be a convex (possibly unbounded) domain.
		Then for any $p\in(1,\infty)$, $\theta\in\bR$, and $n\in\bZ$ satisfying
		\begin{align*}
			-p-1<\theta-d<-1\,,
		\end{align*}
		if $f\in \bW_{p,\theta+2p}^n(\Omega,T)$, then equation \eqref{2410161019}--\eqref{2410161020} with any ellipticity constant $0<\nu\leq1$ in \eqref{2410161039} has a unique solution $u$ in $\bW_{p,\theta}^{n+2}(\Omega,T)$.
		Moreover, we have
		\begin{align}\label{241018628}
			\|u\|_{\bW_{p,\theta}^{n+2}(\Omega,T)}\leq N\| f\|_{\bW_{p,\theta+2p}^{n}(\Omega,T)}\,,
		\end{align}
		where $N=N(d,p,\alpha,n,\theta)$. 
		In particular, $N$ is independent of $\Omega$.
	\end{thm}
	
			\vspace{1mm}\noindent
	\textbf{3) Totally vanishing exterior Reifenberg condition.}
	The totally vanishing exterior Reifenberg condition (abbreviated as $\langle\mathrm{TVER}\rangle$) extends the concept of bounded vanishing Reifenberg domains introduced in \eqref{22.02.26.41} and the discussion below.
	
	To clarify the main concept of $\langle\mathrm{TVER}\rangle$ in Definition \ref{2209151117}.(3), we present a simplified version in Definition \ref{221013228}, denoted by $\langle\mathrm{TVER}\rangle^*$.
	Note that $\langle\mathrm{TVER}\rangle^*$ is a sufficient condition for $\la \mathrm{TVER}\ra$.
	In Figure \ref{230212856}, we illustrate the differences between the vanishing Reifenberg condition, $\langle\mathrm{TVER}\rangle^*$, and $\langle\mathrm{TVER}\rangle$.
	
	\begin{defn}\label{221013228}
		We say that an open set  $\Omega$ satisfies $\langle\mathrm{TVER}\rangle^*$ if for each $\delta\in(0,1)$, there exist $R_{0,\delta},\,R_{\infty,\delta}>0$ satisfying the following: For every $p\in\partial \Omega$ and $r>0$ with $r\leq R_{0,\delta}$ or $r\geq R_{\infty,\delta}$, there exists a unit vector $e_{p,r}\in\bR^d$ such that
		\begin{align*}
			\Omega\cap B_r(p)\subset \{x\in B_r(p)\,:\,(x-p)\cdot e_{p,r}<\delta r\}\,.
		\end{align*}
	\end{defn}
	
	As shown in Example \ref{220910305}, $\langle\mathrm{TVER}\rangle^*$ is satisfied by bounded domains of the following types: the vanishing Reifenberg domains, $C^1$-domains, domains satisfying the exterior ball condition, and finite intersections of these domains.
	Furthermore, several unbounded domains also satisfy $\langle\mathrm{TVER}\rangle^*$ (see \eqref{241025520}).
	
	\begin{thm}[see Corollary \ref{22.07.17.109}]\label{230214437}
		Suppose that $\Omega$ satisfies $\langle\mathrm{TVER}\rangle^*$.
		Then for any $p\in(1,\infty)$, $\theta\in\bR$, and $n\in\bZ$ satisfying
		\begin{align*}
			-p-1<\theta-d<-1\,,
		\end{align*}
		the same assertion as that in Theorem \ref{221005646} holds.
		Here, $N$ in \eqref{241018628} depends only on $d$, $p$, $n$, $\theta$, and  $\big\{R_{0,\delta}/R_{\infty,\delta}\big\}_{\delta\in(0,1]}$.
	\end{thm}

	Parabolic equations in bounded vanishing Reifenberg domains have been investigated in the literature, such as Byun and Wang \cite{Rparabolic} and Choi and Kim \cite{Reifweight2}.
	These studies mainly focused on equations with variable coefficients, and also provided weighted $L_p$-estimates for Muckenhoupt $A_p$-weight functions.
	However, these studies mostly dealt with bounded vanishing Reifenberg domains.
	In contrast with these works, Theorem \ref{230214437} considers domains satisfying $\langle\mathrm{TVER}\rangle^\ast$, thereby including bounded vanishing Reifenberg domains.
	
			\vspace{1mm}\noindent
	\textbf{4) Lipschitz cones.} 
	Lipschitz cones are defined by $\Omega:=\{r\sigma\,:\,r>0\,\,,\,\,\,\,\sigma\in\cM\}$ with $\cM$ a Lipschitz subdomain of $\bS^{d-1}$.
	In Section \ref{0074}, we derive a pointwise estimate for homogeneous solutions of the equation
	\begin{align*}
		u_t=\Delta u\quad \text{in}\quad (0,1]\times \big(B_1\cap \Omega\big)\quad;\quad u=0\quad \text{on}\quad (0,1]\times \big(B_1\cap \partial\Omega\big)\,,
	\end{align*}
	where $B_1:=B_1(0)$ (see Theorem \ref{2208221223} and equation \eqref{2208221229}).
	
			\vspace{1mm}\noindent
	\textbf{5) Domains with the capacity density condition.}
	We consider domains $\Omega$ satisfying
	\begin{align}\label{2302101253}
		\inf_{\substack{p\in\partial\Omega\\r>0}}\frac{\mathrm{Cap}\left(\Omega^c\cap \overline{B}_r(p),B_{2r}(p)\right)}{\mathrm{Cap}\left(\overline{B}_r(p),B_{2r}(p)\right)}>0\,,
	\end{align}
	where $\mathrm{Cap}(K,U)$ denotes the $L_2$-capacity of $K$ relative to $U$ (for the definition, see \eqref{230324942}).
	Condition \eqref{2302101253} has been studied extensively in the literature, including \cite{aikawa2002, AA, KilKos1994, kinnunen2021, lewis}.
	In particular, the volume density condition \eqref{230212413} is a sufficient condition for \eqref{2302101253}.
	As an application, we consider the equation
	$$
	\partial_t^\alpha u=\Delta u+f_0+\sum_{i=1}^d D_if^i\quad\text{in}\quad (0,T]\times \Omega\quad ;\quad u(0,x)=0\quad \text{for}\quad x\in\Omega
	$$
	for a domain $\Omega$ satisfying \eqref{2302101253}.
	We establish an unweighted $L_p$ solvability result for $p$ near $2$ (see Theorem \ref{230210357}), and prove H\"older continuity of the solution under the assumption that
	$|f_0|\lesssim \rho^{-2+\epsilon}$ and $|f_1|+\cdots+|f_d|\lesssim\rho^{-1+\epsilon}$ for some $\varepsilon>0$ (see Theorem \ref{220602322}).
	
		\vspace{1mm}\noindent
	\textbf{6) Domains $\Omega\subset \bR^d$ with $\dim_{\cA}(\Omega^c)<d-2$.}
	Recall the definition of the Aikawa dimension, $\dim_{\cA}(E)$, given in \eqref{241018630}.
	
	\begin{thm}[see Corollary \ref{22.02.19.300}]
		Let $\Omega\subset \bR^d$, $d\geq 3$, satisfy $\dim_{\cA}(\Omega^c)=:\beta_0<d-2$.
		Then for any $p\in(1,\infty)$, $\theta\in\bR$, and $n\in\bZ$ satisfying 
		$$
		-d+\beta_0<\theta<(p-1)(d-\beta_0)-2p\,,
		$$
		the same assertion as that in Theorem \ref{2302141109} holds.
		Here, $N$ in \eqref{241018627} depends only on $d$, $p$, $n$, $\theta$, $\beta_0$, $\{A_{\beta}\}_{\beta>\beta_0}$ in \eqref{22.02.08.2}.
	\end{thm}

	\subsection{Notations.}
	\begin{itemize}
		
		\item We use all notations introduced in \cite[Subsection 1.3]{Seo202404}.

		\item  The letter $N$ denotes a finite positive constant which may have different values along the argument, whose dependence will be specified in context;  $N=N(a,b,\cdots)$ means that this $N$ depends only on the parameters inside the parentheses.
		
		\item  For a list of parameters $L$, $A\lesssim_{L} B$ means that $A\leq N(L)B$, and $A\simeq _{L}B$ means that $A\lesssim_{L} B$ and $B\lesssim_{L} A$.
		
		\item $a \vee b :=\max\{a,b\}$, $a \wedge b :=\min\{a,b\}$. 
		
		\item For a Lebesgue measurable set $E\subset \bR^d$, $|E|$ denotes the Lebesgue measure of $E$.
		
		
		\item 	For $x=(x_1,\ldots,x_d)$, $y=(y_1,\ldots,y_d)$ in $\bR^d$,  $x\cdot y:=(x,y)_{\mathbb{R}^d} :=\sum^d_{i=1}x_iy_i$ denotes the standard inner product.
		
		\item $B_r(p):=B(p,r):=\big\{x\in\bR^d\,:\,|x-p|<r\big\}$, and $B_r:=B_r(0)$.
		
		\item For a fixed open set $\cO\subset \bR^d$, we usually denote $\rho(x):=d(x,\partial\cO)$ when there is no confusion.
		
		\item For a set $E\subset \bR^d$, $1_E$ denotes the function defined by $1_E(x)=1$ for $x\in E$, and $1_E(x)=0$ for $x\notin E$.
		For a function $f$ defined in $E$, $f1_E$ denotes the function defined as $\big(f1_E\big)(x)=f(x)$ if $x\in E$, and $\big(f1_E\big)(x)=0$ if $x\notin E$.
		
		\item For an open set $\cO\subseteq \bR^d$, $\cD'(\cO)$ denotes the set of all distributions on $\cO$, which is the dual of $C_c^{\infty}(\cO)$.
		For $f\in \cD'(\cO)$ and $\varphi\in C^{\infty}_c(\domain)$, the expression $\la f,\varphi\ra $ denotes the evaluation of $f$ with the test function $\varphi$.

		\item  For  any multi-index $\alpha=(\alpha_1,\ldots,\alpha_d)$, $\alpha_i\in \{0\}\cup \bN$,   
		$$
		\partial_tf:=\frac{\mathrm{d}f}{\mathrm{d} t}\,\,,\,\,\,\,\partial_t^nf:=\frac{\mathrm{d}^n f}{\mathrm{d}t^n}\,\,,\,\,\,\, D_if:=\frac{\mathrm{d} f}{\mathrm{d} x_i}\,\,,\,\,\,\, D^{\alpha}_xf(x):=D^{\alpha_1}_1\cdots D^{\alpha_d}_df(x).
		$$
		We denote $|\alpha|:=\sum_{i=1}^d \alpha_i$. 
		For the second order derivatives we denote $D_iD_jf$ by $D_{ij}f$. We often abbreviate 
		$|gf_x|:=\sum_{i=1}^d|g D_if|$, $|gf_{xx}|:=\sum_{i,j=1}^d|gD_{ij}f|$, and  $\big|gD^m_x f\big|:=\sum_{|\alpha|=m}|gD^\alpha_x f|$.
		We extend these notations to a sublinear function $\|\cdot\|:\cD'(\Omega)\rightarrow [0,+\infty]$; for example, $\|gf_x\|:=\sum_{i=1}^d\|g D_if\|$.

		\item $\Delta f:=\sum_{i=1}^d D_{ii}f$ denotes the Laplacian of a function (or distribution) $f$ defined on an open subset of $\mathbb{R}^d$.

		\item  For a measure space $(A, \cA, \mu)$, a Banach space $(B,\|\cdot\|_B)$, and $p\in[1,\infty]$, we write $L_p(A,\cA, \mu;B)$ for the set of all $B$-valued $\overline{\cA}$-measurable functions $f$ satisfying
		\begin{alignat*}{2}
			&\|f\|^p_{L_p(A,\cA,\mu;B)}:=\int_{A} \lVert f\rVert^p_{B} \dd \mu<\infty&&\qquad\text{if}\quad p\in[1,\infty)\,;\\
			&\|f\|_{L_\infty(A,\cA,\mu;B)}:=\underset{x\in A}{\mathrm{ess\,sup}}\,\|f(x)\|_B<\infty&&\qquad\text{if}\quad  p=\infty\,.
		\end{alignat*}
		Here, $\overline{\cA}$ is the completion of $\cA$ with respect to $\mu$.  
		We will drop $\cA$ or $\mu$ or even $B$ in $L_p(A,\cA, \mu;B)$ when they are obvious in the context.

		

		
		
	\end{itemize}

	%


	\mysection{Preliminary}
	
	Throughout this section, $\Omega$ is an open set in $\bR^d$ (with $d\in\bN$), and $\rho(x)$ denotes the boundary distance function $d(x,\partial\Omega)$.
	
	\subsection{Fractional calculus}
	
	We recall the definitions of the Riemann–Liouville fractional integral and the Caputo fractional derivative.
	
	\begin{defn}\label{240907807}\,
		
		\begin{enumerate}
			\item Let $\alpha>0$.
			For a measurable function $f:[0,T]\rightarrow \bR$, $I_t^{\alpha}f$ denotes the Riemann–Liouville fractional integral defined by 
			$$
			I_t^{\alpha}f(t_0):=\int_0^{t_0}\frac{(t_0-s)^{-1+\alpha}}{\Gamma(\alpha)}f(s)\dd s
			$$
			provided that the right-hand side is well-defined, \textit{i.e.}, $(t_0-\,\cdot\,)^{-1+\alpha}f(\,\cdot\,)\in L_1\big((0,t_0]\big)$.	
			We also define $I_t^0 f(t)=f(t)$.
			
			\item Let $0<\alpha<1$.
			For $f\in C^1\big([0,T]\big)$, $\partial^{\alpha}_tf$ denotes the Caputo fractional derivative defined by $\partial^\alpha_tf:=\partial_tI_t^{1-\alpha}\big(f-f(0)\big)=I_t^{1-\alpha}\big(\partial_t f\big)$.
		\end{enumerate}
	\end{defn}
	
	Using the following relation between the beta function and the gamma function: \begin{align}\label{240420130}
		\int_s^t(t-r)^{-1+\beta_1}(r-s)^{-1+\beta_2}\dd r=\frac{\Gamma(\beta_1)\Gamma(\beta_2)}{\Gamma(\beta_1+\beta_2)}(t-s)^{-1+\beta_1+\beta_2}\,,
	\end{align}
	for all $\beta_1,\,\beta_2>0$, we obtain that for any $\beta_1,\,\beta_2\geq 0$, $I_t^{\beta_1}\big(I_t^{\beta_2} f\big)=I_t^{\beta_1+\beta_2}f$.

	\begin{lemma}\label{240426733}
		Let $F\in C_{\mathrm{loc}}^1\big([0,T)\big)$, $p\in(1,\infty)$, and $\alpha\in(0,1)$. Then
		\begin{align}\label{240911430}
			\partial_t^\alpha\big(|F|^p\big)(t)\leq p|F(t)|^{p-2}F(t)\partial_t^\alpha F(t)\,.
		\end{align}
		Moreover, if $F(0)=0$, then
		\begin{align}\label{241027730}
			0\leq I^{1-\alpha}_t\big(|F|^p\big)(t)\leq p\int_0^t|F(s)|^{p-2}F(s)\partial_t^{\alpha}F(s)\dd s\,.
		\end{align}
	\end{lemma}
	When $p=2$, this lemma is implicit in the proof of \cite[Proposition 4.1]{DONG2019289}.
	We extend the argument to general $1<p<\infty$.
	
	\begin{proof}[Proof of Lemma \ref{240426733}]
		By the definition of $\partial_t^\alpha:=I_t^{1-\alpha}\partial_t$, we obtain that 
		\begin{align*}
			J(t):=\,&p|F(t)|^{p-2}F(t)\partial_t^\alpha F(t)-\partial_t^\alpha \big(|F|^p\big)(t)\\
			=\,&p\int_0^t\frac{(t-s)^{-\alpha}}{\Gamma(1-\alpha)}\Big(|F(t)|^{p-2}F(t)F'(s)-|F(s)|^{p-2}F(s)F'(s)\Big)\dd s.
		\end{align*}
		For a fixed $t\in (0,T]$, put
		$$
		I(s):=|F(s)|^p-|F(t)|^p-p|F(t)|^{p-2}F(t)\big(F(s)-F(t)\big)\,.
		$$
		Observe that 
		\begin{align}\label{240825340}
			0\leq I(s)\leq (t-s)\cdot 2p\cdot\|F\|_\infty^{p-1}\|F'\|_\infty\qquad \text{and}\qquad \\
			I'(s)=p\Big(|F(s)|^{p-2}F(s)F'(s)-|F(t)|^{p-2}F(t)F'(s)\Big)\,,\nonumber
		\end{align}
		where the first inequality in \eqref{240825340} follows from the convexity of the map $t\mapsto |t|^p$.
		These imply that
		\begin{align*}
			J(t)=\,&-\int_0^t\frac{(t-s)^{-\alpha}}{\Gamma(1-\alpha)}I'(s)\dd s\\
			=\,&-\lim_{\epsilon\searrow 0}\bigg(\frac{\epsilon^{-\alpha}}{\Gamma(1-\alpha)}I(t-\epsilon)-\frac{t^{-\alpha}}{\Gamma(1-\alpha)}I(0)-\alpha\int_0^{t-\epsilon}\frac{(t-s)^{-1-\alpha}}{\Gamma(1-\alpha)}I(s)\dd s\bigg)\\
			=\,&\frac{t^{-\alpha}}{\Gamma(1-\alpha)}I(0)+\alpha\int_0^{t}\frac{(t-s)^{-1-\alpha}}{\Gamma(1-\alpha)}I(s)\dd s\,.
		\end{align*}
		Since $I(s)\geq 0$, we have $J(t)\geq 0$, and therefore \eqref{240911430} follows.
		
		Since 
		$$
		I_t^1\partial_t^\alpha \big(|F|^p\big)=I_t^{1-\alpha}I_t^1\partial_t \big(|F|^p\big)=I_t^{1-\alpha}\big(|F(\cdot)|^p-|F(0)|^p\big)\,,
		$$
		\eqref{241027730} follows from \eqref{240911430}. 
	\end{proof}

	\subsection{Superharmonic function}\label{241014932}
	
	The weight system in the main theorem (Theorem \ref{22.02.18.6}) consists of the boundary distance function and a superharmonic function.
	We recall the definitions of a superharmonic function and a classical superharmonic function.
	
	A locally integrable function $\phi:\Omega\rightarrow \bR$ is called a \textit{superharmonic} if $\Delta \phi\leq 0$ in the sense of distributions, \textit{i.e.}, for any nonnegative $\zeta\in C_c^{\infty}(\Omega)$, 
	$$
	\int_{\Omega}\phi\,\Delta \zeta\, \dd x\leq 0\,.
	$$
	A function $\phi:\Omega\rightarrow (-\infty,+\infty]$ is called a \textit{classical superharmonic function} if the following conditions are satisfied:
	\begin{enumerate}
		\item $\phi$ is lower semi-continuous in $\Omega$.
		\item For any $x\in \Omega$ and  any $r>0$ such that $\overline{B}_r(x)\subset \Omega$, 
		\begin{align*}
			\phi(x)\geq \frac{1}{\big|B_r(x)\big|}\int_{B_r(x)}\phi(y)\dd y\,.
		\end{align*}
		\item $\phi\not\equiv +\infty$ on each connected component of $\Omega$.
	\end{enumerate}
	As shown in \cite{AG}, $\phi$ is harmonic if and only if both $\phi$ and $-\phi$ are classical superharmonic functions.

	We collect several lemmas for superharmonic functions and classical superharmonic functions.
	The proofs of Lemmas \ref{240315329} - \ref{21.04.23.5} are given in \cite[Section 2]{Seo202404}.
	
	\begin{lemma}\label{240315329}\,
		A function $\phi:\Omega\rightarrow [-\infty,+\infty]$ is superharmonic if and only if there exists a classical superharmonic function $\phi_0$ in $\Omega$ such that $\phi=\phi_0$ almost everywhere in $\Omega$.
	\end{lemma}

	\begin{lemma}\label{21.04.23.3}
		Let $\phi$ be a classical superharmonic function in $\Omega$.
		
		\begin{enumerate}
			\item If $\phi$ is twice continuously differentiable, then $\Delta\phi\leq 0$.
			
			\item $\phi$ is locally integrable in $\Omega$.
			
			\item For any compact set $K\subset \Omega$, $\phi$ has the minimum value on $K$. 
			
			\item For $\epsilon>0$, put
			\begin{align}\label{21.04.23.1}
				\phi^{(\epsilon)}(x)=\int_{B_1(0)}\big(\phi 1_{\Omega}\big)(x-\epsilon y)\cdot N_0\,\ee^{-1/(1-|y|^2)}\dd y\,,
			\end{align}
			where $N_0:=\big(\int_{B_1}\ee^{-1/(1-|y|^2)}\dd y\big)^{-1}$.
			Then for any compact set $K\subset \Omega$ and $0<\epsilon<d(K,\Omega^c)$, the following statements hold:
			\begin{enumerate}
				\item $\phi^{(\epsilon)}$ is infinitely differentiable on $\bR^d$.
				\item $\phi^{(\epsilon)}$ is a classical superharmonic function in $K^{\circ}$.
				\item For any $x\in K$, $\phi^{(\epsilon)}(x)\nearrow \phi(x)$ as $\epsilon\searrow 0$.
			\end{enumerate}
		\end{enumerate}
	\end{lemma}

	\begin{lemma}\label{21.04.23.5}
		Let $\phi$ be a positive superharmonic function in $\Omega$, and denote by $\phi^{(\epsilon)}$ the function defined in \eqref{21.04.23.1}.
		\begin{enumerate}
			\item For any $c\leq 1$, $\phi^c$ is locally integrable in $\Omega$.
			
			\item If $f\in L_1(\Omega)$ and  $\mathrm{supp}(f)$ is a compact subset of $\Omega$, then for any $c\in\bR$,
			\begin{align*}
				\lim_{\epsilon\rightarrow 0}\int_{\Omega}|f|\big(\phi^{(\epsilon)}\big)^c\dd x= \int_{\Omega}|f|\phi^c\dd x\,.
			\end{align*}
			
			\item If $f\in L_\infty(\Omega)$ and $\mathrm{supp}(f)$ is a compact subset of $\Omega$, then for any $c\leq 1$,
			\begin{align*}
				\lim_{\epsilon\rightarrow 0}\int_{\Omega}f\big(\phi^{(\epsilon)}\big)^c\dd x= \int_{\Omega}f\phi^c\dd x\,.
			\end{align*}
		\end{enumerate}
	\end{lemma}

	\subsection{Harnack function and regular Harnack function}\label{241014933}
	One of our primary methods is the localization argument used in \cite{Krylov1999-1}.
	We employ the notions of the Harnack function and the regular Harnack function to extend this argument to various weight systems.
	
	\begin{defn}\label{21.10.14.1}\,
		
		\begin{enumerate}
			\item A measurable function $\psi:\Omega\rightarrow \bR_+$ is called a \textit{Harnack function} if there exists a constant $C=:\mathrm{C}_1(\psi)>0$ such that
			\begin{align}\label{22.02.17.1}
				\underset{B(x,\rho(x)/2)}{\mathrm{ess\,sup}}\,\psi\leq C\underset{B(x,\rho(x)/2)}{\mathrm{ess\,inf}}\,\psi\quad\text{for all}\,\,x\in\Omega\,.
			\end{align}
			
			\item A function $\Psi\in C^{\infty}(\Omega)$ is called a \textit{regular Harnack function} if $\Psi>0$ and there exists a sequence of constants $\big(C^{(k)}\big)_{k\in\bN}=:\mathrm{C}_2(\Psi)$ such that for each $k\in\bN$,
			\begin{align*}
				|D_x^k\Psi|\leq C^{(k)}\,\rho^{-k}\Psi\quad\text{on}\quad\Omega\,.
			\end{align*}
			
			\item Let $\psi$ be a measurable function and $\Psi$ be a regular Harnack function in $\Omega$. 
			$\Psi$ is called a \textit{regularization} of $\psi$ if there exists a constant $C=:\mathrm{C}_3(\psi,\Psi)>0$ such that
			$$
			C^{-1}\Psi\leq\psi\leq C\,\Psi\quad\text{almost everywhere on}\,\,\Omega. 
			$$
		\end{enumerate}
	\end{defn}
	
	It is worth mentioning that, as shown in \cite[Lemma 3.4]{Seo202404}, \eqref{22.02.17.1} is equivalent to the existence of constants $r\in(0,1)$ and $N_r>0$ such that 
	\begin{align}\label{241012350}
		\underset{B(x,r\rho(x))}{\mathrm{ess\,sup}}\,\psi\leq N_{r}\underset{B(x,r\rho(x))}{\mathrm{ess\,inf}}\,\psi\quad\text{for all}\,\,x\in\Omega\,.
	\end{align}
	
	The following example and lemma are given in \cite[Subsection 3.1]{Seo202404}:
	
	\begin{example}\label{21.05.18.2}\,
		
		\begin{enumerate}
			\item For any $E\subset \Omega^c$, the function $x\mapsto d(x,E)$ is a Harnack function in $\Omega$.
			In addition, $\mathrm{C}_1\big(d(\,\cdot\,,E)\big)$ can be chosen as $3$.
			
			\item If $\Psi\in C^{\infty}(\Omega)$ satisfies $\Psi>0$ and $\Delta\Psi=-\Lambda\Psi$ for some constant $\Lambda\geq 0$, then $\Psi$ is a regular Harnack function in $\Omega$.
			In addition, $\mathrm{C}_2(\Psi)$ can be chosen to depend only on $d$.
			
			\item For any regular Harnack function $\Psi$ in $\Omega$, and $\sigma\in\bR$, $\Psi^\sigma$ is also a regular Harnack function in $\Omega$.
			In addition, $\mathrm{C}_2(\Psi^\sigma)$ can be chosen to depend only on $d$,  $\sigma$, $\mathrm{C}_2(\Psi)$.
			
			\item Let $\Psi$ and $\Phi$ be regularizations of $\psi$ and $\phi$, respectively.
			Then $\Psi\Phi$, $\Psi+\Phi$, and $\frac{\Phi\Psi}{\Phi+\Psi}$ are regularizations of $\psi\phi$, $\max (\psi,\phi)$, and $\min (\psi,\phi)$, respectively.
		\end{enumerate}
	\end{example}

	\begin{lemma}\label{21.05.27.3}
		\,\,
		
		\begin{enumerate}
			\item If $\psi$ is a Harnack function, $\psi$ admits a regularization.
			For this regularization of $\psi$ denoted by $\widetilde{\psi}$, $\mathrm{C}_2(\widetilde{\psi})$ and $\mathrm{C}_3(\psi,\widetilde{\psi})$ can be chosen to depend only on $d$ and $\mathrm{C}_1(\psi)$.
			
			\item If $\Psi$ is a regular Harnack function, then it is also a Harnack function and $\mathrm{C}_1(\Psi)$ can be chosen to depend only on $d$ and $\mathrm{C}_2(\Psi)$.
		\end{enumerate}
	\end{lemma}

	We conclude this subsection with a lemma describing the boundary behavior of regular Harnack functions in a John domain.
	Recall that a domain $\Omega$ is called a John domain it satisfies the following condition:
	\begin{itemize}
		\item[] There exist a point $x_0\in\Omega$ and constants $L_0,\,\epsilon_0>0$ such that for any $x\in \Omega$, there exists a rectifiable curve $\gamma:[0,L]\rightarrow \Omega$, parametrized by arclength, such that 
		$L\leq L_0$, $\gamma(0)=x$, $\gamma(L)=x_0$, and $	d(\gamma(t),\partial\Omega)\geq \epsilon_0 t/L$ for all $t\in[0,L]$.
	\end{itemize}

	\begin{lemma}\label{220819318}
		Let $\Omega$ be a John domain.
		For each fixed $x_0$, there exist constants $N,\,A>0$ such that for every regular Harnack function $\Psi$ in $\Omega$,
		$$
		N^{-1}\rho(x)^{A}\leq \frac{\Psi(x)}{\Psi(x_0)}\leq N\rho(x)^{-A}\quad \text{for every}\quad x\in\Omega\,.
		$$
		
	\end{lemma}
	\begin{proof}
		Let $\Psi$ be a regular Harnack function in $\Omega$.
		Then $\Psi$ is also a continuous Harnack function (see Lemma \ref{21.05.27.3}.(2)).
		It follows from \cite[Corollary 3.4]{VM} that for an arbitrary domain $\Omega$,
		\begin{align}\label{220818200}
			N_0^{-(k(x,x_0)+1)}\leq \frac{\Psi(x)}{\Psi(x_0)}\leq N_0^{k(x,x_0)+1}\qquad\text{for all}\quad x_0,\,x\in \Omega\,,
		\end{align}
		where $N_0\geq 1$ is a constant depending only on $\mathrm{C}_1(\Psi)$.
		Here, $k(x,x_0)\geq 0$ denotes the quasihyperbolic distance between $x$ and $x_0$ (see \cite[paragraph 2.5]{VM} for the definition). 
		In addition, Gehring and Martio \cite[Theorem 3.11]{GO} proved that if $\Omega$ is a John domain, then for every $x_0\in \Omega$ there exists $A,\,B>0$ depending only on $\Omega$ and $x_0$ such that
		\begin{align}\label{220818159}
			k(x,x_0)\leq A \log\frac{1}{\rho(x)}+B\qquad\text{whenever}\quad x\in \Omega\,.
		\end{align}
		Combining \eqref{220818200} and \eqref{220818159}, we complete the proof. 
	\end{proof}
	
	\subsection{Weighted function spaces}
	In this subsection, we introduce the Krylov-type weighted Sobolev space $H_{p,\theta}^\gamma(\Omega)$ and its generalization $\Psi H_{p,\theta}^\gamma(\Omega)$.
	
	The theory of $H_{p,\theta}^{\gamma}(\Omega)$ was initiated by Krylov \cite{Krylov1999-0,Krylov1999-1,Krylov2001} and further developed by Lototsky \cite{Lo0,Lo1}.
	These developments were motivated by stochastic parabolic equations and aimed at controlling the boundary behavior of derivatives of solutions, as discussed in detail in \cite{Krylov1994, Krylov1999-1}.
	In \cite{Seo202404}, the present author proposed a generalization, denoted by $\Psi H_{p,\theta}^\gamma(\Omega)$, by introducing Harnack functions as weights.
	
	We suppose that $d\in\bN$, $p\in(1,\infty)$, $\gamma,\,\theta\in\bR$, $\Omega$ is an open set in $\bR^d$, and $\Psi$ is a regular Harnack function in $\Omega$.
	For brevity, we write
	$$
	\mathbf{I}=\{d,\,p,\,\gamma,\,\theta,\,\mathrm{C}_2(\Psi)\}\,,
	$$
	where $\mathrm{C}_2(\Psi)$ denotes the constant appearing in Definition \ref{21.10.14.1}.
	
	We first recall the definition of Bessel potential spaces and Besov spaces on $\bR^d$.
	$H_p^{\gamma}(\bR^d)$ denotes the space of Bessel potentials with the norm
	\begin{align*}
		\|f\|_{H_p^{\gamma}(\bR^d)}:=\|(1-\Delta)^{\gamma/2}f\|_{L_p(\bR^d)}:=\big\|\cF^{-1}\big[(1+|\xi|^2)^{\gamma/2}\cF(f)(\xi)\big]\big\|_p\,,
	\end{align*}
	where $\cF$ and $\cF^{-1}$ denote the Fourier transform and its inverse, respectively.
	For $\gamma\in \bR$, by $B_p^{\gamma}(\bR^d)$ we denote the Besov space whose norm is given by
	\begin{align*}
		\|f\|_{B_p^{\gamma}(\bR^d)}:=\big\|(1-\Delta)^{(\gamma-1)/2}f\big\|_{B_p^1(\bR^d)}\,,
	\end{align*}
	where
	\begin{align*}
		\|f\|_{B_p^1(\bR^d)}^p:=\|f\|_p^p+\int_{\bR^d}\int_{\bR^d}\frac{|f(x+h)-2f(x)+f(x-h)|^p}{|h|^{d+p}}\dd h\,\dd x\,.
	\end{align*}
	It is known that $H_p^{\gamma}(\bR^d)$ coincides with the Triebel-Lizorkin space $F_{p,2}^\gamma (\bR^d)$, and $B_p^{\gamma}(\bR^d)$ coincides with $B_{p,p}^{\gamma}(\bR^d)=F_{p,p}^\gamma (\bR^d)$ (see \cite[Subsection 2.3.5, Theorem 2.3.8]{triebel}).
	For $n\in\bN_0$ and $s\in(0,1)$, the Bessel potential space $H_p^{n}(\bR^d)$ is equivalent to the Sobolev space
	\begin{align*}
		W_p^{n}(\bR^d):=\bigg\{f\in\cD'(\bR^d)\,:\,\sum_{k=0}^n\int_{\bR^d}|D_x^kf|^p\dd x<\infty\bigg\}\,,
	\end{align*}
	and $B_{p}^{n+s}(\bR^d)$ coincides with the Sobolev-Slobodeckij space,
	\begin{align*}
		W_p^{n+s}(\bR^d):=\,&\Big\{f\in W_p^{n}\,:\,\int_{\bR^d}\int_{\{y:|y-x|< 1\}}\frac{|D_x^nf(x)-D_x^nf(y)|^p}{|x-y|^{d+sp}}\dd y\,\dd x<\infty\Big\}\,.
	\end{align*}
	
	Throughout this subsection, $X=H$ or $B$, so that $\big(X_p^\gamma(\bR^d), X_{p,\theta}^{\gamma}(\Omega), \Psi X_{p,\theta}^\gamma(\Omega)\big)$ stands for either $$
	\big(H_p^\gamma(\bR^d), H_{p,\theta}^{\gamma}(\Omega), \Psi H_{p,\theta}^\gamma(\Omega)\big)\quad \text{or}\quad \big(B_p^\gamma(\bR^d), B_{p,\theta}^{\gamma}(\Omega), \Psi B_{p,\theta}^\gamma(\Omega)\big)\,.
	$$
	To define the spaces $X_{p,\theta}^\gamma(\Omega)$ and $\Psi X_{p,\theta}^{\gamma}(\Omega)$,
	take $\trho\in C^{\infty}(\Omega)$ such that, for each $k\in\bN_0$, there exists $N_k=N(d,k)>0$ such that
	\begin{align*}
		N_0^{-1}\rho(x)\leq \trho(x)\leq N_0\rho(x)\quad\text{and}\quad |D_x^k\trho(x)|\leq N_k\,\trho(x)^{-k+1}\quad \text{for all}\quad x\in\Omega\,.
	\end{align*}
	(see, \textit{e.g.},  Lemma \ref{21.05.27.3}.(1)).
	Take $\zeta_0\in C_c^{\infty}(\bR_+)$ such that 
	\begin{align*}
		\zeta_0\geq 0\,\,,\,\,\,\
		\text{supp}(\zeta_0)\subset [\ee^{-1},\ee]\quad\text{and}\quad \sum_{n\in\bZ}\zeta_0(\ee^n\,\cdot\,)\equiv 1\quad\text{on}\,\,\bR_+\,.
	\end{align*}
	For arbitrary $\xi\in C^{\infty}(\bR_+)$, we denote
	\begin{align*}
		\xi_{(n)}(x):=\xi(\ee^{-n}\trho(x))\,.
	\end{align*}
	Note that the following properties hold for $\zeta_{0,(n)}:=\zeta_0\big(\ee^{-n}\trho(\,\cdot\,)\big)$:
	\begin{align*}
		\begin{split}
			&\sum_{n\in\bZ}\zeta_{0,(n)}\equiv 1\quad\text{on}\,\,\Omega\,,\\
			&\text{supp}(\zeta_{0,(n)})\subset \{x\in\Omega\,:\,\ee^{n-1}\leq \trho(x)\leq \ee^{n+1}\}\,,\\
			&\zeta_{0,(n)}\in C^{\infty}(\bR^d)\quad\text{and}\quad |D_x^{\alpha}\zeta_{0,(n)}|\leq N(d,\alpha,\zeta)\,\ee^{-n|\alpha|}\,.
		\end{split}
	\end{align*}
	
	\begin{defn}\label{220610533}\,\,
		
		\begin{enumerate}
			\item By $X_{p,\theta}^{\gamma}(\Omega)$ we denote the class of all distributions $f\in\cD'(\Omega)$ such that
			\begin{align*}
				\|f\|^p_{X^{\gamma}_{p,\theta}(\Omega)}:=\sum_{n\in\bZ}\ee^{n\theta}\|\big(\zeta_{0,(n)}f\big)(\ee^n\cdot)\|_{X^{\gamma}_p(\bR^d)}^p<\infty\,.
			\end{align*}
			
			\item By $\Psi X_{p,\theta}^{\gamma}(\Omega)$ we denote the class of all distributions $f\in\cD'(\Omega)$ such that $f=\Psi g$ for some $g\in X_{p,\theta}^{\gamma}(\Omega)$. The norm in $\Psi X_{p,\theta}^{\gamma}(\Omega)$ is defined by
			$$
			\|f\|_{\Psi X_{p,\theta}^{\gamma}(\Omega)}:=\|\Psi^{-1}f\|_{X_{p,\theta}^{\gamma}(\Omega)}\,.
			$$
		\end{enumerate}
	\end{defn}
	We also denote $L_{p,\theta}(\Omega):=H_{p,\theta}^0(\Omega)$ and $\Psi L_{p,\theta}(\Omega):=\Psi H_{p,\theta}^0(\Omega)$.

	The following equivalent norm for $\Psi H_{p,\theta}^\gamma(\Omega)$ was established in \cite[Lemma 3.12]{Seo202404}.
	
	\begin{lemma}\label{240911814}
		For any $n\in\bN_0$, 
		\begin{align}\label{2409098051}
			\|f\|_{\Psi H_{p,\theta}^{n}(\Omega)}^p\simeq_{\mathbf{I}} \sum_{k=0}^n\int_{\Omega}|\rho^kD_x^kf|^p\Psi^{-p}\rho^{\theta-d}\dd x\,.
		\end{align}
	\end{lemma}

	Proposition \ref{220418435} shows that, for $n\in\bN_0$ and $s\in(0,1)$, 
	$$
	\|f\|_{\Psi B_{p,\theta}^{n+s}(\Omega)}\simeq \|f\|_{\Psi H_{p,\theta}^{n}(\Omega)}+[D_x^nf]_{\Psi B_{p,\theta+np}^{s}(\Omega)}\,,
	$$ 
	where
	\begin{alignat}{2}\label{2409098052}
		[g]_{\Psi B_{p,\widetilde{\theta}}^{s}(\Omega)}^p:=\int_{\Omega}\bigg(\int_{h:|h|<\frac{\rho(x)}{2}}\frac{|g(x+h)-g(x)|^p}{|h|^{d+s p}}\dd h\bigg)\Psi(x)^p\rho(x)^{sp +\widetilde{\theta}-d}\dd x\,.
	\end{alignat}

	We now summarize the properties of $\Psi X_{p,\theta}^{\gamma}(\Omega)$. 
	Since $\zeta_0$ is fixed and the spaces $X_{p,\theta}^{\gamma}(\Omega)$ are independent of the choice of $\zeta_0$ (see \eqref{220526530}), we omit the dependence on $\zeta_0$.
	
	Most properties of $\Psi H_{p,\theta}^\gamma(\Omega)$ in Lemmas \ref{220527502111}–\ref{22.04.11.3} are proved in \cite{Seo202404}, while the corresponding results for $\Psi B_{p,\theta}^\gamma(\Omega)$ can be derived in a similar manner.
	However, Lemmas \ref{220527502111}.(4)-(5) and \ref{220527502}.(1)-(3), (9) are not covered in \cite{Seo202404}, and we focus on these cases below.
	It is worth noting that the complete proofs of Lemmas \ref{220527502111}-\ref{22.04.11.3} are provided also in \cite[Subsection 3.2, Appendix A]{Seo202304}.
	
	\begin{lemma}\label{220527502111}
		\,\,
		
		\begin{enumerate}
			\item $\Psi X_{p,\theta}^{\gamma}(\Omega)$ is a Banach space.
			
			\item $C_c^{\infty}(\Omega)$ is dense in $\Psi X_{p,\theta}^{\gamma}(\Omega)$.
			
			\item $\Psi X_{p,\theta}^{\gamma}(\Omega)$ is a reflexive Banach space with the dual $\Psi X_{p',\theta'}^{-\gamma}(\Omega)$, where 
			\begin{align}\label{241002631}
				\frac{1}{p}+\frac{1}{p'}=1\quad\text{and}\quad \frac{\theta}{p}+\frac{\theta'}{p'}=d\,.
			\end{align}
			Moreover, we have
			\begin{align}\label{220526214}
				\sup_{g\in C_c^{\infty}(\Omega),\,g\neq 0}\frac{\big|\la f,g \ra
					\big|}{\| g\|_{\Psi X_{p',\theta'}^{-\gamma}(\Omega)}}\simeq_{\mathbf{I}}\|f\|_{\Psi X_{p,\theta}^{\gamma}(\Omega)}\,.
			\end{align}
			
			\item Let $p_i\in(1,\infty)$ and $\gamma_i$, $\theta_i\in\bR$ for $i=0,\,1$.
			For each $t\in (0,1)$,
			\begin{align}\label{241031140}
				&\big[ \Psi X^{\gamma_0}_{p_0,\theta_0}(\Omega), \Psi X^{\gamma_1}_{p_1,\theta_1}(\Omega)\big]_t\simeq_N \Psi X^{\gamma_t}_{p_t,\theta_t}(\Omega)\,,
			\end{align}
			where $N=N(d,p_i,\theta_i,\gamma_i,t;i=1,2)$.
			Here, $[Y_0,Y_1]_t$ is the complex interpolation space of $Y_0$ and $Y_1$ (see \cite[Section 1.9]{triebel} for the definition and properties of the complex interpolation spaces), and the parameters $p_t\in(1,\infty)$ and $\gamma_t,\,\theta_t\in\bR$ satisfy
			\begin{align}\label{220526223}
				\frac{1}{p_t}=\frac{1-t}{p_0}+\frac{t}{p_1}\,\,,\,\,\,\,\gamma_t=(1-t)\gamma_0+t\gamma_1\,\,\,,\,\,\,\,\frac{\theta_t}{p_t}=(1-t)\frac{\theta_0}{p_0}+t\frac{\theta_1}{p_1}\,.
			\end{align}
			
			\item Let $p_i\in(1,\infty)$ and $\gamma_i$, $\theta_i\in\bR$ for $i=0,\,1$, with $\gamma_0\neq \gamma_1$.
			For each $t\in (0,1)$,
			\begin{align}\label{241108522}
				\big(\Psi X^{\gamma_0}_{p_0,\theta_0}(\Omega),\Psi X^{\gamma_1}_{p_1,\theta_1}(\Omega)\big)_{t,p_t}\simeq_N \Psi  B^{\gamma_t}_{p_t,\theta_t}(\Omega)
			\end{align}
			where $N=N(d,p_i,\theta_i,\gamma_i,t;i=1,2)$.
			Here, $(Y_0,Y_1)_{t,p_t}$ is the real interpolation space of $Y_0$ and $Y_1$ (see \cite[Section 1.3]{triebel} for the definition and properties of the real interpolation spaces), and the parameters $p_t\in(1,\infty)$ and $\gamma_t,\,\theta_t\in\bR$ satisfy \eqref{220526223}.
		\end{enumerate}
	\end{lemma}
	\begin{proof}
		(4) and (5) follow from the interpolation properties of $X_p^\gamma(\mathbb{R}^d)$ (see \cite[Section 2.4]{triebel2}) and the result in \cite[Theorem 1.2.4]{triebel}, where the latter is applied to the maps
		$$
		S^\Psi:\cD'(\Omega)\rightarrow \cD'(\bR^d)^{\bZ}:=\Big\{(f_n)_{n\in\bZ}\,:\,f_n\in\cD'(\bR^d)\Big\}\,\,,\,\, S^\Psi f:=\Big(\big(\Psi^{-1}f\zeta_{0,(n)}\big)(\ee^n\,\cdot\,)\Big)_{n\in\bZ}\,;
		$$
		$$
		R^\Psi:\cD'(\bR^d)^{\bZ}\rightarrow \cD'(\Omega)\quad,\quad R^\Psi\{f_n\}:=\Psi\sum_{n\in\bZ}\zeta_{1,(n)}(\,\cdot\,)f_n(\ee^{-n}\,\cdot\,)\,,
		$$
		where $\zeta_1(\,\cdot\,)=\zeta_0(\ee^{-1}\,\cdot\,)+\zeta_0(\,\cdot\,)+\zeta_0(\ee \,\cdot\,)$.
		Note that for any $p\in(1,\infty)$ and $\theta,\,\gamma\in\bR$, $S^\Psi$ is bounded from $\Psi H_{p,\theta}^\gamma(\Omega)$ to $\ell^p_\theta\big(H_p^\gamma(\bR^d)\big)$, while $R^\Psi$ is bounded from $\ell^p_\theta\big(H_p^\gamma(\bR^d)\big)$ to $\Psi H_{p,\theta}^\gamma(\Omega)$, where 
		$$
		\ell^p_\theta\big(H_p^\gamma(\bR^d)\big):=\Big\{(f_n)_{n\in\bZ}\subset \cD'(\bR^d)^\bZ\,:\,\big\|(f_n)_{n\in\bZ}\big\|_{\ell^p_\theta(\bZ;H_p^\gamma(\bR^d))}^p:=\sum_{n\in\bZ} \ee^{n\theta} \|f_n\|_{H_p^\gamma(\bR^d)}^p\Big\}\,.
		$$
		Moreover, $R^\Psi\circ S^\Psi$ is the identity map on $\Psi H_{p,\theta}^\gamma(\Omega)$.
		For more detailed proof, see \cite[Proposition A.2]{Seo202304}.
	\end{proof}
	
	For $k\in\bN_0$, $0<\alpha\leq 1$, and $\delta\in\bR$, we define
	\begin{align*}
		|f|_k^{(\delta)}:=\sum_{i=0}^k\big\|\rho^{\delta+i}D_x^if\big\|_{C(\Omega)}\quad,\quad |f|_{k,\alpha}^{(\delta)}\,\,\,\,&:=|f|_k^{(\delta)}+[\,\trho^{\delta+k+\alpha} D_x^kf\,]_{C^{0,\alpha}(\Omega)}\,.
	\end{align*}
	It is worth noting that for any regular Harnack function $\Psi$,
	\begin{alignat}{2}
		&|\Psi^{-1} f|_k^{(\delta)}&&\simeq_{N_1} \sum_{i=0}^k\sup_{\Omega}\left|\Psi^{-1}\rho^{\delta+i}D_x^if\right|\,,\label{240912237}\\
		&|\Psi^{-1}f|_{k,\alpha}^{(\delta)}&&\simeq_{N_2}|\Psi^{-1}f|_k^{(\delta)}+\sup_{x\in\Omega}\bigg(\Psi^{-1}(x)\rho^{\delta+k+\alpha}(x)\sup_{y:|y-x|<\frac{\rho(x)}{2}}\frac{\left|D_x^kf(x)-D_x^kf(y)\right|}{|x-y|^{\alpha}}\bigg)\,,\label{240912238}
	\end{alignat}
	where $N_1=N\big(d,k,\delta,\mathrm{C}_2(\Psi)\big)$ and $N_2=N\big(d,k,\alpha,\delta,\mathrm{C}_2(\Psi)\big)$.
	\eqref{240912237} follows by direct calculation, and \eqref{240912238} is proved in \cite[Proposition 3.13]{Seo202404}.
	
	\begin{lemma}\label{220527502}
		\,\,
		
		\begin{enumerate}
			\item If $1<p\leq 2$ (resp. $p\geq 2$), then
			\begin{align}\label{2411011106}
				\|f\|_{\Psi H_{p,\theta}^{\gamma}(\Omega)}\lesssim_{\mathbf{I}}\,(\text{resp.} \gtrsim_\mathbf{I})\,\|f\|_{\Psi B_{p,\theta}^{\gamma}(\Omega)}\,.
			\end{align}

			\item For any $s<\gamma$,
			\begin{align}\label{2411011107}
				\|f\|_{\Psi H_{p,\theta}^{s}(\Omega)}+\|f\|_{\Psi B_{p,\theta}^{s}(\Omega)}\lesssim_{\mathbf{I},s}\|f\|_{\Psi X_{p,\theta}^{\gamma}(\Omega)}\,.
			\end{align}
			
			\item Let $p_i\in(1,\infty)$ and $\gamma_i,\,\theta_i\in\bR$ for $i=0,\,1$, such that
			$$
			\gamma_0>\gamma_1\,\,,\quad \gamma_0-\frac{d}{p_0}=\gamma_1-\frac{d}{p_1}\,\,,\quad \frac{\theta_0}{p_0}=\frac{\theta_1}{p_1}\,.
			$$
			Then
			\begin{align}\label{2411011108}
				\|f\|_{\Psi X_{p_1,\theta_1}^{\gamma_1}(\Omega)}+\|f\|_{\Psi B_{p_1,\theta_1}^{\gamma_1}(\Omega)}\leq N\|f\|_{\Psi X_{p_0,\theta_0}^{\gamma_0}(\Omega)}
			\end{align}
			where $N=N(d,p_i,\gamma_i,\theta_i;i=0,1)$.

			\item For $k\in\bN$ such that $k>|\gamma|$, if $a\in C^{k,(0)}(\Omega)$, then
			\begin{align*}
				\|af\|_{\Psi X^{\gamma}_{p,\theta}(\Omega)}\lesssim_{\mathbf{I}}\|a\|_{C^{k,(0)}(\Omega)}\|f\|_{\Psi X^{\gamma}_{p,\theta}(\Omega)}\,.
			\end{align*}
			
			\item For any $\eta\in C_c^{\infty}(\bR_+)$,
			\begin{align*}
				\sum_{n\in\bZ}\ee^{n\theta}\big\|\big(\eta_{(n)}f\big)(\ee^n\cdot)\big\|^p_{\Psi X^{\gamma}_{p}}\lesssim_{\mathbf{I},\eta}\|f\|_{\Psi X^{\gamma}_{p,\theta}(\Omega)}^p\,.
			\end{align*}
			If $\eta$ additionally satisfies
			\begin{align*}
				\inf_{t\in\bR_+}\Big[\sum_{n\in\bZ}\eta(\ee^nt)\Big]>0\,,
			\end{align*}
			then
			\begin{align}\label{220526530}
				\|f\|_{\Psi X^{\gamma}_{p,\theta}(\Omega)}^p\simeq_{\mathbf{I},\eta} \sum_{n\in\bZ}\ee^{n\theta}\big\|\big(\eta_{(n)}f\big)(\ee^n\cdot)\big\|^p_{\Psi X^{\gamma}_{p}(\bR^d)}\,.
			\end{align}
			
			\item For any $s\in\bR$,
			\begin{align}\label{220526558}
				\|\trho^{\,s} f\|_{\Psi X^{\gamma}_{p,\theta}(\Omega)}\simeq_{\mathbf{I},s}\|f\|_{\Psi X^{\gamma}_{p,\theta+sp}(\Omega)}\,.
			\end{align}

			\item For any $k\in\bN$,
			\begin{align}\label{240911813}
				\|f\|_{\Psi X_{p,\theta}^{\gamma}(\Omega)}\simeq_{\mathbf{I},k}\sum_{i=0}^k\|D_x^i f\|_{\Psi X_{p,\theta+ip}^{\gamma-k}(\Omega)}\,.
			\end{align}

			\item Let $k\in\bN_0$ and $0<\alpha<1$ satisfy $\gamma-d/p>k+\alpha$.
			Then for any $f\in \Psi H^{\gamma}_{p,\theta}(\Omega)$,
			\begin{align}\label{241031713}
				\left|\Psi^{-1} f\right|^{(\theta/p)}_{k,\alpha}\leq_{\mathbf{I},k,\alpha} \|f\|_{\Psi X^{\gamma}_{p,\theta}(\Omega)}\,.
			\end{align}

			\item Let $t\in(0,1)$, and let $p_i\in(1,\infty),\,\theta_i,\,\gamma_i\in\bR$ ($i=0,\,1,\,,\,t$) be constants satisfying \eqref{220526223}.
			Then
			\begin{align}\label{241031137}
				\|f\|_{\Psi X_{p_t,\theta_t}^{\gamma_t}(\Omega)}\lesssim_N \|f\|_{\Psi X_{p_0,\theta_0}^{\gamma_0}(\Omega)}^{1-t}\|f\|_{\Psi X_{p_1,\theta_1}^{\gamma_1}(\Omega)}^{t}\,,
			\end{align}
			where $N=N(\mathrm{C}_2(\Psi),\gamma_i,p_i,\theta_i,t;i=1,2)$.

			\item For any $k,\,l\in\bN_0$, and a regular Harnack function $\Phi$ in $\Omega$,
			\begin{align}\label{2411011115}
				\big\|\big(D_x^k\Phi\big) D_x^lf\big\|_{\Psi X_{p,\theta}^\gamma(\Omega)}\lesssim \|\Phi f\|_{\Psi X_{p,\theta-(k+l)p}^{\gamma+l}(\Omega)}\,,
			\end{align}
			where $N=N(\mathbf{I},\mathrm{C}_2(\Phi),k,l)$.
			
		\end{enumerate}
	\end{lemma}

	\begin{proof}
		(1) This follows from the fact that $B_p^{\gamma}(\bR^d)\subset H_p^{\gamma}(\bR^d)$ if $1<p\leq 2$ and $H_p^{\gamma}(\bR^d)\subset B_p^{\gamma}(\bR^d)$ if $p\geq 2$ (see \cite[Proposition 2.3.2/2.(iii)]{triebel2}).
		
		(2) This follows from the fact that $X_p^{\gamma}(\bR^d)\subset H_p^{s}(\bR^d)\cap B_p^s(\bR^d)$ if $\gamma>s$ (see \cite[Proposition 2.3.2/2.(ii)]{triebel2}).

		(3) Note that $p_0<p_1$.
		Since $X_{p_0}^{\gamma_0}(\bR^d)\subset X_{p_1}^{\gamma_1}(\bR^d)\cap B_{p_1}^{\gamma_1}(\bR^d)$ (see \cite[Theorem 2.7.1]{triebel2}), we have
		\begin{align*}
			&\Big(\sum_{n\in\bZ}\ee
			^{n\theta_1}\|\big(f\zeta_{0,(n)}\big)(\ee^n\cdot)\|_{X_{p_1}^{\gamma_1}(\bR^d)}^{p_1}\Big)^{1/p_1}+\Big(\sum_{n\in\bZ}\ee^{n\theta_1}\|\big(f\zeta_{0,(n)}\big)(\ee^n\cdot)\|_{B_{p_1}^{\gamma_1}(\bR^d)}^{p_1}\Big)^{1/p_1}\\
			\leq \,&N\Big(\sum_{n\in\bZ}\ee^{n\theta_0}\|\big(f\zeta_{0,(n)}\big)(\ee^n\cdot)\|_{X_{p_0}^{\gamma_0}(\bR^d)}^{p_1}\Big)^{1/p_1}\leq N\Big(\sum_{n\in\bZ}\ee^{n\theta_0}\|\big(f\zeta_{0,(n)}\big)(\ee^n\cdot)\|_{X_{p_0}^{\gamma_0}(\bR^d)}^{p_0}\Big)^{1/p_0}\,,
		\end{align*}
		where $N=N(d,p_i,\gamma_i;i=0,1)$.
		
		(9) This follows from \eqref{241031140} in combination with \cite[Theorem 1.9.3.(f)]{triebel}.
	\end{proof}
	
	\begin{lemma}\label{22.02.16.1}
		There exist bounded linear maps 
		$$
		\Lambda_0\,:\,\Psi H_{p,\theta}^{\gamma}\rightarrow \Psi H_{p,\theta}^{\gamma+1}(\Omega)\quad\text{and}\quad \,\Lambda_1,\,\ldots,\,\Lambda_d:\Psi H_{p,\theta}^{\gamma}\rightarrow \Psi H_{p,\theta-p}^{\gamma+1}(\Omega)
		$$
		such that, for any $f\in \Psi H_{p,\theta}^{\gamma}(\Omega)$, we have 
		$$
		f=\Lambda_0 f+\sum_{i=1}^dD_i(\Lambda_if)
		$$
		and 
		\begin{align*}
			\begin{gathered}
				\|\Lambda_0 f\|_{\Psi H_{p,\theta}^{\gamma+1}(\Omega)}+\sum_{i=1}^d\|\Lambda_i f\|_{\Psi H_{p,\theta-p}^{\gamma+1}(\Omega)}\leq N\|f\|_{\Psi H_{p,\theta}^{\gamma}(\Omega)}\,,
			\end{gathered}
		\end{align*}	
		where $N=N(d,p,\gamma,\theta,\mathrm{C}_2(\Psi))$.
	\end{lemma}
	
	The following corollary follows directly from \eqref{220526214} and Lemma \ref{22.02.16.1} (see also the proof of \cite[Lemma 3.12(2)]{Seo202404}).
	
	\begin{corollary}\label{240909836}
		For any $k\in\bN$, there exist bounded linear maps $\Lambda_\alpha:\Psi X_{p,\theta}^\gamma(\Omega)\rightarrow \Psi X_{p,\theta-|\alpha|p}^{\gamma+k}(\Omega)$, for all multi-indices $\alpha$ with $|\alpha|\leq k$, such that 
		$$
		f=\sum_{|\alpha|\leq k}D^\alpha_x\Lambda_\alpha f\,.
		$$
		Moreover,
		\begin{align}\label{241111154}
			\|f\|_{\Psi X_{p,\theta}^{\gamma}(\Omega)}\simeq_{\mathbf{I},k} \inf\bigg\{\sum_{|\alpha|\leq k}\|f_{\alpha}\|_{\Psi X_{p,\theta-|\alpha|p}^{\gamma+k}(\Omega)}:\,f=\sum_{|\alpha|\leq k}D_x^{\alpha}f_{\alpha}\bigg\}\,.
		\end{align}
	\end{corollary}
	
	\begin{lemma}\label{22.04.11.3}
		Let $\eta\in C_c^{\infty}(\bR^d)$ satisfy $\eta=1$ on $B_1(0)$ and $\mathrm{supp}(\eta)\subset B_2(0)$.
		For each $i\in\bN$, let $N(i)\in\bN$ be a constant satisfying
		$$
		\mathrm{supp}\Big(\sum_{|n|\leq i}\zeta_{0,(n)}\Big)\subset \big\{x\in\Omega\,:\,\big(N(i)/2\big)^{-1}\leq \rho(x)\leq N(i)/2\big\}\,.
		$$
		Let $\Lambda_i$, $\Lambda_{i,j}$, $\Lambda_{i,j,k}$ be linear functionals on $\cD'(\Omega)$ defined as
		\begin{align*}
			\Lambda_{i}f:=\Big(\sum_{|n|\leq i}\zeta_{0,(n)}\Big)f\,\,\,,\,\,\,\,\Lambda_{i,j}f=\eta(j^{-1}\cdot\,)\Lambda_if\,\,\,,\,\,\,\,\Lambda_{i,j,k}f=\big(\Lambda_{i,j}f\big)^{(N(i)^{-1}k^{-1})}\,,
		\end{align*}
		where $\big(\Lambda_{i,j}f\big)^{(\epsilon)}$ is defined in the same way as in \eqref{21.04.23.1}.
		Then for any $p\in(1,\infty)$, $\gamma$,\,$\theta\in\bR$, and regular Harnack function $\Psi$, the following hold:
		\begin{enumerate}
			\item For any $f\in \cD'(\Omega)$, $\Lambda_{i,j,k}f\in C_c^{\infty}(\Omega)$\,.
			
			\item  For any $f\in \Psi H_{p,\theta}^{\gamma}(\Omega)$,
			\begin{align}
				&\sup_i\|\Lambda_if\|_{\Psi H_{p,\theta}^{\gamma}(\Omega)}\leq N_1\|f\|_{\Psi	H_{p,\theta}^{\gamma}(\Omega)}\nonumber\\
				&\,\sup_j\|\Lambda_{i,j}f\|_{\Psi H_{p,\theta}^{\gamma}(\Omega)}\leq N_2\|f\|_{\Psi H_{p,\theta}^{\gamma}(\Omega)}\label{2401021134}\\
				&\,\sup_k\|\Lambda_{i,j,k}f\|_{\Psi H_{p,\theta}^{\gamma}(\Omega)}\leq N_3\|f\|_{\Psi H_{p,\theta}^{\gamma}(\Omega)}\,,\nonumber
			\end{align}
			where $N_1$, $N_2$, $N_3$ are constants independent of $f$.
			
			\item For any $f\in \Psi H_{p,\theta}^{\gamma}(\Omega)$,
			\begin{align}\label{2401021135}
				\lim_{k\rightarrow \infty}\Lambda_{i,j,k}f=\Lambda_{i,j}f\,\,\,,\,\,\,\lim_{j\rightarrow \infty}\Lambda_{i,j}f=\Lambda_{i}f\,\,\,,\,\,\,\lim_{i\rightarrow \infty}\Lambda_i f=f\quad\text{in}\quad \Psi H_{p,\theta}^{\gamma}(\Omega)\,.
			\end{align}
		\end{enumerate}
	\end{lemma}
	
	\mysection{Key estimates for the heat equation}\label{0030}
	
	For $0<\nu_1\leq \nu_2<\infty$ and $T\in (0,\infty]$, we denote
	\begin{itemize}
		\item $\mathrm{M}(\nu_1,\nu_2)$: the set of all real symmetric $d\times d$ matrices $(\alpha^{ij})_{d\times d}$ satisfying
		$$
		\nu_1|\xi|^2\leq \sum_{i,j=1}^d\alpha^{ij}\xi_i\xi_j\leq \nu_2|\xi|^2\qquad\forall\,\, \xi\in\bR^d\,.
		$$
		
		\item $\cM_T(\nu_1,\nu_2)$: the set of all $\cL:=\sum_{i,j=1}^da^{ij}(\cdot)D_{ij}$, where $\{a^{ij}(\cdot)\}_{i,j=1,\ldots,d}$ is a collection of time-measurable functions in $(0,T)$ such that $\big(a^{ij}(t)\big)_{d\times d}\in \mathrm{M}(\nu_1,\nu_2)$ for all $t\in(0,T)$.
		
	\end{itemize}
	Throughout this section we assume that $\Omega$ is an open set in $\bR^d$,
	\begin{align}\label{241028314}
		T\in(0,\infty)\,\,,\,\,\,\,0<\nu_1\leq \nu_2<\infty\,\,,\,\,\text{and}\,\,\,\,\cL\in\cM_T(\nu_1,\nu_2)\,.
	\end{align}
	We deal with the classical ($\alpha=1$) and time-fractional ($0<\alpha<1$) parabolic equations,
	\begin{align}\label{2206201120}
		\partial_t^\alpha u=\cL u+f:=\sum_{i,j=1}^da^{ij}(t)D_{ij} u + f\,\,,\quad t\in(0,T]\quad;\quad u(0,\cdot)=u_0\,.
	\end{align}
	Here, $\partial_t^1$ denotes the classical time derivative $\frac{\mathrm{d}}{\mathrm{d}t}$, while for $0<\alpha<1$, $\partial_t^\alpha$ denotes the Caputo fractional derivative introduced in Definition \ref{240907807}.
	
	In this paper, our main focus is on the classical and time-fractional heat equations, \eqref{para} and \eqref{parafrac}.
	For these equations, Lemma 2.6 of \cite{Seo202404} can be applied directly.
	However, we also consider the operator $\cL$ in \eqref{2206201120}, which involves variable coefficients.
	To treat the operator $\cL$, we introduce the following definition:

	\begin{defn}\label{21.11.10.1}
		Let $\phi$ be a positive superharmonic function in $\Omega$.
		For $\delta\in (0,1]$ and $p\in(1,\infty)$, we denote by $\cI(\phi,p,\delta)$ the set of all constants $\mu\in (-\frac{1}{p},1-\frac{1}{p})$ such that there exists a constant $\mathrm{C}_4>0$ satisfying the inequality 
		\begin{align}\label{21.07.12.1}
			\int_{\Omega\cap\{u\neq 0\}}|u|^{p-2}|\nabla u|^2\phi^{-\mu p} \dd x\leq \mathrm{C}_4\int_{\Omega}\Big(-\sum_{i,j=1}^d\alpha^{ij}D_{ij}u\Big)\cdot u|u|^{p-2}\phi^{-\mu p} \dd x
		\end{align}
		for all $u\in C_c^{\infty}(\Omega)$ and $(\alpha^{ij})_{d\times d}\in \mathrm{M}(\delta,1)$.
	\end{defn}
	
	We use the set $\cI(\phi,p,\delta)$ to formulate the main theorems, specifically Theorems \ref{05.11.2} and \ref{22.02.18.6}.
	According to Lemma \ref{03.30}, if $\delta=1$, corresponding to the case where only the Laplacian $\Delta$ is considered, then $\cI(\phi,p,1)$ is the same as $(-\frac{1}{p},1-\frac{1}{p})$.
	Notably, the following proposition guarantees the existence of a non-empty interval contained in $\cI(\phi,p,\delta)$, even when $\delta\in(0,1)$, without imposing additional assumptions on $\Omega$, $\phi$, or $p$:
	
	\begin{prop}\label{05.11.1}
		Let $p\in(1,\infty)$ and $\delta\in(0,1]$.
		\begin{enumerate}
			\item If
			\begin{align}\label{22.04.12.1029}
				\mu\in \Big(-\frac{(p-1)/p}{p(\delta^{-1/2}+1)/2-1},\frac{(p-1)/p}{p(\delta^{-1/2}-1)/2+1}\Big)\,,
			\end{align}
			then for any positive superharmonic function $\phi$ in $\Omega$, it holds that $\mu\in \cI(\phi,p,\delta)$ and the constant $\mathrm{C}_4$ in \eqref{21.07.12.1} can be chosen to depend only on $\delta$, $p$, and $\mu$.
			In particular, 
			$$
			\cI(\phi,p,1)=\left(-\frac{1}{p},1-\frac{1}{p}\right)\,.
			$$
			
			\item Suppose $\phi>0$ satisfies that for any $(\alpha^{ij})_{d\times d}\in\mathrm{M}(\delta,1)$, $\sum_{i,j=1}^d\alpha^{ij}D_{ij}\phi\leq 0$ in the sense of distributions.
			Then
			$$
			\cI(\phi,p,\delta)=\left(-\frac{1}{p},1-\frac{1}{p}\right)\,.
			$$
			Moreover, for any $\mu\in (-\frac{1}{p},1-\frac{1}{p})$, the constant $\mathrm{C}_4$ in \eqref{21.07.12.1} can be chosen to depend only on $\delta$, $p$, and $\mu$.
		\end{enumerate}
	\end{prop}
	
	Proposition \ref{05.11.1}.(2) is applied to results concerning convex domains and the totally vanishing exterior Reifenberg condition; see Subsections \ref{convex} and \ref{ERD}, respectively.

	The proof of Proposition \ref{05.11.1} relies on the following two lemmas:
	
	\begin{lemma}[Lemma 2.6 in \cite{Seo202404}]\label{03.30}
		Let $p\in(1,\infty)$ and $c<1$, and suppose that $u\in C(\Omega)$ satisfies
		\begin{equation}\label{22.01.25.2}
			\begin{gathered}
				\mathrm{supp}(u)\,\,\text{is a compact subset of}\,\,\,\Omega\,,\\
				u\in C_{\mathrm{loc}}^2\big(\{x\in\Omega\,:\,u(x)\neq 0\}\big)\,\,\,,\,\,\,\text{and}\quad \int_{\{u\neq 0\}}|u|^{p-1}|D_x^2u|\dd x<\infty\,,\
			\end{gathered}
		\end{equation}
		and $\phi$ is a positive superharmonic function in a neighborhood of $\mathrm{supp}(u)$.
		\begin{enumerate}
			\item If $\phi$ is twice continuously differentiable, then 
			\begin{align}\label{230121211}
				\int_{\Omega}|u|^p\phi^{c-2}|\nabla \phi|^2 \dd x \leq \Big(\frac{p}{1-c}\Big)^2\int_{\Omega\cap\{u\neq 0\}}|u|^{p-2}|\nabla u|^2\phi^c \dd x\,.
			\end{align}
			
			\item If we additionally assume that $c\in (-p+1,1)$ and $(\Delta u)1_{\{u\neq 0\}}$ is bounded, then
			\begin{align*}
				\int_{\Omega\cap\{u\neq 0\}}|u|^{p-2}|\nabla u|^2\phi^c \dd x\leq N\int_{\Omega\cap\{u\neq 0\}}(-\Delta u)\cdot u|u|^{p-2}\phi^c \dd x\,,
			\end{align*}
			where $N=N(p,c)>0$.
			
			\item If the Hardy inequality \eqref{hardy} holds for $\Omega$, then
			\begin{align}\label{22062422512}
				\int_{\Omega}|u|^p\phi^c\rho^{-2}\dd x\leq N\int_{\Omega\cap\{u\neq 0\}}|u|^{p-2}|\nabla u|^2 \phi^c \dd x\,,
			\end{align}
			where $N=N(p,c,\mathrm{C}_0(\Omega))>0$.
		\end{enumerate}
	\end{lemma}
		
	\begin{lemma}[Lemma A.1 in \cite{Seo202404}]\label{21.04.23.4}
		Let $p\in(1,\infty)$ and $u\in C(\bR^d)$ satisfy \eqref{22.01.25.2}.
		
		\begin{enumerate}
			\item $|u|^{p/2-1}u\in W_2^1(\bR^d)$ and $D_i(|u|^{p/2-1}u)=\frac{p}{2}|u|^{p/2-1}(D_iu)1_{\{u\neq 0\}}$.
			
			\item  $|u|^p\in W_1^2(\bR^d)$ and 
			\begin{align*}
				\begin{split}
					D_i\big(|u|^p\big)\,&=p|u|^{p-2}uD_iu 1_{\{u\neq 0\}}\,\,;\\
					D_{ij}\big(|u|^p\big)\,&=\big(p|u|^{p-2}uD_{ij}u+p(p-1)|u|^{p-2}D_iuD_ju\big)\,1_{\{u\neq 0\}}\,.
				\end{split}	
			\end{align*}
		\end{enumerate}
	\end{lemma}

	\begin{proof}[Proof of Proposition \ref{05.11.1}]
		
		(1) Let $\mu$ satisfy \eqref{22.04.12.1029}.
		Our goal is to prove \eqref{21.07.12.1} for each $u\in C_c^\infty(\Omega)$ with a  constant $\mathrm{C}_4$ depending only on $\mu$, $p$, and $\delta$.
		Let us fix $u\in C_c^\infty(\Omega)$.
		By Lemma \ref{21.04.23.5}, by considering $\phi^{(\epsilon)}$ in \eqref{21.04.23.1} instead of $\phi$, it suffices to prove \eqref{21.07.12.1} only for a positive smooth superharmonic function $\phi$ in a neighborhood of $\text{supp}(u)$.
		
		Put $c:=-\mu p\in(-p+1,1)$ and $v=u\phi^{c/(2p-2)}$. 
		Using Lemmas \ref{21.04.23.4} and \ref{03.30}.(1), and the fact that $(\alpha^{ij})\in\mathrm{M}(\delta,1)$, we have
		\begin{align*}
			&-\sum_{i,j}\int_{\Omega}\alpha^{ij}u_{ij}|u|^{p-2}u\phi^c\dd x\\
			=\,&(p-1)\sum_{i,j}\int_{\Omega}\alpha^{ij}u_iu_j|u|^{p-2}\phi^c\dd x+c\sum_{i,j}\int_{\Omega}\alpha^{ij}|u|^{p-2}uu_i\phi_j\phi^{c-1}\dd x\\
			=\,&(p-1)\int_{\Omega}\Big(\sum_{i,j}\alpha^{ij}v_iv_j\Big)|v|^{p-2}\phi^{c'}\dd x-\frac{c^2}{4(p-1)}\int_{\Omega}|v|^p\Big(\sum_{i,j}\alpha^{ij}\phi_i\phi_j\Big)\phi^{c'-2}\dd x\\
			\geq \,&(p-1)\delta \int_{\Omega}|\nabla v|^2|v|^{p-2}\phi^{c'}\dd x-\frac{c^2}{4(p-1)}\int_{\Omega}|v|^p|\nabla \phi|^2\phi^{c'-2}\dd x\\
			\geq\,&\kappa'\int_{\Omega}|\nabla v|^2|v|^{p-2}\phi^{c'}\dd x\,,
		\end{align*}
		where $c':=\frac{(p-2)c}{2p-2}<1$ and 
		$
		\kappa'=\delta(p-1)-\frac{1}{4(p-1)}\big(\frac{pc}{1-c'}\big)^2\,.
		$
		Observe that $\kappa'>0$ if and only if $\mu:=-\frac{c}{p}$ satisfies \eqref{22.04.12.1029}. 
		Therefore it suffices to show that
		\begin{align}\label{220606927}
			\int_{\Omega}|u|^{p-2}|\nabla u|^2\phi^{c} \dd x\leq N(p,\mu)\int_{\Omega}|v|^{p-2}|\nabla v|^2\phi^{c'} \dd x\,.
		\end{align}
		Note that
		\begin{align*}
			&\int_{\Omega\cap\{u\neq 0\}}|v|^{p-2}|\nabla v|^2\phi^{c'} \dd x\\
			\geq\,&\int_{\Omega\cap\{u\neq 0\}}|u|^{p-2}|\nabla u|^2\phi^{c} \dd x+\frac{c}{p-1}\int_{\Omega\cap\{u\neq 0\}}|u|^{p-2}u(\nabla u \cdot \nabla \phi)\phi^{c-1} \dd x\,.
		\end{align*}
		
		If $c\in[0,1)$, then because $\Delta \phi\leq 0$ on $\mathrm{supp}(u)$, we obtain that 
		$$
		\Delta(\phi^c)=c(\Delta \phi)\phi^{c-1}-c(1-c)|\nabla \phi|^2\phi^{c-1}\leq 0
		$$ 
		on $\text{supp}(u)$. 
		This implies that
		$$
		\frac{c}{p-1}\int_{\Omega\cap\{u\neq 0\}}|u|^{p-2}u(\nabla u \cdot \nabla \phi)\phi^{c-1} \dd x=-\frac{1}{p(p-1)}\int_{\Omega}|u|^{p}\Delta(\phi^c) \dd x\geq 0\,.
		$$
		Therefore \eqref{220606927} holds.
		
		If $c\in(-p+1,0)$, then Lemma \ref{03.30}.(1) implies
		\begin{align*}
			&\int_{\Omega\cap\{u\neq 0\}}|u|^{p-2}|\nabla u|^2\phi^{c} \dd x+\frac{c}{p-1}\int_{\Omega\cap\{u\neq 0\}}|u|^{p-2}u(\nabla u \cdot \nabla \phi)\phi^{c-1} \dd x\\
			\geq\,& \int_{\Omega\cap\{u\neq 0\}}|u|^{p-2}|\nabla u|^2\phi^{c} \dd x\\
			&+\frac{c}{p-1}\Big(\int_{\Omega\cap\{u\neq 0\}}|u|^{p-2}|\nabla u|^2\phi^{c} \dd x\Big)^{1/2}\Big(\int_{\Omega}|u|^{p}\phi^{c-2}|\nabla \phi|^2 \dd x\Big)^{1/2}\\
			\geq\,&\frac{p-1+c}{(p-1)(1-c)}\int_{\Omega\cap\{u\neq 0\}}|u|^{p-2}|\nabla u|^2\phi^{c} \dd x\,.
		\end{align*}
		Since $p-1+c>0$, the proof is complete.
		
		(2) For a fixed $\mathrm{A}=(\alpha^{ij})_{d\times d}\in \mathrm{M}(\delta,1)$, take $\mathrm{B}\in\mathrm{M}\big(\sqrt{\delta},1\big)$ such that $\mathrm{B}^2=\mathrm{A}$.
		We denote $u_{\mathrm{B}}(y):=u(\mathrm{B}y)$ and $\phi_{\mathrm{B}}(y):=\phi(\mathrm{B}y)$.
		Since 
		$$
		\Delta \phi_{\mathrm{B}}=\sum_{i,j=1}^d\alpha^{ij}\big(D_{ij}\phi\big)(\mathrm{B}\,\cdot\,)\leq 0
		$$
		in the sense of distributions on $\mathrm{B}^{-1}\Omega:=\{\mathrm{B}^{-1}x\,:\,x\in\Omega\}$,
		Lemma \ref{03.30}.(2) implies that for any $u\in C_c^{\infty}(\Omega)$,
		\begin{alignat*}{2}
			\int_{\Omega\cap\{u\neq 0\}}|u|^{p-2}|\nabla u|^2\phi^{-\mu p}\dd x
			&\leq &&\delta^{-1} \mathrm{det(B)}\int_{\mathrm{B}^{-1}\Omega \cap \{u_\mathrm{B}\neq 0\}}|u_\mathrm{B}|^{p-2}|\nabla u_\mathrm{B}|^2 \phi_{\mathrm{B}}^{-\mu p}\dd y\\
			&\lesssim_{p,\mu}\,&&\delta^{-1} \mathrm{det(B)}\int_{\mathrm{B}^{-1}\Omega}(-\Delta u_\mathrm{B})\cdot u_\mathrm{B}|u_\mathrm{B}|^{p-2}\phi_{\mathrm{B}}^{-\mu p}\dd y\\
			&= &&\delta^{-1}\int_{\Omega}\big(-\alpha^{ij}D_{ij}u\big)\cdot u|u|^{p-2}\phi^{-\mu p} \dd x\,.
		\end{alignat*}
		Therefore the proof is complete.
	\end{proof}

	The following theorem is the main result of this section:
	
	\begin{thm}\label{05.11.2}
		Let $\Omega$ be an open set admitting the Hardy inequality \eqref{hardy} and $\phi$ be a positive superharmonic function in $\Omega$, and let $0<\alpha\leq 1$ and $0<T<\infty$. 
		For any $p\in (1,\infty)$ and  $\mu\in \cI(\phi,p,\nu_1/\nu_2)$,the following estimate holds for every $u\in C_c^{\infty}\big((0,T]\times \Omega\big)$ with $f:=\partial_t^{\alpha}u-\cL u$:	
		\begin{align}\label{230328758}
			\begin{split}
				\int_0^T\int_{\Omega}|u|^{p}\phi^{-\mu p}\rho^{-2}\dd x\dd t
				\leq N \int_0^T\int_{\Omega}|f|^p\phi^{-\mu p}\rho^{2p-2}\dd x\dd t\,,
			\end{split}
		\end{align}
		where $N=N(p,\mu,\mathrm{C}_0(\Omega),\mathrm{C}_4)$.
	\end{thm}
	
	To prove Theorem \ref{05.11.2}, we use the following elementary fact: for $F\in C^1\big([0,T]\big)$ with $F(0)=0$, 
	\begin{align}\label{240426821}
		\int_0^T|F(t)|^{p-2}F(t)\partial_t^\alpha F(t)\dd t\geq 0\,.
	\end{align}
	If $\alpha=1$, then it follows from a direct calculation, while for $0<\alpha<1$, it follows from Lemma \ref{240426733}.

	\begin{proof}[Proof of Theorem \ref{05.11.2}]
		By Lemma \ref{03.30}.(3) and the fact that $\mu\in \cI(\phi,p,\nu_1/\nu_2)$, we have
		\begin{align}\label{230328743}
			\begin{aligned}
				\int_0^{T}\int_{\Omega}|u|^{p}\phi^{-\mu p}\rho^{-2}\dd x\dd t\,&\lesssim_{p,\mu,\mathrm{C}_1(\Omega)} \int_0^{T}\int_{\Omega}|\nabla u|^2|u|^{p-2}\phi^{-\mu p}\dd x\dd t\\
				&\leq \mathrm{C}_4 \int_0^{T}\int_{\Omega}\Big(-\sum_{i,j}a^{ij}(t)D_{ij} u\Big) u|u|^{p-2}\phi^{-\mu p}\dd x\dd t\,.
			\end{aligned}
		\end{align}
		
		Integrating
		$$
		p\,\big(\partial_t^\alpha u\big) u|u|^{p-2}\,\phi^{-\mu p}-p\,\sum_{i,j=1}^da^{ij}(t)D_{ij} u\cdot u|u|^{p-2}\phi^{-\mu p}=pf\cdot u|u|^{p-2}\phi^{-\mu p}
		$$
		over $(0,T]\times \Omega$ and applying Young's inequality yield, for any $\epsilon>0$, 
		\begin{alignat}{2}
			&p\int_0^{T}\int_{\Omega}\Big(-\sum_{i,j}a^{ij}(t)D_{ij} u\Big) u|u|^{p-2}\phi^{-\mu p}\dd x\dd t \nonumber\\
			\leq\,&-p\int_0^T\int_\Omega\big(\partial_t^\alpha u\big)u|u|^{p-2}\phi^{-\mu p}\dd x\dd t\label{230328742}\\
			&+ \epsilon^{-p+1}\int_0^{T}\int_{\Omega}|f|^p\phi^{-\mu p}\rho^{2p-2}\dd x\dd t+(p-1)\epsilon\int_0^{T}\int_\Omega|u|^{p}\phi^{-\mu p} \rho^{-2}\dd x\dd t\nonumber\\
			\leq \,&\epsilon^{-p+1}\int_0^{T}\int_{\Omega}|f|^p\phi^{-\mu p}\rho^{2p-2}\dd x\dd t+(p-1)\epsilon\int_0^{T}\int_\Omega|u|^{p}\phi^{-\mu p} \rho^{-2}\dd x\dd t\,,\nonumber
		\end{alignat}
		where the last inequality follows from \eqref{240426821}.
		Combining \eqref{230328743} and \eqref{230328742} and choosing $\epsilon>0$ sufficiently small, we obtain \eqref{230328758}; note that because $\phi^{-\mu p}$ is locally integrable (see Lemmas \ref{21.04.23.3}.(2) and (3)), the first term in \eqref{230328743} is finite.
	\end{proof}
	
	Recall that $\langle F,\zeta\rangle$ denotes the action of $F\in\cD'(\Omega)$ on $\zeta\in C_c^\infty(\Omega)$.
	\begin{lemma}[Existence of a very weak solution]\label{21.05.25.300}
		Let $\Omega$ be an open set admitting the Hardy inequality \eqref{hardy}.
		For any $f\in C_c^{\infty}\big((0,T]\times\Omega\big)$, there exists $u\in L_{\infty}\big([0,T]\times\Omega\big)$ satisfying the following properties:
		\begin{enumerate}
			\item For any $\zeta\in C_c^{\infty}(\Omega)$ and $\eta\in C_c^{\infty}\big([0,T)\big)$,
			\begin{align}\label{2206301251}
				\int_0^T \Big(I_t^{1-\alpha}\la u(t,\cdot),\zeta\ra \Big)\eta'(t)\dd t=-\int_0^T\la \Delta u(t,\cdot)+f(t,\cdot),\zeta\ra \eta(t)\dd t\,.
			\end{align}
			
			\item For any $p\in(1,\infty)$, $\mu\in\big(-\frac{1}{p},1-\frac{1}{p}\big)$, and any positive superharmonic function $\phi$ in $\Omega$,
			\begin{align}\label{2206131031}
				\int_0^T\int_{\Omega}|u|^p\phi^{-\mu p}\rho^{-2}\dd x\dd t\leq N\int_0^{T}\int_{\Omega}|f|^p\phi^{-\mu p}\rho^{2p-2}\dd x\dd s\,,
			\end{align}
			where $N=N(p,c,\mathrm{C}_0(\Omega))$.
		\end{enumerate}
	\end{lemma}
	Note that \eqref{2206301251} means that $\partial_t^\alpha u=\Delta u+f$ with the initial condition $u(0,\cdot)\equiv0$ (see Remark \ref{2404201227}.(2)).
	
	\begin{proof}[Proof of Lemma \ref{21.05.25.300}]
		Take a sequence of infinitely smooth bounded domains $\big(\Omega_n\big)_{n\in\bN}$ such that
		$$
		\text{supp}(f)\subset [0,T]\times\Omega_1\,\,,\quad \overline{\Omega_n}\subset \Omega_{n+1}\,\,,\quad \bigcup_{n}\Omega_n=\Omega
		$$
		(see, \textit{e.g.}, \cite[Proposition 8.2.1]{DD_2008}).
		For any $H\in C_c^{\infty}((0,T]\times \Omega_1)$, there exists $U\in C^{\infty}\big([0,T]\times \Omega_n\big)\cap C\big([0,T]\times \overline{\Omega_n}\big)$ satisfying
		\begin{align}\label{240909536}
			\partial_t^\alpha U=\Delta U+H1_{\Omega_1}\quad \text{on}\,\,\,\,(0,T]\times\Omega_n\quad; \quad U(t,x)=0\quad\text{if}\,\,\,\, t=0\,\,\,\,\text{or}\,\,\,\,x\in\partial\Omega_n\,,
		\end{align}
		and
		\begin{align}\label{241002610}
			\int_0^T\int_{\Omega_n}|U|^{p-1}|D_x^2 U|\dd x\dd t<\infty\,.
		\end{align}
		It is a classical result if $\alpha=1$. 
		For $0<\alpha<1$, the existence of such $U$ is established in Lemma \ref{240426835}.
		Denote this solution $U$ by $R_nH$, and note the following properties of $R_nH$:
		\begin{itemize}
			\item From \eqref{240909536}, we deduce that for any $\zeta\in C_c^{\infty}(\Omega_n)$ and $\eta\in C_c^{\infty}\big([0,T)\big)$,
			\begin{align}\label{240909550}
				\int_0^T \Big(I_t^{1-\alpha}\la R_nH(t,\cdot),\zeta\ra \Big)\eta'(t)\dd t=-\int_0^T\la \Delta \big(R_nH\big)(t,\cdot)+H(t,\cdot)1_{\Omega_1},\zeta\ra \eta(t)\dd t\,.
			\end{align}

			\item The maximum principle (for when $0<\alpha<1$, see \cite[Theorem 2]{LUCHKO2009218}) yields 
			\begin{align}\label{240909624}
				\big|R_nH(t,x)\big|\leq \frac{\|H\|_\infty}{\Gamma(1+\alpha)}t^\alpha\,.
			\end{align}

			\item $\overline{\Omega}_n$ is a compact subset of $\Omega$, $R_{n}H\in C([0,T]\times \overline{\Omega}_n)\cap C^{\infty}([0,T]\times \Omega_n)$, and $R_{n}H\equiv 0$ on $\overline{\big((0,T]\times \Omega_n\big)}\setminus \big((0,T]\times \Omega_n\big)$.
			These and \eqref{241002610} imply that $R_{n}H(t,\,\cdot\,)1_{\Omega_n}$ satisfies condition \eqref{22.01.25.2}.
			Repeating the proof of Theorem \ref{05.11.2} for $(R_nH)1_{\Omega_n}$ in place of $u$ and using Lemma \ref{03.30}.(2) instead of \eqref{21.07.12.1}, we obtain that for any $p\in(1,\infty)$, $\mu\in\big(-\frac{1}{p},1-\frac{1}{p}\big)$, and any positive superharmonic function $\phi$,
			\begin{align}\label{2206131250}
				\int_0^T\int_{\Omega}|R_{n}H|^{p}\rho^{-2}\phi^{-\mu p}\dd x\dd s
				\leq N(p,c,\mathrm{C}_0(\Omega))\int_0^T\int_{\Omega}|H|^p\phi^{-\mu p}\rho^{2p-2}\dd x\dd s\,.
			\end{align}
		\end{itemize}

		Take $F\in C_c^{\infty}\big([0,T]\times \Omega_1\big)$ such that $|f|\leq F$, and put $f_1:=F+f$ and $f_2:=F$ so that $f_1,\,f_2\geq 0$ and $f=f^1-f^2$.
		
		For $v_n:=R_n\big(f^11_{\Omega_n}\big)$, the maximum principle implies that $0\leq v_n\leq v_{n+1}$.
		We denote the pointwise limit of $v_n$ by $v$.
		In the same manner, we define $w:=\lim_{n\rightarrow\infty}R_n\big(f^2 1_{\Omega_n}\big)$.
		By \eqref{240909624}, $v$ and $w$ are bounded.
		Put
		$$
		u:=v-w=\lim_{n\rightarrow\infty}R_{n}f1_{\Omega_n}\,.
		$$
		it follows from \eqref{240909624} that $|u(t,x)|\lesssim t^\alpha$, which implies that $I^{1-\alpha}\big(|\la u(t,\cdot),\zeta\ra|\big)\lesssim t$.
		Applying Lebesgue's dominated convergence theorem to \eqref{240909550} and Fatou's lemma to \eqref{2206131250} yields \eqref{2206301251} and \eqref{2206131031},, respectively.
	\end{proof}

	\mysection{Solvability of the heat equation}\label{mainresultsection}
	
	Throughout this section, we assume \eqref{241028314} and that $\Omega$ be an open set in $\bR^d$, $d\in\bN$.

	\subsection{Solution spaces}\label{0052}
	This subsection introduces the solution spaces $\Psi\mathring{\cH}_{p,\theta}^{\alpha,\gamma}(\Omega,T)$ in the case where the initial datum $u_0$ in \eqref{para} and \eqref{parafrac} is zero, and $\Psi\cH_{p,\theta}^{\alpha,\gamma}(\Omega,T)$ for general $u_0$.
	
	Throughout this subsection, we assume that $\Psi$ is a regular Harnack function in $\Omega$, that $0<\alpha\le 1$, $1<p<\infty$, and $\gamma,\theta\in\bR$; we then denote
	\begin{equation}\label{221015645}
		\begin{alignedat}{2}
			\bH_{p}^{\gamma}(\bR^d,T)&:=L_p\big((0,T);H^{\gamma}_{p}(\bR^d)\big)\,\,,&\quad \bL_{p}(\bR^d,T)&:=L_p\big((0,T);L_p(\bR^d)\big)\,\,,\\
			\bH^{\gamma}_{p,\theta}(\Omega,T)&:=L_p\big((0,T); H^{\gamma}_{p,\theta}(\Omega)\big)\,\,,&\quad \bL_{p,\theta}(\Omega,T)&:=L_p\big((0,T);L_{p,\theta}(\bR^d)\big)\,\,,\\
			\Psi\bH^{\gamma}_{p,\theta}(\Omega,T)&:=L_p\big((0,T);\Psi H^{\gamma}_{p,\theta}(\Omega)\big)\,\,,&\quad \Psi\bL_{p,\theta}(\Omega,T)&:=L_p\big((0,T);\Psi L_{p,\theta}(\bR^d)\big)\,.
		\end{alignedat}
	\end{equation}

	Recall the definition of the Riemann-Liouville fractional integral, $I_t^\beta$, in Definition \ref{240907807}.
	
	\begin{defn}\label{240419512}
		\,
		
		\begin{enumerate}
			\item $\Psi\mathring{\cH}_{p,\theta}^{\alpha,\gamma+2}(\Omega,T)$ denotes the set of all $u\in \Psi\bH_{p,\theta}^{\gamma+2}(\Omega,T)$ such that there exists $f\in \Psi\bH_{p,\theta+2p}^{\gamma}(\Omega,T)$ satisfying the following: for any $\phi\in C_c^{\infty}(\Omega)$, and $\eta\in C_c^{\infty}\big([0,T)\big)$,
			\begin{align}\label{240419720}
				\int_0^T \Big(I_t^{1-\alpha}\la u(t),\phi\ra \Big)\eta'(t)\dd t=-\int_0^T\la f(t),\phi\ra \eta(t)\dd t\,,
			\end{align}
			where 
			$$
			I_t^{1-\alpha}\la u(t),\phi\ra:=\int_0^t\frac{(t-s)^{-\alpha}}{\Gamma(\alpha)}\la u(s),\phi\ra\dd s\,.
			$$
			For such $u\in \Psi\mathring{\cH}_{p,\theta}^{\alpha,\gamma+2}(\Omega,T)$, we denote $\partial_t^\alpha u:=f$ and define
			$$
			\|u\|_{\Psi\mathring{\cH}_{p,\theta}^{\alpha,\gamma+2}(\Omega,T)}:=\|u\|_{\Psi \bH_{p,\theta}^{\gamma+2}(\Omega,T)}+\|\partial_t^\alpha u\|_{\Psi \bH_{p,\theta+2p}^{\gamma}(\Omega,T)}\,.
			$$
			
			\item For $0<\frac{1}{p}<\alpha\leq 1$, $\Psi\cH_{p,\theta}^{\alpha,\gamma+2}(\Omega,T)$ denotes the set of all $u:(0,T]\rightarrow \cD'(\Omega)$ such that $u\in\Psi\bH_{p,\theta}^{\gamma+2}(\Omega,T)$, and there exist $u_0\in \Psi B_{p,\theta+2/\alpha}^{\gamma+2-2/(p\alpha)}(\Omega)$ and $f\in\Psi\bH_{p,\theta+2p}^{\gamma}(\Omega,T)$ satisy the following: for any $\phi\in C_c^{\infty}(\Omega)$,
			\begin{align}\label{240419514}
				\la u(t),\phi\ra=\la u_0,\phi\ra+I_t^{\alpha}\la f(t),\phi\ra\quad \text{for all}\,\,\,\,t\in(0,T]\,.
			\end{align}
			For such $u\in \Psi\cH_{p,\theta}^{\alpha,\gamma+2}(\Omega,T)$, we set $u(0):=u_0$, $\partial_t^\alpha u:=f$ and define
			$$
			\|u\|_{\Psi \cH_{p,\theta}^{\alpha,\gamma+2}(\Omega,T)}:=\|u\|_{\Psi \bH_{p,\theta}^{\gamma+2}(\Omega,T)}+\|u(0)\|_{\Psi B_{p,\theta+2/\alpha}^{\gamma+2-2/(\alpha p)}(\Omega)}+\|\partial_t^\alpha u\|_{\Psi \bH_{p,\theta+2p}^{\gamma}(\Omega,T)}\,.
			$$
		\end{enumerate}
	\end{defn}
	
	We similarly define $\mathring{\cH}_{p}^{\alpha,\gamma}(\bR^d,T)$ and $\cH_{p}^{\alpha,\gamma}(\bR^d,T)$ in the same manner as $\Psi \mathring{\cH}_{p,\theta}^{\alpha,\gamma}(\Omega,T)$ and $\Psi \cH_{p,\theta}^{\alpha,\gamma}(\Omega,T)$, respectively, by replacing the domain $\Omega$ with $\bR^d$ and the function spaces $\Psi \bH_{p,\theta}^{\gamma}(\Omega,T)$ and $\Psi B_{p,\theta}^\gamma(\Omega)$ with $\bH_{p}^{\gamma}(\bR^d,T)$ and $B_{p}^\gamma(\bR^d)$, respectively.
	
	\begin{remark}
		The space $\Psi \mathring{\cH}_{p,\theta}^{\alpha,\gamma+2}(\Omega,T)$ is designed for the problem with zero initial data in \eqref{2206201120}.
		The motivation for this definition comes from the identity that, for $F\in C^{1}\big([0,T]\big)$ with $F(0)=0$, we have
		$$
		\partial_t^{\alpha}F(t):=I_t^{1-\alpha}\partial_t F(t)=\partial_t I_{t}^{1-\alpha}F(t)\,.
		$$ 
		Consequently, for any $\eta\in C_c^{\infty}\big([0,T)\big)$,
		$$
		\int_0^T \Big(I_t^{1-\alpha}F(t) \Big)\eta'(t)\dd t=-\int_0^T\big(\partial_t^{\alpha}F(t)\big) \eta(t)\dd t\,.
		$$
		Notably, if $\alpha<\frac{1}{p}$, then the problem with nonzero initial data can be reduced to the zero initial data case.
		More precisely, let $u\in \Psi\bH_{p,\theta}^{\gamma+2}(\Omega,T)$, $u_0\in \Psi H_{p,\theta+2p}^{\gamma}(\Omega)$, and $f\in \Psi \bH_{p,\theta+2p}^{\gamma}(\Omega)$, and suppose that they satisfy \eqref{240419514}.
		Define
		$$
		\widetilde{f}(t,\cdot)=\frac{t^{-\alpha}}{\Gamma(1-\alpha)}u_0(\cdot)+f(t,\cdot)\,,
		$$ 
		Then $\widetilde{f}\in \Psi\bH_{p,\theta+2p}^{\gamma}(\Omega,T)$, and $\la u(t),\phi\ra=I_t^{\alpha}\la \widetilde{f}(t),\phi\ra$ for all $t\in(0,T]$.
		Consequently, $u\in \Psi \mathring{\cH}_{p,\theta}^{\gamma+2}(\Omega,T)$ (see Lemma \ref{240426524}.(1)).
		
		The space $\Psi \cH_{p,\theta}^{\alpha,\gamma+2}(\Omega,T)$ is intended for the problem with nonzero initial data in \eqref{2206201120} when $\alpha>\frac{1}{p}$.
		For $\alpha=\frac{1}{p}$, \cite{DONG2021107494} provides a result for $u_0$ in $\bR^d$, under additional assumptions on the regularity of $u_0$.
		We refer the reader to \cite{KimWoo2023} for a detailed discussion of $ \cH_{p,\theta}^{\alpha,\gamma+2}(\bR^d,T)$ for general $\alpha\in(0,1)$.
	\end{remark}
	
	\begin{remark}\label{2404201227}
		\,
		
		\begin{enumerate}
			\item 
			All quantities in \eqref{240419720} are well-defined. 
			More precisely, $I_t^{1-\alpha}|\la u(t),\phi\ra|$ and $\la f(t),\phi\ra$ lie in $L_{p}([0,T])$.
			Indeed,
			\begin{align}
				\begin{aligned}\label{240426508}
					\bigg(\int_0^{T} \Big(I_t^{1-\alpha}|\la u(t),\phi\ra| \Big)^p\dd t\bigg)^{1/p}
					&\leq \int_0^{T} \frac{s^{-\alpha}}{\Gamma(1-\alpha)}\bigg(\int_s^{T}|\la u(t-s),\phi\ra|^p \dd t\bigg)^{1/p}\dd s\\
					&\lesssim  \frac{T^{1-\alpha}}{\Gamma(2-\alpha)}\|u\|_{\Psi\bH_{p,\theta}^{\gamma+2}(\Omega,T)}\|\phi\|_{\Psi^{-1}H_{p',\theta'}^{-\gamma-2}(\Omega)}\,,\\
					\Big(\int_0^{T}|\la f(t),\phi\ra|^p\dd t\Big)^{1/p}\,&\lesssim \|f\|_{\Psi \bH_{p,\theta+2p}^{\gamma}(\Omega,T)}\|\phi\|_{\Psi^{-1}H_{p',\theta'-2p'}^{-\gamma}(\Omega)}\,,
				\end{aligned}
			\end{align}
			where $p'$ and $\theta'$ are the constants in \eqref{241002631} (the first inequality follows from Minkowski’s inequality, whereas the latter two follow from \eqref{220526214}).
			
			Since $C_c^{\infty}(\Omega)$ is a separable topological vector space, identity \eqref{240419720} implies the uniqueness of $\partial_t^\alpha u:=f\in\Psi \bH_{p,\theta+2p}^\gamma(\Omega,T)$ for $u\in \Psi\mathring{\cH}_{p,\theta}^{\alpha,\gamma+2}(\Omega,T)$.

			\item The term $I_t^\alpha \la f(t),\phi\ra$ in \eqref{240419514} is well-defined for all $t\in(0,T]$, as can be seen from the following estimate:
			\begin{align}\label{240428116}
				\begin{split}
					I_t^{\alpha}\big|\la f(t),\phi\ra \big|\,&\lesssim_\alpha  \Big(\int_0^t(t-s)^{-(1-\alpha)p'}\dd s\Big)^{1/p'} \Big(\int_0^t\big|\la f(s),\phi\ra \big|^p\dd s\Big)^{1/p}\\
					&\lesssim t^{\alpha-1/p}\|f\|_{\Psi \bH_{p,\theta+2p}^{\gamma}(\Omega,T)}\|\phi\|_{\Psi^{-1}H_{p',\theta'-2p'}^{-\gamma}(\Omega)}
				\end{split}
			\end{align}
			where $p'=\frac{p}{p-1}$ (the second inequality follows from \eqref{220526214} and that $(1-\alpha)p'<1$ because $\alpha>\frac{1}{p}$).
			
			\eqref{240428116} and \eqref{220526214} also imply that
			\begin{align}\label{240911422}
				\sup_{0<t\leq T}\frac{|\la u(t)-u_0,\phi\ra|}{t^{\alpha-1/p}}\lesssim \sup_{0<t\leq T}\frac{I_t^{\alpha}|\la f(t),\phi\ra|}{t^{\alpha-1/p}}\lesssim \|f\|_{\bH_p^\gamma(\Omega,T)}\|\phi\|_{H_{p',\theta'}^{-\gamma}(\Omega)}\,.
			\end{align}
			This shows that $u(0):=u_0$ satisfying \eqref{240419514} is uniquely determined as $\lim_{t\searrow 0}u(t)$.
			
			For the uniqueness of $f$ in \eqref{240419514}, observe that if $I_t^\alpha\la f(t),\phi\ra=0$ for all $t\in(0,T]$, then $I_t^1\la f(t),\phi\ra=0$ for all $t\in(0,T]$.
			Hence $\la f(t),\phi\ra=0$ for almost every $t\in(0,T]$, and therefore $f=0$ in $\Psi \bH_{p,\theta+2p}^\gamma(\Omega,T)$.
		\end{enumerate}
	\end{remark}

	The following lemma shows a relation between \eqref{240419720} and \eqref{240419514}.
	In particular, if $\alpha>\frac{1}{p}$, then
	$$
	\Psi \mathring{\cH}_{p,\theta}^{\alpha,\gamma+2}(\Omega,T)=\Big\{u\in\Psi \cH_{p,\theta}^{\alpha,\gamma+2}(\Omega,T)\,:\,u(0)=0\Big\}\,.
	$$
	
	\begin{lemma}\label{240426524}\,
		
		\begin{enumerate}
			\item Let $\phi\in C_c^{\infty}(\Omega)$, and assume that $u(t),u_0\in\cD(\Omega)$ for $t\in(0,T]$, and that $\langle f(\cdot),\phi\rangle\in L_1((0,T))$.
			If equation \eqref{240419514} holds for almost every $t\in(0,T]$, then for any $\eta\in C_c^{\infty}([0,T))$,
			\begin{align}\label{240420138}
				\int_0^T \Big(I_t^{1-\alpha}\langle u(t)-u_0,\phi\rangle \Big)\eta'(t)\dd t
				= -\int_0^T\langle f(t),\phi\rangle \eta(t)\dd t\,.
			\end{align}
			
			\item Let  $0<\frac{1}{p}<\alpha\leq 1$. 
			Suppose that $u\in\Psi\bH_{p,\theta}^{\alpha,\gamma+2}(\Omega)$, $u_0\in \Psi B_{p,\theta+2}^{\gamma+2-2/(\alpha p)}(\Omega)$, and $f\in \Psi \bH_{p,\theta+2p}^\gamma(\Omega,T)$ satisfy \eqref{240420138} for all $\phi\in C_c^{\infty}(\Omega)$ and $\eta\in C_c^{\infty}\big([0,T)\big)$.
			Then there exists a unique $\overline{u}\in \Psi \cH_{p,\theta}^{\alpha,\gamma+2}(\Omega,T)$ such that $\overline{u}=u$ almost everywhere on $(0,T]$, $\overline{u}(0)=u_0$, and $\partial_t^\alpha \overline{u}=f$.
		\end{enumerate}
	\end{lemma}
	
	Note that in Lemma \ref{240426524}.(1), $\langle f(\cdot),\phi\rangle\in L_1((0,T))$ implies that 
	$$
	I_t^{1-\alpha} I_t^\alpha|\langle f(t),\phi\rangle| = I_t^1 |\langle f(t),\phi\rangle|<\infty\,.
	$$
	Hence $I_t^\alpha\langle f(t),\phi\rangle$ is well-defined for almost every $t$.
	
	\begin{proof}[Proof of Lemma \ref{240426524}]
		(1) Equation \eqref{240420138} follows directly from \eqref{240419514} and \eqref{240420130} by straightforward calculation.
		
		(2) Define $\overline{u}(t):=u_0+I_t^{\alpha}f(t)$ in the sense of distributions, \textit{i.e.},
		\begin{align}\label{240419724}
			\la \overline{u}(t),\phi\ra :=\la u_0,\phi\ra+\int_0^t\frac{(t-s)^{-1+\alpha}}{\Gamma(\alpha)} \la f(s),\phi\ra\dd s
		\end{align}
		By Remark \ref{2404201227}.(2), $\overline{u}(t)$ is well-defined in $\cD'(\Omega)$ for any $t\in(0,T]$.
		
		We first claim that $\la u(t),\phi\ra=\la \overline{u}(t),\phi\ra$ for almost every $t\in(0,T]$.
		For arbitrary fixed $\eta\in C_c^{\infty}\big([0,T)\big)$, put
		\begin{align*}
			\widetilde{\eta}(t):=\int_t^T\frac{(r-t)^{-1+\alpha}}{\Gamma(\alpha)}\eta(r)\dd r\,.
		\end{align*}
		Then $\widetilde{\eta}\in C_c^{\infty}([0,T))$, and
		$\partial_t^k\widetilde{\eta}(t)=\int_t^T\frac{(r-t)^{-1+\alpha}}{\Gamma(\alpha)}(\partial_t^k\eta)(r)\dd r$.
		It follows from \eqref{240420130} that
		\begin{align*}
			\int_0^T \Big(I_t^{1-\alpha}\la u(t)-u_0,\phi\ra \Big)\widetilde{\eta}\,'(t)\dd t=-\int_0^T\la u(r)-u_0,\phi\ra \eta(r)\dd r\,.
		\end{align*}
		On the other hand, \eqref{240420138} and \eqref{240419724} imply that
		\begin{align*}
			\int_0^T \Big(I_t^{1-\alpha}\la u(t)-u_0,\phi\ra \Big)\widetilde{\eta}\,'(t)\dd t\,&=-\int_0^t\la f(t),\phi\ra \widetilde{\eta}(t)\dd t\\
			&=-\int_0^T\la \overline{u}(t)-u_0,\phi\ra \eta(t)\dd t\,.
		\end{align*}
		Therefore we have
		\begin{align*}
		\int_0^T\la u(t)-u_0,\phi\ra \eta(t)\dd t\,&=-\int_0^T \Big(I_t^{1-\alpha}\la u(t)-u_0,\phi\ra \Big)\widetilde{\eta}\,'(t)\dd t\\
		&=\int_0^T\la \overline{u}(t)-u_0,\phi\ra \eta(t)\dd t\,.
		\end{align*}
		Since this equality holds for all $\eta\in C_c^{\infty}\big([0,T)\big)$, we conclude that $\la u(t),\phi\ra=\la \overline{u}(t),\phi\ra$ for almost every $t\in(0,T]$.
		
		It follows from the separability of $C_c^{\infty}(\Omega)$ that $u(t)=\overline{u}(t)$ in $\cD'(\Omega)$ for almost every $t\in(0,T]$, which implies that $\overline{u}\in\Psi\bH_{p,\theta}^{\gamma+2}(\Omega)$.
		By \eqref{240419724}, $\overline{u}\in \Psi\cH_{p,\theta}^{\alpha,\gamma+2}(\Omega,T)$ with $\overline{u}(0)=u_0$ and $\partial_t^\alpha\overline{u}=f$.
		
		To prove the uniqueness of $\overline{u}$, consider the case that $\overline{u}\in \Psi\cH_{p,\theta}^{\alpha,\gamma+2}(\Omega,T)$ with $\overline{u}(0)= 0$ and $\partial_t^\alpha \overline{u}=0$.
		Then, by \eqref{240911422}, $\overline{u}(t)=0$ for all $t\in (0,T]$.
		Therefore the proof is complete.
	\end{proof}
	
	\begin{lemma}\label{240911115} Let $0<\alpha\leq 1$, $1<p<\infty$, and $0<T<\infty$.
		\begin{enumerate}
			\item $\Psi\mathring{\cH}_{p,\theta}^{\alpha,\gamma}(\Omega,T)$ is a Banach space, and $C_c^{\infty}\big((0,\infty)\times \Omega\big)$ is dense in $\Psi \mathring{\cH}_{p,\theta}^{\alpha,\gamma+2}(\Omega,T)$.
			
			\item Suppose that $\alpha>\frac{1}{p}$. Then 
			$\Psi\cH_{p,\theta}^{\alpha,\gamma}(\Omega,T)$ is a Banach space, and $C_c^{\infty}\big([0,\infty)\times \Omega\big)$ is dense in $\Psi \cH_{p,\theta}^{\alpha,\gamma+2}(\Omega,T)$.
		\end{enumerate}			
	\end{lemma}
	
	\begin{proof}
		It follows from \eqref{240426508}–\eqref{240911422} that both $\Psi\mathring{\cH}_{p,\theta}^{\gamma}(\Omega,T)$ and $\Psi\cH_{p,\theta}^{\gamma}(\Omega,T)$ are Banach spaces.
		Therefore, it remains to prove that $C_c^{\infty}\big((0,\infty)\times\Omega\big)$ and $C_c^{\infty}\big([0,\infty)\times\Omega\big)$ are dense in $\Psi\mathring{\cH}_{p,\theta}^{\alpha,\gamma+2}(\Omega,T)$ and $\Psi\cH_{p,\theta}^{\alpha,\gamma+2}(\Omega,T)$, respectively.
		
		(1) Let $\delta>0$ be given.
		By Lemma \ref{22.04.11.3}, there exists $i_0,j_0,k_0\in\bN$ such that
		\begin{align}\label{2404273131}
			\|u-\Lambda_{i_0,j_0,k_0}u\|_{\Psi \mathring{\cH}_{p,\theta}^{\alpha,\gamma+2}(\Omega,T)}<\delta/2\,,
		\end{align}
		where $\Lambda_{i_0,j_0,k_0}$ is the operator defined in Lemma \ref{22.04.11.3}.
		Set $v:=\Lambda_{i_0,j_0,k_0}u$ and $g:=\partial_t^\alpha v=\Lambda_{i_0,j_0,k_0}(\partial_t^\alpha u)$.
		Observe that $v\in \Psi \mathring{\cH}_{p,\theta}^{\alpha,\gamma+2}(\Omega,T)$.
		For convenience, for $t<0$, we put $v(t)=0$.
		Take $\eta_0\in C_c^{\infty}(\bR)$ such that $\mathrm{supp}(\eta_0)\subset [1/2,1)$ and $\int_\bR \eta_0(t)\dd t=1$.
		For $\epsilon\in(0,1)$, we define
		$$
		v^{(\epsilon)}(t):=\int_0^1v(t-s)\eta_0(\epsilon s)\dd s\quad \text{and}\quad g^{(\epsilon)}(t):=\int_0^1g(t-s)\eta_0(\epsilon s)\dd s\,.
		$$
		Then $v^{(\epsilon)}(t)=0$ for all $t\in (0,\epsilon/2)$. 
		It is easy to see that  $v^{(\epsilon)}\in\mathring{\cH}_{p,\theta}^{\alpha,\gamma+2}(\Omega,T)$ with $\partial_t^\alpha\big(v^{(\epsilon)}\big)=g^{(\epsilon)}$.
		Since $v\in\bH_{p,\theta}^{\gamma+2}(\Omega,T)$, the definition of $v^{(\epsilon)}$ implies that $v^{(\epsilon)}\rightarrow v$ as $\epsilon\rightarrow 0$ in $\bH_{p,\theta}^{\gamma+2}(\Omega,T)$.
		In the same manner, $g^{(\epsilon)}\rightarrow g$ in $\bH_{p,\theta+2p}^{\gamma}(\Omega)$.
		Therefore, we conclude that there exists $\epsilon>0$ such that $\|v-v^{(\epsilon)}\|_{\mathring{\cH}_{p,\theta}^{\alpha,\gamma+2}(\Omega,T)}<\delta/2$.
		By combining this with \eqref{2404273131}, we have $\|u-v^{(\epsilon)}\|_{\mathring{\cH}_{p,\theta}^{\alpha,\gamma+2}(\Omega,T)}<\delta$.
		It follows from the definition of $v^{(\epsilon)}$ that $v^{(\epsilon)}$ lie in $C_c^{\infty}\big((0,T]\big)$.
		Therefore the proof is complete.
		
		(2) The proof of this assertion is almost the same with the proof of (1) of this lemma (for $t<0$, put $v(t)= v(0)$ instead of $v(t)\equiv 0$).
		We leave the details to the reader.
	\end{proof}

	\begin{lemma}\label{2205241011}
		Let $\Psi'$ be a regular Harnack function, $p'\in(1,\infty)$, and $\gamma',\,\theta'\in\bR$.
		If $f\in \Psi \bH_{p,\theta}^{\gamma}(\Omega,T)\cap \Psi'\bH_{p',\theta'}^{\gamma'}(\Omega,T)$, then for any $\epsilon>0$, there exist $g\in C_c^{\infty}((0,T)\times\Omega)$ such that
		$$
		\|f-g\|_{\Psi \bH_{p,\theta}^{\gamma}(\Omega,T)}+\|f-g\|_{\Psi'\bH_{p',\theta'}^{\gamma'}(\Omega,T)}<\epsilon\,.
		$$
	\end{lemma}
	
	\begin{proof}
		Since $\Psi H_{p,\theta}^{\gamma}(\Omega)$ and $\Psi'H_{p',\theta'}^{\gamma'}(\Omega)$ are separable Banach spaces, by a standard mollification argument, it suffices to consider the case where $f$ is a $\Psi H_{p,\theta}^{\gamma}(\Omega)\cap \Psi' H_{p',\theta'}^{\gamma'}(\Omega)$-valued continuous function in $[0,T]$. 
		In this case, for any $\epsilon>0$, there exists a sufficiently large $N\in\bN$ and functions $\eta_1,\,\ldots,\,\eta_N\in C_c^{\infty}\big((0,T)\big)$ such that
		$$
		\|\,f-\widetilde{f}\,\|_{\Psi \bH_{p,\theta}^\gamma(\Omega,T)}+\|\,f-\widetilde{f}\,\|_{\Psi' \bH_{p',\theta'}^{\gamma'}(\Omega,T)}<\epsilon\,,
		$$
		where 
		$$
		\widetilde{f}(t,\cdot)=\sum_{k=1}^N\eta_k(t)f\big(kT/N,\,\cdot\,\big)\,.
		$$
		Since $C_c^{\infty}((0,T)\times \Omega)$ is dense in $\Psi \bH_{p,\theta}^{\gamma}(\Omega,T)\cap \Psi' \bH_{p',\theta'}^{\gamma'}(\Omega,T)$ (see Lemma \ref{22.04.11.3}), this completes the proof.
	\end{proof}
	
	We conclude this subsection with the following parabolic embedding theorem for $\Psi\cH_{p,\theta}^{\gamma+2}(\Omega)$.
	
	\begin{prop}\label{2204160313}
		Let $\alpha>\frac{1}{p}$ and $\frac{1}{p}\leq \beta< \alpha$.
		Then, for any $u\in\Psi\cH^{\alpha,\gamma+2}_{p,\theta}(\Omega,T)$ and $0\leq s< t\leq T$,
		\begin{align}\label{24082783611}
			\|u(t)-u(s)\|_{\Psi B_{p,\theta+2p\beta/\alpha}^{\gamma+2-2\beta/\alpha}(\Omega)}\leq  N |t-s|^{\beta-1/p}\|u\|_{\Psi \cH_{p,\theta}^{\alpha,\gamma+2}(\Omega,T)}\,,
		\end{align}
		where $N=N(d,p,\gamma,\theta,\alpha,\beta)$.
		In addition,
		\begin{align}\label{240827836111111}
			\|u(t)-u(s)\|_{\Psi H_{p,\theta+2p}^{\gamma}(\Omega)}\leq N |t-s|^{\alpha-1/p}\|u\|_{\Psi \cH_{p,\theta}^{\alpha,\gamma+2}(\Omega,T)}\,,
		\end{align}
		where $N=N(d,p,\gamma,\theta,\alpha)$.
	\end{prop}
	
	\begin{proof}
		For any $\gamma',\theta'\in\bR$, the map $f\mapsto \Psi^{-1}f$ is an isometric isomorphism from $\Psi H_{p,\theta'}^{\gamma'}(\Omega)$ (resp. $\Psi B_{p,\theta'}^{\gamma'}(\Omega)$, $\Psi \cH_{p,\theta'}^{\gamma'}(\Omega,T)$) to $H_{p,\theta'}^{\gamma'}(\Omega)$ (resp. $B_{p,\theta'}^{\gamma'}(\Omega)$, $\cH_{p,\theta'}^{\gamma'}(\Omega,T)$).
		Therefore, it suffices to prove the proposition for the case $\Psi\equiv1$.
		The proof of this case is almost the same as that of \cite[Theorem 7.1]{Krylov2001}.
		Since $u\in\cH_{p,\theta}^{\gamma+2}(\Omega,T)$, the function
		$$
		u_n(t,x):=u(t,\ee^nx)\zeta_{0,(n)}(\ee^nx)\in \cH_{p}^{\gamma+2}(\bR^d,T)
		$$
		satisfies 
		$$
		\big(\partial_t^\alpha u_n\big)(s,\cdot)=\partial_t^\alpha u(s,\ee^n\,\cdot\,)\zeta_{0,(n)}(\ee^n\,\cdot)
		$$
		in the sense of distributions on $\bR^d$.
		Since $u_n\in \cH_p^{\gamma+2}(\bR^d,T)$, by \eqref{2408261012} with $A:=\ee^{2n/\alpha}$, we obtain, for $\beta\in[1/p,1)$,
		\begin{align*}
			\begin{split}
				&\ee^{2np\beta/\alpha}\|u_n(t)-u_n(s)\|_{B_p^{\gamma+2-2\beta/\alpha}(\bR^d)}^p\\
				\leq N &|t-s|^{p \beta -1}\big(\|u_n\|_{\bH_p^{\gamma+2}(\bR^d,T)}^p+\ee^{2np}\|\partial_t^\alpha u\|_{\bH_p^\gamma(\bR^d,T)}^p+\ee^{2/\alpha}\|u_0\|_{B_p^{\gamma+2-2/(p\alpha)}(\bR^d)}^p\big)\,,
			\end{split}
		\end{align*}
		where $N=N(d,p,\gamma,\beta)$.
		This yields that 
		\begin{alignat*}{3}
			&&&\|u(t)-u(s)\|_{B_{p,\theta+2p\beta/\alpha}^{\gamma+2-2\beta/\alpha}(\Omega)}^p\\
			&=\,&&\sum_{n\in\bZ}\ee^{n(\theta+2p\beta)}\|u_n(t)-u_n(s)\|_{B_{p}^{\gamma+2-2\beta/\alpha}(\bR^d)}^p\\
			&\lesssim&&|t-s|^{\beta p-1}\sum_{n\in\bZ}\ee^{n\theta}\Big(\|u_n\|_{\bH_p^{\gamma+2}(\bR^d,T)}^p+\ee^{2np}\|\partial_t^\alpha u\|_{\bH_p^\gamma(\bR^d,T)}^p+\ee^{2n/\alpha}\|u_0\|_{B_p^{\gamma+2-2/(p\alpha)}(\bR^d)}^p\Big)\\
			&=&&|t-s|^{\beta p-1}\Big(\|u\|_{\bH_{p,\theta}^{\gamma+2}(\Omega,T)}^p+\|\partial_t^\alpha u\|_{\bH_{p,\theta+2p}^\gamma(\Omega,T)}^p+\|u_0\|_{B_{p,\theta+2/\alpha}^{\gamma+2-2/(p\alpha)}(\Omega)}^p\Big)\,,
		\end{alignat*}
		which means \eqref{24082783611}.
		Estimate \eqref{240827836111111} follows in the same way, using \eqref{2408261013} in place of \eqref{2408261012}.
	\end{proof}

	\subsection{Solvability of parabolic equations}\label{0053}
	We now state the main result of this paper.
	Let $0<\alpha\le1$ and $0<T<\infty$, and consider the equation
	\begin{align}\label{220616124}
		\partial_t^\alpha u=\cL u + f \quad\text{in}\,\,\,(0,T],
	\end{align}
	where $\cL\in\cM_T(\nu_1,\nu_2)$, $0<\nu_1\leq \nu_2<\infty$ (see Section \ref{0030} for the definition of $\cM_T(\nu_1,\nu_2)$).
	
	\begin{thm}\label{22.02.18.6}
		Let $\Omega$ be an open set admitting the Hardy inequality \eqref{hardy} and let $\psi$ be a superharmonic Harnack function in $\Omega$ with its regularization $\Psi$.
		Then for any $p\in(1,\infty)$, $\mu\in \cI(\psi,p,\nu_1/\nu_2)$, and $\gamma\in\bR$, the following statements hold:
		
		\begin{enumerate}
			\item 
			For any $f\in \Psi^{\mu} \bH^{\gamma}_{p,d+2p-2}(\Omega,T)$, equation \eqref{220616124}
			has a unique solution $u$ in $\Psi^{\mu}\mathring{\cH}^{\alpha,\gamma+2}_{p,d-2}(\Omega,T)$.
			Moreover, we have
			\begin{align}\label{240824444}
				\|u\|_{\Psi^{\mu}\mathring{\cH}^{\alpha,\gamma+2}_{p,d-2}(\Omega,T)}\leq N\| f\|_{\Psi^{\mu}\bH^{\gamma}_{p,d+2p-2}(\Omega,T)}\,,
			\end{align}
			where $N=N(d,\alpha,p,\gamma,\mu,\mathrm{C}_0(\Omega),\mathrm{C}_2(\Psi),\mathrm{C}_3(\psi,\Psi),\mathrm{C}_4)$.
			
			\item If $\frac{1}{p}<\alpha$, then for any $u_0\in \Psi^{\mu} B^{\gamma+2-2/(p\alpha)}_{p,d+2/\alpha-2}(\Omega)$ and $f\in \Psi^{\mu} \bH^{\gamma}_{p,d+2p-2}(\Omega,T)$, equation \eqref{220616124} with $u(0,\cdot)=u_0$
			has a unique solution $u$ in $\Psi^{\mu}\cH^{\alpha,\gamma+2}_{p,d-2}(\Omega,T)$.
			Moreover, we have
			\begin{align}\label{240824443}
				\|u\|_{\Psi^{\mu}\cH^{\alpha,\gamma+2}_{p,d-2}(\Omega,T)}\leq N\left(\| u_0\|_{\Psi^{\mu} B^{\gamma+2-2/(p\alpha)}_{p,d+2/\alpha-2}(\Omega,T)}+\| f\|_{\Psi^{\mu}\bH^{\gamma}_{p,d+2p-2}(\Omega,T)}\right)\,,
			\end{align}
			where $N=N(d,\alpha,p,\gamma,\mu,\mathrm{C}_0(\Omega),\mathrm{C}_2(\Psi),\mathrm{C}_3(\psi,\Psi), \mathrm{C}_4)$.
			
		\end{enumerate}

	\end{thm}
	Recall that $\mathrm{C}_0(\Omega)$ denotes the constant in \eqref{hardy}, $\mathrm{C}_2(\Psi)$ and $\mathrm{C}_3(\psi,\Psi)$ denote the constants in Definition \ref{21.10.14.1}, and $\mathrm{C}_4$ denotes the constant in Definition \ref{21.11.10.1}.
	
	\begin{remark}\label{230212657}
		Theorem \ref{22.02.18.6} can be reformulated without explicitly including $\Psi$, by setting $\Psi=\widetilde{\psi}$, where $\widetilde{\psi}$ is the function given in Lemma \ref{21.05.27.3}.(1).
		In addition, when $\gamma\in\bN_0$, the equivalent norms of $\Psi \bH_{p,\theta}^{\gamma}(\Omega,T)$ and $\Psi B_{p,\theta}^{\gamma+s}(\Omega)$, $s\in(0,1)$, are given in \eqref{2409098051} and Proposition \ref{220418435}, respectively.
		For equivalent norms in the case where $-\gamma\in \bN$, use \eqref{241111154}.
		For $p=2$ and $\alpha=1$, one may also use the identity $\Psi^\mu B_{2,d}^{\gamma+1}(\Omega)=\Psi^\mu H_{2,d}^{\gamma+1}(\Omega)$ (see \eqref{2411011106}).
	\end{remark}
	
	To prove Theorem \ref{22.02.18.6}, we need the help of the following three lemmas:
	\begin{lemma}\label{22.04.14.1}
		Suppose that $u\in \mathring{\cH}^{\alpha,\gamma+1}_p(\bR^d,T)$ satisfies  $f:=\partial_t^\alpha u-\cL u \in\bH^{\gamma}_p(\bR^d,T)$.
		Then $u\in\mathring{\cH}^{\alpha,\gamma+2}_p(\bR^d,T)$, and
		\begin{align}\label{241013401}
			\|u\|_{\mathring{\cH}^{\alpha,\gamma+2}_{p}(\bR^d,T)}\leq N\left(\|u\|_{\bH^{\gamma+1}_{p}(\bR^d,T)}+\| f\|_{\bH^{\gamma}_{p}(\bR^d,T)}\right)
		\end{align}
		where $N=N(d,p,\alpha,\nu_1,\nu_2)$.
		In particular, $N$ is independent of $T$.
	\end{lemma}
	\begin{proof}
		By considering $v:=(1-\Delta)^{\gamma/2}u$ and $g:=(1-\Delta)^{\gamma/2}f$, it suffices to prove the estimate for $\gamma=0$.
		Let $\gamma=0$.
		Considering the equation $\partial_t^\alpha u=\cL u-u +(u+f)$,
		\eqref{241013401} follows from \cite[Theorem 1.2]{Krylov2001-1} for $\alpha=1$, and from \cite[Theorem 2.1]{DONG2019289} for $0<\alpha<1$.
	\end{proof}
	
	\begin{lemma}[Higher order estimates]\label{21.05.13.9}\,
		Let $u\in\Psi \bH^{s+2}_{p,\theta}(\Omega,T)$ for some $s\in\bR$, and let $f\in \Psi\bH^{\gamma}_{p,\theta+2p}(\Omega,T)$.
		Suppose that \eqref{240419720} holds, where $f$ in \eqref{240419720} is replaced by $\cL u+f$.
		Then $u$ belongs to $\Psi\mathring{\cH}^{\alpha,\gamma}_{p,\theta}(\Omega,T)$ with $\partial_t^\alpha u=\cL u+f$. 
		Moreover, we have
		\begin{align}\label{2409111129}
			\begin{split}
				\|u\|_{\Psi \mathring{\cH}^{\alpha, \gamma+2}_{p,\theta}(\Omega,T)}\leq N\left(\|u\|_{\Psi \bH^{s}_{p,\theta}(\Omega,T)}+\| f\|_{\Psi \bH^{\gamma}_{p,\theta+2p}(\Omega,T)}\right)\,,
			\end{split}
		\end{align}
		where $N=N(d,p,\gamma,\theta,\nu_1,\nu_2,\mathrm{C}_2(\Psi),s)$.
	\end{lemma}
	
	\begin{proof}
		\textbf{Step 1.} First, we consider the case $s\geq \gamma+1$.
		Since 
		\begin{align*}
			\|u\|_{\Psi\bH_{p,\theta}^{\gamma+1}(\Omega,T)}\lesssim_{d,p,s,\gamma} \| u\|_{\Psi\bH_{p,\theta}^{s}(\Omega,T)}
		\end{align*}
		(see \eqref{2411011107}), it suffices to prove the case $s=\gamma+1$.
		Recall the relation between $\zeta_0$ and the space $\Psi H_{p,\theta}^\gamma$ given in Definition \ref{220610533}. 
		Put
		\begin{align*}
			v_n(t,x)=\zeta_0\big(\ee^{-n}\trho(\ee^nx)\big)\Phi(\ee^nx) u(\ee^{2n/\alpha}t,\ee^nx)\,,
		\end{align*}
		where $\Phi=\Psi^{-1}$, so that $\mathrm{C}_2(\Phi)$ depends only on $d$ and $\mathrm{C}_2(\Psi)$.
		Note that
		\begin{align}\label{240419259}
			&\sum_{n\in\bZ}\ee^{n(\theta+2/\alpha)}\left\|v_n\right\|_{\bH_p^{\gamma+1}(\bR^d,\ee^{-2n/\alpha}T)}^p\nonumber\\
			=\,&\sum_{n\in\bZ}\ee^{n\theta}\int_0^{T}\left\|\zeta_0\big(\ee^{-n}\trho(\ee^n\cdot)\big)\Phi(\ee^n\cdot) u(t,\ee^n\cdot)\right\|_{H_p^{\gamma+1}(\bR^d)}^p\dd t\\
			=\,&\|\Phi u\|_{\bH_{p,\theta}^{\gamma+1}(\Omega,T)}^p=\| u\|_{\Psi\bH_{p,\theta}^{\gamma+1}(\Omega,T)}^p<\infty\,.\nonumber
		\end{align}
		This implies that $v_n\in \bH_p^{\gamma+1}(\bR^d,\ee^{-2n/\alpha}T)$. 
		In the same manner as in \eqref{240419259}, $\partial_t^\alpha v_n$ is well-defined in the sense of \eqref{240419720} and belongs to $\bH_p^{\gamma-1}(\bR^d,\ee^{-2n/\alpha}T)$.
		Therefore we have $v_n\in \mathring{\cH}_p^{\alpha,\gamma+1}(\bR^d,\ee^{-2n/\alpha}T)$.
		Moreover,
		\begin{align*}
			\begin{split}
				\partial_t^\alpha v_n=\cL_n v_n+\widetilde{f}_n\,,
			\end{split}
		\end{align*}
		where $\cL_nv(t,\cdot):=\sum_{i,j=1}^da^{ij}(\ee^{2n/\alpha}t)D_{ij}v(t,\cdot)$, and
		\begin{align*}
			\widetilde{f}_n(t,x):=&\,\big(\zeta_{0,(n)}\Phi\, f\big)(\ee^{2n/\alpha}t,\ee^nx)+\ee^{2n}\Big[\big(\zeta_{0,(n)}\Phi\, \cL u\big)-\cL\big(\zeta_{0,(n)}\Phi u\big)\Big](\ee^{2n/\alpha}t,\ee^nx)\\
			=&\,\big(\zeta_{0,(n)}\Phi\, f\big)(\ee^{2n/\alpha}t,\ee^nx)-\sum_{i,j=1}^da^{ij}(\ee^{2n/\alpha}t)F_{n,ij}(\ee^{2n/\alpha}t,\ee^nx)\,;\\
			F_{n,ij}:=&\,\ee^{2n}\zeta_{0,(n)}\Big(2D_j\Phi\cdot D_iu+D_{ij}\Phi\cdot u\Big)\\
			&\quad +\ee^n\big(\zeta_0'\big)_{(n)}\Big(2D_i\trho\cdot D_j(\Phi u)+D_{ij}\trho\cdot \Phi u\Big)	+\big(\zeta_0''\big)_{(n)}D_i\trho\cdot D_j\trho\cdot \Phi u\,.
		\end{align*}
		Using \eqref{2411011115} and \eqref{220526558}, we obtain 
		\begin{align}\label{22.04.15.1}
			&\sum_{n\in\bZ}\ee^{n(\theta+2/\alpha)}\big\|\widetilde{f}_n\big\|_{\bH^{\gamma}_p(\bR^d,\ee^{-2n/\alpha}T)}^p\nonumber\\
			\lesssim_N &\left\|\Phi f\right\|^p_{\bH_{p,\theta+2p}^{\gamma}(\Omega,T)}+\left\|\Phi_x u_x\right\|^p_{\bH_{p,\theta+2p}^{\gamma}(\Omega,T)}+\left\|\Phi_{xx} u\right\|^p_{\bH_{p,\theta+2p}^{\gamma}(\Omega,T)}
			\\
			&+\left\|\trho_x(\Phi u)_x\right\|^p_{\bH_{p,\theta+p}^{\gamma}(\Omega,T)}+\left\|\trho_{xx}\Phi u\right\|^p_{\bH_{p,\theta+p}^{\gamma}(\Omega,T)}+\left\|\trho_x\trho_x\Phi u\right\|^p_{\bH_{p,\theta}^{\gamma}(\Omega,T)}\nonumber
			\\
			\lesssim_N & \|\Phi f\|^p_{\bH_{p,\theta+2p}^{\gamma}(\Omega)}
			+\|\Phi u\|^p_{\bH^{\gamma+1}_{p,\theta}(\Omega)}=\|f\|^p_{\Psi\bH_{p,\theta+2p}^{\gamma}(\Omega)}
			+\|u\|^p_{\Psi\bH^{\gamma+1}_{p,\theta}(\Omega)}<\infty\,,\nonumber
		\end{align}
		where $N=N(d,p,\gamma,\theta,\mathrm{C}_2(\Psi))$.
		Hence $\tilde{f}_n\in\bH^{\gamma}_p(\bR^d,\ee^{-2n/\alpha}T)$ for all $n\in\bZ$.
		
		By Lemma \ref{22.04.14.1}, we have
		\begin{alignat}{2}\label{220506140}
			\|v_n\|_{\bH^{\gamma+2}_{p}(\bR^d,\ee^{-2n/\alpha}T)}\lesssim_N\|v_n\|_{\bH^{\gamma+1}_{p}(\bR^d,\ee^{-2n/\alpha}T)}+\| \widetilde{f}_n\|_{\bH^{\gamma}_{p}(\bR^d,\ee^{-2n/\alpha}T)}\,,
		\end{alignat}
		where $N=N(d,p,\gamma,\theta,\mathrm{C}_2(\Psi),\nu_1,\nu_2)$.
		Combining \eqref{240419259}, \eqref{220506140},  and \eqref{22.04.15.1}, we obtain
		\begin{alignat}{2}\label{241106159}
			\|u\|_{\Phi\bH^{\gamma+2}_{p,\theta}(\Omega,T)}^p&=&& \sum_{n\in\bZ}\ee^{n(\theta+2)}\|v_n\|_{\bH^{\gamma+2}_p(\bR^d,\ee^{-2n/\alpha}T)}^p\nonumber\\
			&\lesssim&& \sum_{n\in\bZ}\ee^{n(\theta+2)}\left(\|v_n\|_{\bH^{\gamma+1}_{p}(\bR^d,\ee^{-2n/\alpha}T)}^p+\| \widetilde{f}_n\|_{\bH^{\gamma}_{p}(\bR^d,\ee^{-2n/\alpha}T)}^p\right)\\
			&\lesssim&&\| u\|_{\Psi\bH^{\gamma+1}_{p,\theta}(\Omega,T)}^p+\| f\|^p_{\Psi\bH_{p,\theta+2p}^{\gamma}(\Omega,T)}\,.\nonumber
		\end{alignat}
		Noting that \eqref{240419720} holds for $\partial_t^\alpha u:=\cL u+f$, and
		\begin{align}\label{241106200}
			\begin{split}
				\| \partial_t^\alpha u\|_{\Psi\bH_{p,\theta+2p}^{\gamma}(\Omega,T)}\leq\,& \| \cL u\|_{\Psi\bH_{p,\theta+2p}^{\gamma}(\Omega,T)}+\| f\|_{\Psi\bH_{p,\theta+2p}^{\gamma}(\Omega,T)}\\
				\lesssim\,& \| u\|_{\Psi\bH_{p,\theta}^{\gamma+2}(\Omega,T)}+\|f\|_{\Psi\bH_{p,\theta+2p}^{\gamma}(\Omega,T)}
			\end{split}
		\end{align}
		(see \eqref{2411011115}).
		These yield $u\in\Psi \mathring{\cH}_{p,\theta}^{\alpha,\gamma+2}(\Omega,T)$. 
		In addition, \eqref{241106159} and \eqref{241106200} imply \eqref{2409111129} for $s=\gamma+1$.
		Therefore the proof for $s=\gamma+1$ is complete.

		\textbf{Step 2.} Recall that the case $s\geq \gamma+1$ was proved in Step 1.
		For $s<\gamma+1$, take $k\in\bN$ such that
		$$
		\gamma+1-k\leq s<\gamma+2-k\,,
		$$
		and repeatedly apply the result in Step 1 with $(s,\gamma)$ replaced by $(s,\gamma-k)$, $(\gamma-k,\gamma+1-k)$, ..., $(\gamma-1,\gamma)$.
		Then we obtain $u\in \mathring{\cH}_{p,\theta}^{\alpha,\gamma+2}$ and
		\begin{align*}
			\|\Phi u\|_{\mathring{\cH}^{\alpha,\gamma+2}_{p,\theta}(\Omega,T)}\,&\lesssim \|\Phi u\|_{\bH^{\gamma+1}_{p,\theta}(\Omega,T)}+\|\Phi f\|_{\bH_{p,\theta+2p}^{\gamma}(\Omega,T)}\\
			&\lesssim \,\,\,\,\cdots\,\,\,\,\\
			&\lesssim \|\Phi u\|_{\bH^{\gamma-k}_{p,\theta}(\Omega,T)}+\|\Phi f\|_{\bH_{p,\theta+2p}^{\gamma}(\Omega,T)}\\
			&\lesssim \|\Phi u\|_{\bH^{s}_{p,\theta}(\Omega,T)}+\|\Phi f\|_{\bH_{p,\theta+2p}^{\gamma}(\Omega,T)}\,.
		\end{align*}
		Therefore the proof is complete.
	\end{proof}

	\begin{lemma}\label{22.04.15.0430}
		Suppose the following holds:
		\begin{itemize}
			\item[]For any $f\in \Psi \bH_{p,\theta+2p}^{\gamma}(\Omega)$, equation \eqref{220616124} has a unique solution $u$ in $\Psi \mathring{\cH}_{p,\theta}^{\alpha,\gamma+2}(\Omega,T)$.
			Moreover, we have
			\begin{align}\label{2409111136}
				\|u\|_{\Psi \mathring{\cH}_{p,\theta}^{\alpha,\gamma+2}(\Omega,T)}\leq  N_{\gamma}\|f\|_{\Psi \bH_{p,\theta+2p}^{\gamma}(\Omega,T)}\,,
			\end{align}
			where $N_{\gamma}$ is a constant independent of $u_0$, $f$, and $u$.
		\end{itemize}
		Then for all $s\in\bR$, the following holds:
		\begin{itemize}
			\item[]
			For any $f\in \Psi\bH_{p,\theta+2p}^{s}(\Omega,T)$, equation \eqref{220616124} has a unique solution $u$ in $\Psi  \mathring{\cH}_{p,\theta}^{\alpha,s+2}(\Omega,T)$.
			Moreover, we have
			\begin{align}\label{240419317}
				\|u\|_{\Psi \mathring{\cH}_{p,\theta}^{\alpha,s+2}(\Omega,T)}\leq N_{s}\|f\|_{\Psi \bH_{p,\theta+2p}^{s}(\Omega,T)}\,,
			\end{align}
			where $N_s$ is a constant depending only on $d,\,p,\,\gamma,\,\theta,\,\nu_1,\,\nu_2,\,\mathrm{C}_2(\Psi),\,N_{\gamma},\,s$.
		\end{itemize}
	\end{lemma}
	\begin{proof}
		To prove uniqueness, suppose that $\overline{u}\in \Psi \mathring{\cH}_{p,\theta}^{\alpha,s+2}(\Omega,T)$ satisfies $\partial_t^\alpha\overline{u}=\cL \overline{u}$ and $\overline{u}(0,\cdot)\equiv 0$.
		By Lemma \ref{21.05.13.9}, $\overline{u}$ belongs to $\Psi \mathring{\cH}_{p,\theta}^{\alpha,\gamma+2}(\Omega,T)$.
		By the uniqueness of solutions in the assumption of this lemma, $\overline{u}(t)\equiv 0$ for almost every $t\in(0,T]$. 
		Therefore, uniqueness of solutions follows.
		Thus, it remains to show the existence of solutions and estimate \eqref{240419317}.
		
		\textbf{Step 1.} We first consider the case $s> \gamma$.
		Let $f\in\Psi \bH^{s}_{p,\theta+2p}(\Omega,T)$.
		Since $\Psi \bH^s_{p,\theta+2p}(\Omega,T)\subset \Psi \bH^{\gamma}_{p,\theta+2p}(\Omega,T)$, it follows that $f$ belongs to $\Psi \bH^{\gamma}_{p,\theta+2p}(\Omega,T)$, and hence there exists a solution $u\in \Psi \mathring{\cH}_{p,\theta}^{\alpha,\gamma+2}(\Omega,T)$ of equation \eqref{220616124}.
		It follows from Lemma \ref{21.05.13.9}, \eqref{2409111136}, and \eqref{2411011107} that
		\begin{alignat*}{2}
			\left\|u\right\|_{\Psi \mathring{\cH}^{\alpha,s+2}_{p,\theta}(\Omega,T)}\,&\lesssim_N\left\|u\right\|_{\Psi \bH^{\gamma+2}_{p,\theta}(\Omega,T)}+\left\|f\right\|_{\Psi \bH^{s}_{p,\theta+2p}(\Omega,T)}\\
			\, &\leq N_{\gamma}\|f\|_{\Psi \bH_{p,\theta+2p}^{\gamma}(\Omega,T)}+\|f\|_{\Psi \bH_{p,\theta+2p}^s(\Omega,T)}\\
			\,&\lesssim_N (N_{\gamma}+1)\left\|f\right\|_{\Psi \bH^{s}_{p,\theta+2p}(\Omega)}\,,
		\end{alignat*}
		where $N=N(d,p,\theta,\gamma,\mathrm{C}_2(\Psi),s)$.
		Therefore $u$ belongs to $\Psi \mathring{\cH}_{p,\theta}^{\alpha,s+2}(\Omega,T)$, and the proof is complete.
		
		\textbf{Step 2.} Consider the case $s<\gamma$. By mathematical induction, it suffices to show that if this lemma holds for $s=s_0+1$, then it also holds for $s=s_0$. 
		(Recall that the case $s\geq\gamma$ was proved in Step 1.)
		
		Let us assume that this lemma holds for $s=s_0+1$. 
		For $f\in \Psi \bH_{p,\theta+2p}^{s_0}(\Omega,T)$, by Lemma \ref{22.02.16.1}, there exist functions $f^0\in \Psi \bH_{p,\theta+2p}^{s_0+1}(\Omega,T)$ and $f^1,\,\ldots,\,f^d\in \Psi \bH_{p,\theta+p}^{s_0+1}(\Omega,T)$	such that $f=f^0+\sum_{i=1}^dD_if^i$ and
		\begin{align*}
			\left\|f^0\right\|_{\Psi \bH_{p,\theta+2p}^{s_0+1}(\Omega,T)}+\sum_{i=1}^d\left\|\trho^{\,-1}f^i\right\|_{\Psi \bH_{p,\theta+2p}^{s_0+1}(\Omega,T)}
			\lesssim_N\|f\|_{\Psi \bH_{p,\theta+2p}^{s_0}(\Omega,T)}\,,
		\end{align*}
		where $N=N(d,p,\theta,s_0,\mathrm{C}_2(\Psi))$.
		By the induction hypothesis, there exist $v^0,\,\cdots,\,v^d\in \Psi \mathring{\cH}_{p,\theta}^{\alpha,s_0+3}(\Omega,T)$ such that
		\begin{align*}
			\partial_t^\alpha v^0=\cL v^0 + f^0\quad\text{and}\quad \partial_t^\alpha v^i=\cL v^i + \trho^{\,-1}f^i\quad\text{for }i=1,\,\ldots,\,d\,,
		\end{align*}
		and
		\begin{equation}\label{21.09.30.100}
			\begin{alignedat}{2}
				\sum_{i=0}^d\left\|v^i\right\|_{\Psi \mathring{\cH}^{\alpha,s_0+3}_{p,\theta}(\Omega,T)}&\leq&&N_{s_0+1}\bigg(\left\|f^0\right\|_{\Psi \bH^{s_0+1}_{p,\theta+2p}(\Omega)}+\sum_{i=1}^d\left\|\trho^{\,-1} f^i\right\|_{\Psi \bH^{s_0+1}_{p,\theta+2p}(\Omega)}\bigg)\\
				&\lesssim_N\,&&N_{s_0+1}\|f\|_{\Psi \bH_{p,\theta+2p}^{s_0}(\Omega)}\,,
			\end{alignedat}
		\end{equation}
		where $N=N(d,p,\theta,s_0,\mathrm{C}_2(\Psi))$.
		Put $v=v^0+\sum_{i=1}^dD_i\big(\trho v^i\big)$, and observe that
		$$
		\partial_t^\alpha u=\cL v+f-\sum_{i=1}^dD_i\big(\cL(\trho v^i)-\trho\cL v^i\big)\,.
		$$
		Using \eqref{2411011115} and \eqref{220526558}, we obtain
		\begin{alignat*}{2}
			\left\|D_i\left(\cL(\trho v^i)-\trho\cL v^i\right)\right\|_{\Psi \bH^{s_0+1}_{p,\theta+2p}(\Omega)}\,&\lesssim_N \|\cL(\trho v^i)-\trho\cL v^i\|_{\Psi \bH^{s_0+2}_{p,\theta+p}(\Omega)}\\
			&\lesssim_N  \big\|\trho_{xx} v^i\big\|_{\Psi \bH^{s_0+2}_{p,\theta+p}(\Omega)}+\big\|\trho_x v^i_x\big\|_{\Psi \bH^{s_0+2}_{p,\theta+p}(\Omega)}\\
			&\lesssim_N \|v^i\|_{\Psi \bH^{s_0+3}_{p,\theta}(\Omega)}<\infty\,,
		\end{alignat*}
		where $N=N(d,p,\theta,s_0,\mathrm{C}_2(\Psi),\nu_1,\nu_2)$.
		By the induction hypothesis, there exists $w\in \Psi \mathring{\cH}^{\alpha,s_0+3}_{p,\theta}(\Omega,T)$ such that
		$$
		\partial_t^\alpha w=\cL w+\sum_{i=1}^dD_i\big(\cL(\trho v^i)-\trho\cL v^i\big)\quad(=\cL v-f)\,.
		$$
		Such $w$ satisfies
		\begin{equation}\label{21.09.30.200}
			\begin{alignedat}{2}
				\|w\|_{\Psi \mathring{\cH}^{\alpha,s_0+3}_{p,\theta}(\Omega,T)}&\leq&& N_{s_0+1}\sum_{i=1}^d\left\|D_i\left(\cL(\trho v^i)-\trho\cL v^i\right)\right\|_{\Psi \bH^{s_0+1}_{p,\theta+2p}(\Omega)}\\
				&\lesssim_N &&N_{s_0+1}\sum_{i=1}^d\|v^i\|_{\Psi \bH^{s_0+3}_{p,\theta}(\Omega)}\,,
			\end{alignedat}
		\end{equation}
		where $N=N(d,p,\theta,s_0,\mathrm{C}_2(\Psi),\nu_1,\nu_2)$.
		Put $u=v+w=v^0+\sum_{i=1}^d D_i(\trho v^i)+w$.
		Then $u$ satisfies $\partial_t^\alpha u=\cL u+ f$.
		Moreover, by \eqref{21.09.30.100} and \eqref{21.09.30.200}, we obtain \eqref{240419317} for $s=s_0$.
	\end{proof}

	Now, we provide the proof of Theorem \ref{22.02.18.6}.
	
	\begin{proof}[Proof of Theorem \ref{22.02.18.6}]
		
		(1) By Lemma \ref{22.04.15.0430}, it suffices to prove the case $\gamma=0$.
		
		\textbf{\textit{A priori} estimates.} Make use of Theorem \ref{05.11.2} and \eqref{2409098051} to obtain that for any $u\in C_c^{\infty}\big((0,T]\times\Omega\big)$, 
		\begin{align*}
			\begin{split}
				\|u\|_{\Psi^{\mu}\bL_{p,d-2}(\Omega,T)}&\lesssim_N \|\partial_t^\alpha u-\cL u\|_{\Psi^{\mu}\bL_{p,d+2p-2}(\Omega,T)}\,,
			\end{split}
		\end{align*}
		where $N=N(d,\alpha,p,\mu,\mathrm{C}_0(\Omega),\mathrm{C}_2(\Psi),\mathrm{C}_3(\psi,\Psi),\mathrm{C}_4)$.
		Combining this with Lemma \ref{21.05.13.9}, we have
		\begin{equation}\label{2206081117}
			\begin{aligned}
				\|u\|_{\Psi^{\mu}\mathring{\cH}_{p,d-2}^{\alpha,2}(\Omega,T)}\lesssim\,&\|u\|_{\Psi^{\mu}\bL_{p,d-2}(\Omega,T)}+\|\partial_t^\alpha u-\cL u\|_{\Psi^{\mu}\bL_{p,d+2p-2}(\Omega,T)}\\
				\lesssim\,&\|\partial_t^\alpha u-\cL u\|_{\Psi^{\mu}\bL_{p,d+2p-2}(\Omega,T)}\,.
			\end{aligned}
		\end{equation}
		By Lemma \ref{240911115}, \eqref{2206081117} also holds for all $u\in\Psi^{\mu} \mathring{\cH}_{p,d-2}^{\alpha,2}(\Omega,T)$.
		Therefore the \textit{a priori} estimates follow.
		The uniqueness of solutions also follows from \eqref{2206081117}.
		
		\textbf{Existence of solutions.}
		We first consider the case $\cL=\nu_1\Delta$.
		For a fixed $f\in\Psi^{\mu}\bL_{p,d+2p-2}(\Omega,T)$, Lemma \ref{2205241011} implies that there exists a sequence $\big(f^{(n)}\big)_{n\in\bN}\subset C_c^{\infty}([0,T]\times\Omega)$ such that $f^{(n)}\rightarrow f$ in $\Psi^{\mu}\bL_{p,d+2p-2}(\Omega,T)$.
		Make use of Lemmas \ref{21.05.25.300} and \ref{21.05.13.9} (for $\gamma=2$ and $s=0$) to obtain that there exists a solution $u^{(n)}\in \Psi^{\mu}\mathring{\cH}_{p,d-2}^{\alpha,2}(\Omega,T)$ of the equation
		\begin{align}\label{220613418}
			\partial_t^\alpha u^{(n)}=\nu_1\Delta u^{(n)}+f^{(n)}\,.
		\end{align}
		By \eqref{2206081117}, 
		\begin{alignat*}{2}
			\|u^{(n)}-u^{(m)}\|_{\Psi^{\mu}\mathring{\cH}_{p,d-2}^{\alpha,2}(\Omega,T)}&\lesssim\,&&\|\big(\partial_t^\alpha-\nu_1\Delta\big) \big(u^{(n)}-u^{(m)}\big)\|_{\Psi^{\mu}\bL_{p,d+2p-2}(\Omega,T)}\\
			&=&&\|f^{(n)}-f^{(m)}\|_{\Psi^{\mu}\bL_{p,d+2p-2}(\Omega,T)}\,.
		\end{alignat*}
		This implies that $\big(u^{(n)}\big)_{n\in\bN}$ is a Cauchy sequence in $\Psi^{\mu}\mathring{\cH}_{p,d-2}^{\alpha,2}(\Omega,T)$.
		Since $\Psi^{\mu}\mathring{\cH}_{p,d-2}^{\alpha,2}(\Omega,T)$ is a Banach space, there exists $u\in \Psi^{\mu}\mathring{\cH}_{p,d-2}^{\alpha,2}(\Omega,T)$ such that $u^{(n)}\rightarrow u$ in $\Psi^{\mu}\mathring{\cH}_{p,d-2}^{\alpha,2}(\Omega,T)$.
		Letting $n\rightarrow\infty$ in \eqref{220613418}, we have $\partial_t^\alpha u-\nu_1\Delta u=f$.
		Therefore the theorem is proved for the case $\cL=\nu_1\Delta$.
		
		Now consider a general operator $\cL:=\sum_{i,j}a^{ij}D_{ij}\in\cM_T(\nu_1,\nu_2)$.
		For $s\in[0,1]$, put
		$$
		\cL_s:=(1-s)\nu_1\Delta+s\cL=\sum_{i,j=1}^d\big((1-s)\nu_1\delta^{ij}+sa^{ij}\big)D_{ij}\,.
		$$
		For any $t\in(0,T]$, 
		\begin{align}\label{2206081243}
			\nu_1|\xi|^2\leq \sum_{i,j=1}^d\big((1-s)\nu_1\delta^{ij}+sa^{ij}(t)\big)\xi_i\xi_j\leq \nu_2|\xi|^2\,,
		\end{align}
		which implies that $\cL_s\in \cM_T(\nu_1,\nu_2)$.
		It follows from \eqref{2206081117} that
		\begin{align*}
			\|u\|_{\Psi^{\mu}\mathring{\cH}_{p,d-2}^{\alpha,2}(\Omega,T)}\leq N\|\partial_t^\alpha u-\cL_s u\|_{\Psi^{\mu}\bL_{p,d+2p-2}(\Omega,T)}
		\end{align*}
		for all $u\in \Psi^{\mu}\mathring{\cH}_{p,d-2}^{\alpha,2}(\Omega,T)$ and $s\in[0,1]$. Here, $N$ is the constant in \eqref{2206081117}, which is independent of $s$.
		By the method of continuity (see, \textit{e.g.}, \cite[Theorem 5.2]{GT}), the unique solvability for $\cL_0$ yields that for $\cL_1:=\cL$.
		
		(2) We first prove the existence of solutions.
		By Lemma \ref{240824401}, for $u_0\in \Psi^\mu B_{p,d+2/\alpha-2}^{\gamma+2-2/(p\alpha)}(\Omega)$, there exists $v\in \Psi^\mu \cH_{p,d-2}^{\alpha,\gamma+2}(\Omega,T)$ such that $v(0)=u_0$ and 
		\begin{align}\label{240824440}
			\|v\|_{\Psi^\mu \cH_{p,d-2}^{\alpha,\gamma+2}(\Omega,T)}\leq N(d,p,\alpha,\gamma,\mathrm{C}_2(\Psi),\mu)\|u_0\|_{\Psi^\mu B_{p,d+2/\alpha-2}^{\gamma+2-2/(p\alpha)}(\Omega)}\,.
		\end{align}
		By the definition of $\Psi^\mu\cH_{p,d-2}^{\alpha,\gamma+2}(\Omega,T)$, we have $\partial_t^\alpha v-\cL v\in  \bH_{p,d+2p-2}^\gamma(\Omega,T)$ and 
		\begin{alignat}{2}\label{240824441}
			\|\partial_t^\alpha v-\cL v\|_{\Psi^\mu \bH_{p,d+2p-2}^\gamma(\Omega,T)}\lesssim_{d,\nu_1,\nu_2} \|v\|_{\Psi^\mu \cH_{p,d-2}^{\alpha,\gamma+2}(\Omega,T)}\lesssim \|u_0\|_{\Psi^\mu B_{p,d+2/\alpha-2}^{\gamma+2-2/(p\alpha)}(\Omega)}\,.
		\end{alignat}
		By (1) of this theorem and Lemma \ref{240426524}.(2), there exists $w\in \Psi^\mu \cH_{p,d-2}^{\alpha,\gamma+2}(\Omega,T)$ such that
		$$
		\partial_t^\alpha w=\cL w+f-\big(\partial_t^\alpha v-\cL v\big)\quad\text{in}\quad (0,T]\quad ;\quad w(0,\cdot)=0\,.
		$$
		Moreover, by \eqref{240824444} and \eqref{240824441},
		\begin{align*}
			\|w\|_{\Psi^\mu \cH_{p,d-2}^{\alpha,\gamma+2}(\Omega,T)}\,&\lesssim \|f-(\partial_t^\alpha v-\cL v)\|_{\Psi^\mu \bH_{p,d+2p-2}^\gamma(\Omega,T)}\\
			&\lesssim  \|f\|_{\Psi^\mu \bH_{p,d+2p-2}^\gamma(\Omega,T)}+\|u_0\|_{\Psi^\mu B_{p,d+2/\alpha-2}^{\gamma+2-2/(p\alpha)}(\Omega)}\,.
		\end{align*}
		Put $u:=v+w$. 
		Then $u$ is a solution of \eqref{220616124}, and \eqref{240824443} follows from \eqref{240824440} and \eqref{240824441}.
		
		The uniqueness follows from (1) of this theorem. 
		Indeed, if $u\in \Psi^\mu \cH_{p,d-2}^{\alpha,\gamma+2}(\Omega,T)$ satisfies \eqref{220616124} with $f\equiv 0$ and $u_0=0$, then Lemma \ref{240426524} implies that $u\in\Psi^{\mu}\mathring{\cH}_{p,d-2}^{\alpha,\gamma+2}(\Omega,T)$ with $\partial_t^\alpha u-\cL u=0$.
		By the uniqueness of solutions in (1) of this theorem, we conclude that $u=0$ in $\Psi^{\mu}\mathring{\cH}_{p,d-2}^{\alpha,\gamma+2}(\Omega,T)$. 
		This implies that $\|u\|_{\Psi^{\mu}\cH_{p,d-2}^{\alpha,\gamma+2}(\Omega,T)}=0$ (see Lemma \ref{240426524}).
	\end{proof}
	
	We conclude this subsection by establishing the global uniqueness of solutions.
	\begin{thm}\label{220821002901}\,
		Let \eqref{hardy} holds for $\Omega$ and for each $k=1,\,2$, let  
		\begin{equation*}
			\begin{gathered}
				\text{$\psi_k$ is a superharmonic Harnack function in $\Omega$}\,\,,\quad\text{$\Psi_k$ is a regularization of $\psi_k$}\,\,,\\
				\gamma_k\in\bR\,\,,\,\,\,\,p_k\in(1,\infty)\,\,,\,\,\,\, \mu_k\in \cI(\psi_k,p_k,\nu_2/\nu_1)\,.
			\end{gathered}
		\end{equation*}
		Suppose that $	f\in\bigcap_{k=1,2}\Psi_k^{\,\mu_k}\bH_{p_k,d+2p_k-2}^{\gamma_k}(\Omega,T)$.
		
		\begin{enumerate}
			\item 
			For each $k=1,\,2$, let
			$u^{(k)}\in\Psi_k^{\,\mu_k}\mathring{\cH}_{p_k,d-2}^{\alpha,\gamma_k+2}(\Omega,T)$ be the solution to the equation
			$$
			\partial_t^\alpha u^{(k)}=\cL u^{(k)}+f\,.
			$$
			Then $u^{(1)}(t)=u^{(2)}(t)$ in $\cD'(\Omega)$ for almost every $t\in[0,T]$.
			
			\item Assume further that $p_1,\,p_2>\frac{1}{\alpha}$ and $u_0\in\bigcap_{k=1,2}\Psi_k^{\,\mu_k} B_{p_k,d+2/\alpha-2}^{\gamma_k+2-2/(\alpha p_k)}(\Omega)$. 				
			For each $k=1,\,2$, let
			$u^{(k)}\in\Psi_k^{\,\mu_k}\cH_{p_k,d-2}^{\alpha,\gamma_k+2}(\Omega,T)$ be the solution to the equation
			$$
			\partial_t^\alpha u^{(k)}=\cL u^{(k)}+f\quad;\quad u^{(k)}(0)=u_0\,.
			$$
			Then $u^{(1)}(t)=u^{(2)}(t)$ in $\cD'(\Omega)$ for all $t\in[0,T]$.
		\end{enumerate}
		
	\end{thm}
	\begin{proof}
		(1) Define
		\begin{align}\label{2408251215}
			\mathring{X}:=\bigcap_{k=1,2}\Psi_k^{\,\mu_k}\mathring{\cH}_{p_k,d-2}^{\alpha,\gamma_k+2}(\Omega,T)\quad\text{and}\quad Y:=\bigcap_{k=1,2}\Psi_k^{\,\mu_k}\bH_{p_k,d+2p-2}^{\gamma_k}(\Omega,T)\,,
		\end{align}
		and recall that $f\in Y$.
		
		\textbf{Step 1.} Consider the case $\cL=\nu_1\Delta$.
		Lemma \ref{2205241011} yields that there exists $f_n\in C_c^{\infty}\big((0,T]\times\Omega\big)$ such that $f_n\rightarrow f$ in $Y$.
		Since $\mu_k\in \cI(\psi_k,p_k,\nu_2/\nu_1) \subset (-1/p_k,1-1/p_k)$, it follows from Lemmas \ref{21.05.25.300} and \ref{21.05.13.9} that there exists $u_n\in \mathring{X}$ such that
		$$
		\partial_t^\alpha u_n=\nu_1 \Delta u_n+f_n\,.
		$$
		By Theorem \ref{22.02.18.6}, we obtain that for each $k=1,\,2$,
		$$
		\lim_{n\rightarrow \infty}\|u_n-u^{(k)}\|_{\Psi_k^{\,\mu_k}\mathring{\cH}_{p_k,d-2}^{\alpha,\gamma_k+2}(\Omega,T)}\lesssim \lim_{n\rightarrow \infty}\|f_n-f\|_{\Psi_k^{\,\mu_k}\bH_{p_k,d+2p-2}^{\gamma_k}(\Omega,T)}=0\,.
		$$
		This implies that
		$$
		u^{(1)}(t)=u^{(2)}(t)\quad\text{in}\,\,\,\cD'(\Omega)\,\,,\,\,\,\,\text{for almost every}\,\,t\in[0,T]\,.
		$$
		Therefore the case $\cL=\nu_1\Delta$ is proved.
		
		Note that
		$u^{(1)}(\,\cdot\,)$($=u^{(2)}(\,\cdot\,)$) is the unique solution of
		\begin{align}\label{2206211148}
			\partial_t^\alpha u=\nu_1\Delta u+f\,,
		\end{align}
		in the class $\mathring{X}$.
		Indeed, $\Psi_1^{\,\mu_1}\mathring{\cH}_{p_1,d-2}^{\alpha,\gamma_1+2}(\Omega,T)$ admits a unique solution to equation \eqref{2206211148}, and $\mathring{X}\subset \Psi_1^{\,\mu_1}\mathring{\cH}_{p_1,d-2}^{\alpha,\gamma_1+2}(\Omega,T)$.

		\textbf{Step 2.} Let $\cL\in\cM_T(\nu_1,\nu_2)$.
		For $r\in[0,1]$, we denote $\cL_r:=(1-r)\nu_1\Delta+r\cL$.
		By \eqref{2206081243}, Theorem \ref{22.02.18.6} implies that 
		\begin{align*}
			\|u\|_{\mathring{X}}&:=\sum_{k=1,2}\|u\|_{\Psi_k^{\,\mu_k}\mathring{\cH}_{p_k,d-2}^{\alpha,\gamma_k+2}(\Omega,T)}\\
			&\leq N\sum_{k=1,2}\|\partial_t^\alpha u-\cL_ru\|_{\Psi_k^{\,\mu_k}\bH_{p_k,d+2p-2}^{\gamma_k+2}(\Omega,T)}=N\|\partial_t^\alpha u-\cL_ru\|_Y
		\end{align*}
		for all $u\in \mathring{X}$, where $N$ is independent of $u$ and $r\in[0,1]$.
		In addition, by the result in Step 1, the map $u\mapsto \partial_t^\alpha u-\cL_0u$ is a bijective map from $\mathring{X}$ to $Y$.
		Therefore the method of continuity ensures that for any $f\in Y$, there exists a unique solution $u\in \mathring{X}$ of the equation
		\begin{align}\label{2204160344}
			\partial_t^\alpha u=\cL u+f\,.
		\end{align}
		For each $k=1,\,2$,  $u^{(k)}$ is the unique solution of equation \eqref{2204160344} in $\Psi_k^{\,\mu_k}\mathring{\cH}_{p_k,d-2}^{\alpha,\gamma_k+2}(\Omega,T)$. 
		Therefore $u=u^{(k)}$ in $\Psi_k^{\,\mu_k}\cH_{p_k,d-2}^{\alpha,\gamma_k+2}(\Omega,T)$.
		Consequently, $u^{(1)}(t)=u(t)=u^{(2)}(t)$ for almost every $t\in (0,T]$.

		(2) Use the notations in \eqref{2408251215}, define 
		\begin{align*}
			\begin{gathered}
				X:=\bigcap_{k=1,2}\Psi_k^{\,\mu_k}\cH_{p_k,d-2}^{\alpha,\gamma_k+2}(\Omega,T)\quad\text{and}\quad Z:=\bigcap_{k=1,2}\Psi_k^{\,\mu_k}B_{p_k,d+2/\alpha-2}^{\gamma_k+2-2/(\alpha p_k)}(\Omega)\,,
			\end{gathered}
		\end{align*}
		and recall that $u_0\in Z$.
		By Lemma \ref{22.04.11.3}, there exists $u_{0,n}\in C_c^{\infty}(\Omega)$ such that $u_{0,n}\rightarrow u_0$ in $Z$.
		Recall that $T<\infty$, and hence 
		$$
		\|u_{0,n}\|_{\Psi_k^{\,\mu_k}\bH_{p_k,d+2p_k-2}^{\gamma_k}(\Omega,T)}=T^{1/p}\|u_{0,n}\|_{\Psi_k^{\,\mu_k}H_{p_k,d+2p_k-2}^{\gamma_k}(\Omega)}<\infty\,.
		$$
		By (1) of this theorem, there exists $v_n\in \mathring{X}$ such that
		$$
		\partial_t^\alpha v_n=\cL v_n+f+\cL u_{0,n}\,.
		$$
		In view of Lemma \ref{240426524}.(2), we regard $v_n$ as an element of $X$ with $v_n(0)=0$.
		Put $u_n(t,\cdot)=v_n(t,\cdot)+u_{0,n}$.
		Since $u_{0,n}\in C_c^{\infty}(\Omega)$ and $v_n\in X$, $u_n$ also belongs to $X$.
		Moreover, we have
		$$
		\partial_t^\alpha u_{n}-\cL u_{n}=\big(\partial_t^\alpha v_{n}-\cL v_{n}\big)+\big(\partial_t^\alpha u_{0,n}-\cL u_{0,n}\big)=(f+\cL u_{0,n})-\cL u_{0,n}=f
		$$ 
		and $u_n(0,\cdot)=v_{n}(0,\cdot)+u_{0,n}=u_{0,n}$.
		Note that for each $k=1,\,2$,
		$$
		\partial_t^\alpha\big(u^{(k)}-u_n\big)=\cL(u^{(k)}-u_n\big)\quad\text{in}\quad (0,T]\quad;\quad \big(u^{(k)}-u_n\big)(0,\cdot)=u_0-u_{0,n}\,.
		$$
		By Theorem \ref{22.02.18.6}.(2), we obtain that
		\begin{align*}
			\|u^{(k)}-u_{n}\|_{\Psi_k^{\,\mu_k}\cH_{p_k,d-2}^{\alpha,\gamma_k+2}(\Omega,T)}\lesssim \|u_0-u_{0,n}\|_{\Psi_k^{\,\mu_k}B_{p_k,d+2/\alpha-2}^{\gamma_k+2-2/(\alpha p)}(\Omega)}\,.
		\end{align*}
		This implies that for each $k=1,\,2$, $u_n$ converges to $u^{(k)}$ in $\Psi_k^{\,\mu_k}\cH_{p_k,d-2}^{\alpha,\gamma_k+2}(\Omega,T)$. 
		By \eqref{240911422}, for any $t\in(0,T]$, $u_n(t)$ converges to $u^{(k)}(t)$ in $\cD'(\Omega)$.
		Therefore we conclude that $u^{(1)}(t)=u^{(2)}(t)$ in $\cD'(\Omega)$ for all $t\in(0,T]$.
	\end{proof}

	\mysection{Application}\label{app.}

	We apply Theorem \ref{22.02.18.6} to various domain settings: the fat exterior condition, the thin exterior condition, convex domains, the totally vanishing exterior Reifenberg condition, and conic domains.
	Notably, we also treat parabolic equations with time-measurable coefficients in the case where a domain $\Omega$ is convex or satisfying the totally vanishing exterior Reifenberg condition.
	
	Throughout this section, we assume that $\Omega$ is a (possibly unbounded) domain in $\bR^d$ with $d \ge 2$.
	Recall that $\rho(x)$ denotes the boundary distance function $d(x,\partial\Omega)$, and $\trho$ is a regularization of $\rho$ provided by Lemma \ref{21.05.27.3}.(1).
	
	We summarize basic properties of classical superharmonic functions.
	
	\begin{lemma}\label{21.05.18.1}\,
		
		\begin{enumerate}
			\item Let $\phi_1,\,\phi_2$ be classical superharmonic functions in $\Omega$. Then $\phi_1\wedge \phi_2$ is also a classical superharmonic function in $\Omega$.
			
			\item Let $\{\phi_{i}\}_{i\in\cI}$ be an arbitrary family of positive classical superharmonic functions in $\Omega$.
			Then $\inf_{i\in\cI}\phi_{i}$ is a superharmonic function in $\Omega$.
			
			\item Let $\phi_1,\,\phi_2$ be positive classical superharmonic functions in $\Omega$. For any $s\in(0,1)$, $\phi_1^{s}\phi_2^{1-s}$ is also a classical superharmonic function in $\Omega$; in particular, $\phi_1^{s}$ is a classical superharmonic function for all $s\in(0,1)$.

			\item Let $\Omega_1$ and $\Omega_2$ be open sets in $\bR^d$ and $\phi_i$ be a classical superharmonic function in $\Omega_i$, for $i=1,\,2$.
			Suppose that 
			\begin{alignat*}{2}
				&\liminf_{x\rightarrow x_1,x\in\Omega_2}\phi_2(x)\geq \phi_1(x_1)\quad &&\text{for all}\quad  x_1\in \Omega_1\cap \partial\Omega_2\,;\\
				&\liminf_{x\rightarrow x_2,x\in\Omega_1}\phi_1(x)\geq \phi_2(x_2)\quad &&\text{for all}\quad  x_2\in \Omega_2\cap \partial\Omega_1\,.
			\end{alignat*}
			Then the function
			\begin{align*}
				\phi(x):=
				\begin{cases}
					\phi_1(x)&\quad x\in\Omega_1\setminus\Omega_2\\
					\phi_1(x)\wedge \phi_2(x) &\quad x\in \Omega_1\cap \Omega_2\\
					\phi_2(x)&\quad x\in\Omega_2\setminus\Omega_1\\
				\end{cases}
			\end{align*}
			is also a classical superharmonic function in $\Omega1\cup\Omega_2$.
			
		\end{enumerate}
	\end{lemma}
	
	The proof of Lemma \ref{21.05.18.1} is as follows:
	(1) follows directly from the definition of classical superharmonic functions; (2) and (3) are proved in \cite[Theorem 3.7.5, Corollary 3.4.4]{AG}, respectively; (4) follows from \cite[Corollary 3.2.4]{AG}.
	
	Prior to discussing various domain conditions, it is worth recalling that the fat exterior condition and the thin exterior condition are among the most well-known criteria ensuring the Hardy inequality.
	These conditions are closely related to the geometric properties of $\Omega$, specifically the Hausdorff dimension and the Aikawa dimension of $\Omega^c$.
	
	\begin{lemma}\label{240928354}
		A domain $\Omega\subset \bR^d$ admits the Hardy inequality \eqref{hardy} if one of the following conditions holds:
		\begin{enumerate}
			\item(Fat exterior condition) There exist constants $s>d-2$ and $c>0$ such that
			\begin{align}\label{220617253}
				\cH^{s}_{\infty}\big(\Omega^c\cap \overline{B}(p,r)\big)\geq cr^{s}\quad\text{for all }p\in \partial\Omega\,\,\text{and}\,\,r>0\,.
			\end{align}
			Here, $\cH_{\infty}^{s}(E)$ is the Hausdorff content of $E\subset \bR^d$ defined by
			$$
			\cH_{\infty}^{s}(E):=\inf \Big\{\sum_{i\in\bN}r_i^{s}\,:\,E\subset \bigcup_{i\in\bN}B(x_i,r_i)\quad\text{where }x_i\in E\text{ and }r_i>0\Big\}\,.
			$$
			
			\item(Thin exterior condition) $\beta_0:=\dim_{\cA}(\Omega^c)<d-2$, where $\dim_{\cA}(E)$ denotes the Aikawa dimension of $E$, defined as the infimum of $\beta\geq 0$ for which there exists a constant $A_\beta$ such that
			$$
			\sup_{p\in E,\,r>0}\frac{1}{r^{\beta}}\int_{B_r(p)}\frac{1}{d(x,E)^{d-\beta}}\dd x\leq A_\beta<\infty\,,
			$$
			with considering $\frac{1}{0}=+\infty$.
		\end{enumerate}
		Moreover, if the first condition holds, then the constant $\mathrm{C}_0(\Omega)$ in \eqref{hardy} can be chosen to depend only on $d$, $s$, and $c$.
		If the second condition holds, then $\mathrm{C}_0(\Omega)$ can be chosen to depend only on $d$, $\beta_0$, and $\{A_\beta\}_{\beta>\beta_0}$.
	\end{lemma}
	We refer the reader to \cite{ward} for further details and discussion of this dimensional dichotomy.
	
	\begin{remark}\label{241110236}
		The Aikawa dimension is equivalent to the Assouad dimension, which is defined in terms of a covering property (see \cite[Theorem 1.1]{LT}).
		For any $E\subset \bR^d$, the inequality $\dim_{\cH}(E)\leq \dim_{\cA}(E)$ holds,
		and the equality need not hold in general (see \cite[Section 2.2]{lehr}).
		However, if $E$ is Ahlfors regular, for example, if $E$ is a smooth manifold or has a self-similar structure such as the Cantor set or the Koch snowflake, then $\dim_{\cH}(E)=\dim_{\cA}(E)$; see \cite[Lemma 2.1]{lehr} and \cite[Theorem 4.14]{Mattila}.	
	\end{remark}

	\subsection{Domains with fat exterior: Harmonic measure decay property}\label{fatex}\,\,
	
	We begin by recalling several conditions equivalent to the fat exterior condition.
	For $p\in\partial\Omega$ and $r>0$, we denote by $w(\,\cdot\,,p,r):=w\big(\,\cdot\,,\Omega\cap B_r(p),\Omega\cap \partial B_r(p)\big)$ the \textit{harmonic measure} of $\Omega\cap \partial B_r(p)$ over $\Omega\cap B_r(p)$, which is defined as the Perron-Wiener-Brelot solution $u$ of the equation 
	\begin{align*}
		\Delta u=0\quad\text{in}\,\,\,\Omega\cap B_r(p)\quad;\quad u=1_{\Omega\cap \partial B_r(p)}\quad\text{on}\,\,\,\partial \big(\Omega\cap B_r(p)\big)\,.
	\end{align*}
	In other words,
	\begin{align}\label{241031116}
		\begin{split}
			w(x,p,r):=\inf\big\{\phi(x)\,:\,\,&\text{$\phi$ is a classical superharmonic function in $\Omega\cap B_r(p)$}\\
			&\,\, \text{and $\underset{y\rightarrow z}{\liminf}\,\phi(y)\geq 1_{\Omega\cap \partial B_r(p)}(z)$ for all $z\in\partial \big(\Omega\cap B_r(p)\big)$}\big\}\,.
		\end{split}
	\end{align}
	(see Figure \ref{230113736} below).
	Note that $\Omega\cap \partial B_r(p)$ is a relatively open subset of $\partial\big(\Omega\cap B_r(p)\big)$.

	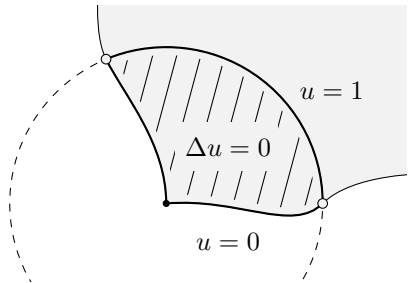
\begin{figure}[h]
		\begin{tikzpicture}[> = Latex]
			\begin{scope}
				\begin{scope}[scale=0.8]
					\clip (-4,-1.8) rectangle (4,2.8);
					
					\begin{scope}[shift={(0.5,0)}]
						\fill[gray!10] (-0.5,-0.5) .. controls +(0,1) and +(0.3,-0.6) ..(-1.5,1.9) .. controls +(-0.3,0.6) and +(0.2,-0.2) .. (-1.7,4) -- (4,4) -- (4,0) .. controls +(-0.2,0) and +(0.5,0.5) .. (2.1,-0.5) .. controls +(-0.5,-0.5) and +(1.2,0.1)..(-0.5,-0.5);
						
					\end{scope}

					\begin{scope}[shift={(0.5,0)}]
						\clip (-0.5,-0.5) .. controls +(0,1) and +(0.3,-0.6) ..(-1.5,1.9) arc (acos(-5/13):0:2.6)  .. controls +(-0.5,-0.5) and +(1.2,0.1)..(-0.5,-0.5);
						\foreach \i in {-4,-3.6,...,3}
						{\draw (\i,-2.8)--(\i+1.5,2.8);}
						\path[fill=gray!10] (-0.35,0.05) rectangle (1.35,0.85);
						
						\draw[gray!10,line width=6pt] (2.1,-0.5) arc (0:acos(-5/13):2.6);
						\draw[gray!10,line width=6pt] (2.1,-0.5) .. controls +(-0.5,-0.5) and  +(1.2,0.1)..	(-0.5,-0.5) .. controls +(0,1) and +(0.3,-0.6) ..(-1.5,1.9);
						
					\end{scope}
					
					\begin{scope}[shift={(0.5,0)}]
						\draw[dashed] (-0.5,-0.5) circle (2.6);
						\draw
						(-0.5,-0.5) .. controls +(0,1) and +(0.3,-0.6) ..(-1.5,1.9) .. controls +(-0.3,0.6) and +(0.2,-0.2) .. (-1.7,4);
						\draw
						(4,0) .. controls +(-0.2,0) and +(0.5,0.5) .. (2.1,-0.5) .. controls +(-0.5,-0.5) and  +(1.2,0.1)..	(-0.5,-0.5);

						\draw[ line width=0.8pt] (2.1,-0.5) arc (0:acos(-5/13):2.6);
						\draw[line width=0.8pt] (2.1,-0.5) .. controls +(-0.5,-0.5) and  +(1.2,0.1)..	(-0.5,-0.5) .. controls +(0,1) and +(0.3,-0.6) ..(-1.5,1.9);
						\draw[fill=gray!10] (2.1,-0.5) circle (0.08);
						\draw[fill=gray!10] (-1.5,1.9) circle (0.08);
						\draw[fill=black] (-0.5,-0.5) circle (0.05);
					\end{scope}

				\end{scope}
				
				\begin{scope}[shift={(0.4,0)}]
					\draw (0.42,0.35) node  {$\Delta u=0$} ;
					\draw (1.8,1.1) node (1) {$u=1$};
					\draw (0.42,-0.9) node (0) {$u=0$};
					
				\end{scope}

			\end{scope}
		\end{tikzpicture}
		\caption{$u:=w(\,\cdot\,,p,r)$}\label{230113736}
	\end{figure}
	
	\begin{defn}\label{2301021117}
		A domain $\Omega$ is said to satisfy the \textit{local harmonic measure decay property} with exponent $\lambda>0$ (abbreviated as `$\mathbf{LHMD}(\lambda)$'), if there exists a constant $M_{\lambda}>0$ such that for any $p\in\partial \Omega$ and $r>0$,
		\begin{align}\label{21.08.03.1}
			w(x,p,r)\leq M_{\lambda}\left(\frac{|x-p|}{r}\right)^{\lambda}\quad\text{for all}\,\,\, x\in \Omega\cap B(p,r)
		\end{align}
	\end{defn}
	
	\begin{remark}
		$\mathbf{LHMD}$ is closely related to the H\"older continuity of the PWB solutions.
		For the moment, we assume that $\Omega$ is bounded and satisfies the following property:
		\begin{itemize}
			\item[] For any $f\in C(\partial\Omega)$, the PWB solution of the equation 
			$$
			\Delta u=0\quad\text{in}\,\,\,\Omega\quad;\quad u=f\quad\text{on}\,\,\,\partial\Omega
			$$
			coincides with the classical solution; that is, $u\in C^\infty(\Omega)\cap C(\overline{\Omega})$ and the equation holds pointwise.
		\end{itemize}
		Here, the PWB solution $u$ is defined as in \eqref{241031116}, by replacing $\big(\Omega\cap B_r(p),1_{\Omega\cap \partial B_r(p)}\big)$ with $(\Omega,f)$.
		We denote this solution $u$ by $H_\Omega f$.
		For $\lambda\in(0,1]$, we denote by $\|H_\Omega\|_{\lambda}$ the operator norm of $H_\Omega$ from $C^{0,\lambda}(\partial \Omega)$ to $C^{0,\lambda}(\Omega)$, \textit{i.e.},
		$$
		\|H_\Omega\|_{\lambda}:=\sup_{f\in C^{0,\lambda}(\partial \Omega),f\not\equiv 0}\frac{\|H_\Omega f\|_{C^{0,\lambda}(\Omega)}}{\|f\|_{C^{0,\lambda}(\partial\Omega)}}\,.
		$$
		It is shown in \cite[Theorems 2 and 3]{aikawa2002} that $\|H_\Omega\|_1=\infty$. Moreover, for $\lambda\in(0,1)$, if $\|H_\Omega\|_{\lambda}<\infty$ then $\Omega$ satisfies $\mathbf{LHMD}(\lambda)$.
		Conversely, if $\Omega$ satisfies $\mathbf{LHMD}(\lambda')$ for some $\lambda'>\lambda$, then $\|H_\Omega\|_{\lambda}<\infty$.
	\end{remark}

	\begin{example}[Section 5 in \cite{Seo202404}]\label{240928314}\,
		
		\begin{enumerate}
			
			\item Every convex domain satisfies \textbf{LHMD}($1$), where $M_1$ in \eqref{21.08.03.1} depends only on $d$.
			
			\item Let $\Omega$ be a domain satisfying the totally vanishing exterior Reifenberg condition (see Definition \ref{2209151117}).
			Then for any $\epsilon>0$, $\Omega$ satisfies \textbf{LHMD}($1-\epsilon$), where $M_{1-\epsilon}$ in \eqref{21.08.03.1} depends only on $d$, $\epsilon$, $\{R_{0,\delta}/R_{\infty,\delta}\}_{\delta>0}$.
			Here, $R_{0,\delta}$ and $R_{\infty,\delta}$ are constants in \eqref{220916111}.
			
			\item For $\delta>0$, we denote by $\Lambda_\delta$ the first eigenvalue of the Dirichlet spherical Laplacian on $\{\sigma=(\sigma_1,\sigma_2,\ldots,\sigma_d)\in\bS^{d-1}\,:\,\sigma_1>-\cos \delta\}$.
			We also define
			\begin{align}\label{241018612}
				\lambda_{\delta}:=-\frac{d-2}{2}+\sqrt{\Big(\frac{d-2}{2}\Big)^2+\Lambda_{\delta}}\,.
			\end{align}
			In particular, when $d=2$, we set $\lambda_0=\frac{1}{2}$.
			It is worth mentioning that $\lambda_\delta=\frac{\pi}{2(\pi-\delta)}$ if $d=2$, $\lambda_\delta=\frac{\delta}{\pi-\delta}$ if $d=4$, and $\lambda_\delta>0$ for all $d\geq 3$ and $\delta>0$; for more information of $\lambda_\delta$, see \cite{BCG,FH}.
			Let $\delta\geq 0$ if $d=2$, and $\delta>0$ if $d\geq 3$, and suppose that $\Omega\subset \bR^d$ is a bounded domain satisfying the exterior $(\delta,R)$-cone condition, $0<R\leq \infty$.
			Here, the exterior $(\delta,R)$-cone condition is that for every $p\in\partial\Omega$, there exists a unit vector $e_p\in\bR^d$ such that
			\begin{align*}
				\{x\in B_R(p)\,:\,(x-p)\cdot e_p\geq |x-p|\cos\delta\}\subset \Omega^c\,. 
			\end{align*}
			(see Figure \ref{230212745} below).
			\begin{figure}[h]\centering
				\begin{tikzpicture}[>=Latex]
					\begin{scope}[shift={(-4.5,0)}]
						\clip (-1.5,-1.5) rectangle (1.5,1.5);
						\begin{scope}[name prefix = p-]
							\coordinate (A) at (0,0.7);
							\coordinate (B) at (-1.4,0.35);
							\coordinate (C) at (0,-1.4);
							\coordinate (D) at (1.4,0.35);
						\end{scope}
						
						\begin{scope}
							\draw[fill=gray!20]
							(p-A) .. controls +(-0.5,{sqrt(3)*0.5}) and +(0,0.7)..
							(p-B) .. controls +(0,-0.7) and +(-0.5,0.5)..
							(p-C) .. controls +(0.5,0.5) and +(0,-0.7) ..
							(p-D) .. controls +(0,0.7) and +(0.5,{sqrt(3)*0.5}) .. (p-A);
						\end{scope}
						
					\end{scope}
					\node[align=center] at (-4.5,-2.34) {A. Lipschitz boundary\\condition};

					\begin{scope}
						\clip (-1.5,-1.7) rectangle (1.5,1.5);
						\begin{scope}[name prefix = p-]
							\coordinate (A) at (0,0.7);
							\coordinate (B) at (-1.4,0.35);
							\coordinate (C) at (0,-1.4);
							\coordinate (D) at (1.4,0.35);
						\end{scope}
						
						\begin{scope}
							\draw[fill=gray!20]
							(p-A) .. controls +(-0.5,{sqrt(3)*0.5}) and +(0,0.7)..
							(p-B) .. controls +(0,-1) and +(0,1.4)..
							(p-C) .. controls +(0,1.4) and +(0,-1) ..
							(p-D) .. controls +(0,0.7) and +(0.5,{sqrt(3)*0.5}) .. (p-A);
						\end{scope}
						\draw[dashed] (p-C) circle (0.25);
						
					\end{scope}
					\node[align=center] at (0,-2.4) {B. Exterior\\$(\frac{\pi}{3},\infty)$-cone condition};
					\node[align=center] at (0, -3.5) {(does not satisfy Lipschitz\\boundary conddition)};
					
					\begin{scope}[shift={(4.5,0)}]
						\clip (-1.5,-1.7) rectangle (1.5,1.5);
						\begin{scope}[name prefix = p-]
							\coordinate (A) at (0,0.35);
							\coordinate (B) at (-1.4,0.45);
							\coordinate (C) at (0,-1.4);
							\coordinate (D) at (1.4,0.45);
						\end{scope}
						
						\begin{scope}
							\draw[fill=gray!20]
							(p-A) ..controls +(0,1.1) and +(0,0.9) ..
							(p-B) .. controls +(0,-1.1) and +(0,1.4)..
							(p-C) .. controls +(0,1.4) and +(0,-1.1) ..
							(p-D) .. controls +(0,0.9) and +(0,1.1) .. (p-A);
						\end{scope}
						\draw[dashed] (p-A) circle (0.25);
						\draw[dashed] (p-C) circle (0.25);
					\end{scope}
					\node[align=center] at (4.5,-2.37) {C. Exterior\\$(0,\infty)$-cone condition};
					\node[align=center] at (4.5, -3.5) {(does not satisfy $(\delta,R)$-cone \\condition, $\forall\,\,\,\delta,\,R>0$)};
					
					%
					%
					%
					%
					
				\end{tikzpicture}
				\caption{Examples for the exterior cone condition}\label{230212745}
			\end{figure}
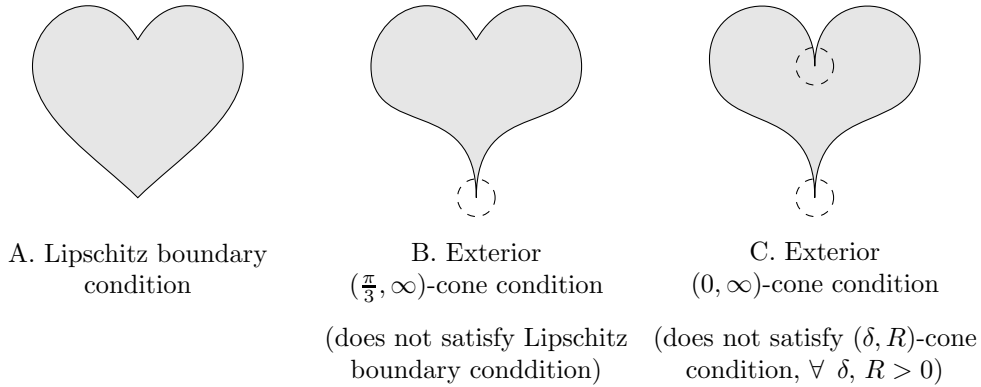
			Then $\Omega$ satisfies \textbf{LHMD}($\lambda_\delta$), where $M_{\lambda_{\delta}}$ in \eqref{21.08.03.1} depends only on $d$, $\delta$, and $\mathrm{diam}(\Omega)/R$.
			Moreover, if $\Omega$ is an unbounded domain satisfying the $(\delta,\infty)$-exterior cone condition, then $\Omega$ also satisfies \textbf{LHMD}($\lambda_\delta$).
		\end{enumerate}
	\end{example}
	
	Here is a relation between the fat exterior condition and \textbf{LHMD}($\lambda$), which is stated in \cite[Lemma 3]{AA} and \cite[Theorem 7.22]{kinnunen2021} (or, see \cite[Lemma 4.10]{Seo202404}).
	
	\begin{prop}\label{240928313}
		The following conditions are equivalent:
		\begin{enumerate}
			\item There exists $s>d-2$ such that the fat exterior condition \eqref{220617253} holds.
			
			\item There exists $\lambda>0$ such that $\mathbf{LHMD}(\lambda)$ holds.
			
			\item There exists $\epsilon_0>0$ such that
			\begin{align}\label{22.02.22.1}
				\inf_{p\in\partial\Omega,r>0}\frac{\mathrm{Cap}\big(\Omega^c\cap \overline{B}(p,r),B(p,2r)\big)}{\mathrm{Cap}\big( \overline{B}(p,r),B(p,2r)\big)}\geq \epsilon_0>0\,.
			\end{align}
			Here, $\mathrm{Cap}\big(K,B\big)$ is the capacity of a compact set $K\subset B$ relative to an open ball $B$, defined as follows: 
			\begin{align}\label{230324942}
				\mathrm{Cap}(K,B):=\inf\left\{\|\nabla f\|_2^2\,:\,f\in C_c^{\infty}(B)\,\,\,,\,\,\, f\geq 1\,\,\text{on}\,\,K\right\}\,.
			\end{align}
		\end{enumerate}
		In particular, constants $(c,s)$ in  \eqref{220617253}, $(\lambda,M_\lambda)$ in \eqref{21.08.03.1}, and $\epsilon_0$ in \eqref{22.02.22.1} depend only on another and $d$.
	\end{prop}
	
	A well-known sufficient condition for \eqref{22.02.22.1} to hold is that
	\begin{align}\label{Acondition}
		\inf_{p\in\partial\Omega,r>0}\frac{|\Omega^c\cap \overline{B}(p,r)|}{| \overline{B}(p,r)|}\geq \epsilon_1>0
	\end{align}
	(see, \textit{e.g.}, \cite[Example 6.18]{kinnunen2021}).
	For further discussion of the capacity density condition, see \cite{Kinhardy, lewis} and the references therein.
	
	In view of Proposition \ref{240928313}, we consider domains satisfying $\mathbf{LHMD}(\lambda)$ for some $\lambda>0$, instead of \eqref{220617253}.
	This assumption is implied by certain geometric conditions, as shown in Example \ref{240928314}.
	
	The following lemma is presented in \cite[Theorem 4.12]{Seo202404}:
	
	\begin{lemma}\label{21.11.08.1}
		Let $\Omega$ satisfy $\mathbf{LHMD}(\lambda)$, $\lambda>0$.
		Then for any $\lambda'\in(0,\lambda)$, there exists a superharmonic function $\phi$ in $\Omega$ such that
		$$
		N^{-1}\rho(x)^{\lambda'}\leq \phi(x)\leq N\rho(x)^{\lambda'}\,,
		$$
		where $N=N(\lambda,\lambda',M_{\lambda})>0$.
	\end{lemma}
	
	Combining Lemma \ref{21.11.08.1} with Theorem \ref{22.02.18.6}, we obtain the following result:
	
	\begin{corollary}\label{22.02.19.3}
		Let $\Omega\subset\bR^d$ satisfy $\mathbf{LHMD}(\lambda)$, $\lambda>0$, and consider the equation
		\begin{align}\label{2208131023}
			\partial_t^\alpha u=\Delta u+f\quad\text{in}\,\,(0,T]\,,
		\end{align}
		where $0<\alpha\leq 1$ and $0<T< \infty$.
		Then for any $p\in(1,\infty)$ and $\theta\in\bR$ satisfying
		\begin{align*}
			-2-(p-1)\lambda<\theta-d<-2+\lambda\,,
		\end{align*}
		and $\gamma\in\bR$,	the following statements hold:

		\begin{enumerate}
			\item 
			For any $f\in \bH^{\gamma}_{p,\theta+2p}(\Omega,T)$, equation \eqref{2208131023}
			has a unique solution $u$ in $\mathring{\cH}^{\alpha,\gamma+2}_{p,\theta}(\Omega,T)$.
			Moreover, 
			\begin{align}\label{241012333}
				\|u\|_{\mathring{\cH}^{\alpha,\gamma+2}_{p,\theta}(\Omega,T)}\leq N\| f\|_{\bH^{\gamma}_{p,\theta+2p}(\Omega,T)}\,,
			\end{align}
			where $N=N(d,p,\alpha,\gamma,\theta,\lambda,M_{\lambda})$.
			
			\item If $\frac{1}{p}<\alpha$, then for any $u_0\in B^{\gamma+2-2/(p\alpha)}_{p,\theta+2/\alpha}(\Omega)$ and $f\in  \bH^{\gamma}_{p,\theta+2p}(\Omega,T)$, equation \eqref{2208131023} with $u(0,\cdot)=u_0$
			has a unique solution $u$ in $\cH^{\alpha,\gamma+2}_{p,\theta}(\Omega,T)$.
			Moreover, 
			\begin{align}\label{2408244431111}
				\|u\|_{\cH^{\alpha,\gamma+2}_{p,\theta}(\Omega,T)}\leq N\left(\| u_0\|_{B^{\gamma+2-2/(p\alpha)}_{p,\theta+2/\alpha}(\Omega)}+\| f\|_{\bH^{\gamma}_{p,\theta+2p}(\Omega,T)}\right)\,,
			\end{align}
			where $N=N(d,p,\alpha,\gamma,\theta,\lambda,M_{\lambda})$.
			
		\end{enumerate}
	\end{corollary}
	
	\begin{proof}
		Take $\lambda'\in(0,\lambda)$ such that
		$$
		-2-(p-1)\lambda'<\theta-d<-2+\lambda'.
		$$
		Lemma \ref{240928354} and Proposition \ref{240928313} imply that $\Omega$ admits the Hardy inequality \eqref{hardy}, where $\mathrm{C}_0(\Omega)$ can be chosen to depend only on $d$, $\alpha$, and $M_{\alpha}$ (in \eqref{21.08.03.1}). 
		By Lemma \ref{21.11.08.1}, there exists a superharmonic Harnack function $\phi$, such that $\trho^{\,\lambda'}$ is a regularization of $\phi$.
		In addition, $\mathrm{C}_3(\phi,\trho^{\,\lambda'})$ depends only on $d$, $\lambda$, $\lambda'$, $M_{\lambda}$.
		Therefore, applying Theorems \ref{22.02.18.6} with $\Psi := \trho^{\,\lambda'}$ (see Proposition \ref{05.11.1}.(1) and \eqref{220526558}), the proof is complete.
	\end{proof}
	
	We next present further results for domains with a fat exterior.
	The first result is an unweighted $L_p$ solvability result when $p$ is close to $2$.
	The second result concerns the H"older regularity of solutions in domains that satisfy the capacity density condition \eqref{22.02.22.1}.

	We denote by $\mathring{W}_p^1(\Omega)$ the closure of $C_c^{\infty}(\Omega)$ in
	$$
	W_p^1(\Omega):=\big\{f\in\cD'(\Omega)\,:\,\|f\|_{W_p^1(\Omega)}:=\|f\|_p+\|\nabla f\|_p<\infty\big\}\,.
	$$

	\begin{thm}\label{230210357}
		Let $\Omega$ satisfy the capacity density condition \eqref{22.02.22.1}. 
		Then, for any $0<\nu_1\leq \nu_2<\infty$, there exists $\epsilon\in(0,1)$, depending only on $d$, $\epsilon_0$ (in \eqref{22.02.22.1}), $\nu_1$, $\nu_2$, such that the following holds:
		
		Let $\cL\in \cM_T(\nu_1,\nu_2)$ and consider the equation
		\begin{align}\label{22081310231111}
			\partial_t^\alpha u=\cL u+f^0+\sum_{i=1}^dD_if^i\quad\text{in}\quad (0,T]\quad;\quad u(0,\cdot)=0
		\end{align}
		in the sense of \eqref{240419720}.
		If $p\in(2-\epsilon,2+\epsilon)$ and $f^0,\,\ldots,\,f^d\in L_p\big((0,T]\times \Omega\big)$, then equation \eqref{22081310231111}
		has a unique solution $u$ in $L_p\big((0,T];\mathring{W}^{1}_{p}(\Omega)\big)$.
		Moreover, we have
		\begin{align}\label{230204228}
			\begin{split}
				&\|\nabla u\|_{L_p((0,T]\times \Omega)}+\min(T^{\alpha/2},D_\Omega)^{-1}\|u\|_{L_p((0,T]\times \Omega)}\\
				\lesssim_{d,p,\epsilon_0} \,&\min(T^{\alpha/2},D_\Omega)\|f^0\|_{L_p((0,T]\times \Omega)}+\sum_{i=1}^d\|f^i\|_{L_p((0,T]\times \Omega)}\,,
			\end{split}
		\end{align}
		where $D_{\Omega}:=\sup_{x\in\Omega}d(x,\partial\Omega)$.
	\end{thm}
	
	\begin{proof}
		We say that \textbf{Sol($p,\theta$)} holds if the following assertion is true:
		\begin{itemize}
			\item[] For any $\cL\in \cM_T(\nu_1,\nu_2)$ and $f\in \bH^{\gamma}_{p,\theta+2p}(\Omega,T)$,
			the equation
			$$
			\partial_t^\alpha u=\cL u+f\quad\text{in}\,\,(0,T]
			$$
			has a unique solution $u$ in $\mathring{\cH}^{\alpha,\gamma+2}_{p,\theta}(\Omega,T)$.
			Moreover, there exists a constant $N>0$, depending only on $d$, $p$, $\theta$, $\alpha$, $\nu_1$, $\nu_2$, and $\epsilon_0$ in \eqref{22.02.22.1}, such that
			$$
			\|u\|_{\mathring{\cH}^{\alpha,\gamma+2}_{p,\theta}(\Omega,T)}\leq N\| f\|_{\bH^{\gamma}_{p,\theta+2p}(\Omega,T)}\,.
			$$
		\end{itemize} 
		
		We first note the following two results for the capacity density condition:
		\begin{itemize}
			\item[(a)] By Proposition \ref{240928313} and Lemma \ref{240928354}, $\Omega$ admits the Hardy inequality \eqref{hardy}, where $\mathrm{C}_0(\Omega)$ can be chosen to depend only on $d$ and $\epsilon_0$.
			Moreover, by Lemma \ref{21.11.08.1}, there exists $\lambda'>0$ and a superharmonic function $\phi$ in $\Omega$ such that $\psi\simeq_{d,\epsilon_0}\rho^{\lambda'}$,
			where ${\lambda'}$ depends only on $d$ and $\epsilon_0$.
			By Theorem \ref{22.02.18.6} (with $\Psi=\trho^{\lambda'}$) and Proposition \ref{05.11.1}.(1), if $\theta\in\bR$ satisfies
			\begin{align}\label{230404901}
				-\frac{(p-1)\lambda'}{p\big(\sqrt{\nu_2/\nu_1}-1\big)/2+1}<\theta-d+2<\frac{(p-1)\lambda'}{p\big(\sqrt{\nu_2/\nu_1}+1\big)/2-1}\,,
			\end{align}
			then \textbf{Sol($p,\theta$)} holds.
			The first term in \eqref{230404901} goes to $-\lambda'\sqrt{\nu_1/\nu_2}$ as $p\rightarrow 2$, while the second term in \eqref{230404901} goes to $\lambda'\sqrt{\nu_1/\nu_2}$ as $p\rightarrow 2$.
			Therefore, there exists $\epsilon_1>0$ (depending only on $\nu_1$, $\nu_2$, and $\alpha$) such that if $p\in(2-\epsilon_1,2+\epsilon_1)$, then $\theta:=d-p$ satisfies \eqref{230404901}, and thus \textbf{Sol($p,d-p$)} holds.
			
			\item[(b)] It is implied by \cite[Theorem 1, Theorem 2]{lewis} (or see \cite[Theorem 3.7, Corollary 3.11]{Kinhardy}) that there exists $\epsilon_2\in(0,1)$ depending only on $d$ and $\epsilon_0$ such that for any $p>2-\epsilon_2$, if $u\in \mathring{W}_p^1(\Omega)$, then 
			\begin{align}\label{230213147}
				\int_{\Omega}\Big|\frac{u(x)}{\rho(x)}\Big|^p\dd x\leq N(d,p,\epsilon_0)\int_{\Omega}|\nabla u|^p\dd x\quad \forall\,\, u\in C_c^{\infty}(\Omega)\,.
			\end{align}
			This implies that $L_{p,d}(\Omega)\cap H_{p,d-p}^1(\Omega)\subset \mathring{W}_p^1(\Omega)$.
			Indeed, $L_{p,d}(\Omega)\cap H_{p,d-p}^1(\Omega)$ is continuously embedded in $W_p^1$, and $C_c^{\infty}(\Omega)$ is dense in $L_{p,d}(\Omega)\cap H_{p,d-p}^1(\Omega)$ (see Lemma \ref{22.04.11.3}).
		\end{itemize}
		Fix $0<\epsilon<\min(\epsilon_1,\epsilon_2)$ and consider $p\in (2-\epsilon,2+\epsilon)$.
		Then for any $p\in(2-\epsilon,2+\epsilon)$, \textbf{Sol($p,d-p$)} holds and $\mathring{W}_p^1(\Omega)\subset H_{p,d-p}^1(\Omega)$.
		
		\textbf{Step 1. Uniqueness of the solution.}
		Suppose that $u\in L_p\big((0,T];\mathring{W}_p^1(\Omega)\big)$ satisfies
		\begin{align}\label{2410061159}
			\partial_t^\alpha u=\cL u\quad \text{in}\quad (0,T]\quad;\quad u(0,\cdot)\equiv 0\,.
		\end{align}
		Since $\mathring{W}_p^1(\Omega)\subset H_{p,d-p}^1(\Omega)$ (see \eqref{230213147} and \eqref{2409098051}), we have $L_p\big((0,T];\mathring{W}_p^1(\Omega)\big)\subset \bH_{p,d-p}^1(\Omega,T)$.
		Consequently, $u\in\bH_{p,d-p}^1(\Omega,T)$ and $\partial_t^\alpha u=\cL u\in\bH_{p,d+p}^{-1}(\Omega,T)$.
		Since \eqref{240420138} holds, we have $u\in \mathring{\cH}_{p,d-p}^{\alpha,1}(\Omega,T)$.
		Since \textbf{Sol($p,d-p$)} holds, the solution of \eqref{2410061159} is unique in $\mathring{\cH}_{p,d-p}^{\alpha,1}(\Omega,T)$, and therefore $u\equiv 0$.

		\textbf{Step 2. Existence of solutions and estimate \eqref{230204228}.}
		To prove the existence of solutions, it is enough to find a solution in $\bL_{p,d}(\Omega,T)\cap \bH_{p,d-p}^1(\Omega,T)$.
		Indeed, if $u\in \bL_{p,d}(\Omega,T)\cap \bH_{p,d-p}^1(\Omega,T)$, then there exists $u_n\in C_c^{\infty}((0,T)\times \Omega)$ such that $u_n\rightarrow u$ in $\bL_{p,d}(\Omega,T)\cap \bH_{p,d-p}^1(\Omega,T)$ (see Lemma \ref{2205241011}).
		Since $L_{p,d}(\Omega)\cap H_{p,d-p}^1(\Omega)\subset \mathring{W}_p^1(\Omega)$, $(u_n)_{n\in\bN}$ is a Cauchy sequence in $L_p\big((0,T];\mathring{W}_p^1(\Omega)\big)$, so that $u=\lim_{n\rightarrow\infty}u_n\in L_p\big((0,T];\mathring{W}_p^1(\Omega)\big)$.
		Hence, $\bL_{p,d}(\Omega,T)\cap \bH_{p,d-p}^1(\Omega,T)\subset L_p\big((0,T];\mathring{W}_p^1(\Omega)\big)$.
		
		Without loss of generality, we assume that $T=1$ by dilation.
		Note that $\epsilon_0$ in \eqref{22.02.22.1} is invariant even if $\Omega$ is replaced by $r\Omega=\{rx\,:\,x\in\Omega\}$ for any $r>0$.
		
		In this step, we use \eqref{2409098051} and the inequalities 
		$$
		D_{\Omega}^{\,-1}\|u\|_p\leq \|\rho^{-1}u\|_p\simeq \|u\|_{\bL_{p,d-p}(\Omega,1)}\quad\text{and}\quad \|f\|_{\bL_{p,d+p}(\Omega,1)}\simeq \|\rho f\|_p\leq D_{\Omega}\|f\|_p
		$$
		(see \eqref{2409098051}) without mentioning. 
		Additionally, we employ the fact that
		\begin{align}\label{240112315}
			\|u\|_{\bH_{p,d-p}^{\gamma+2}(\Omega,1)}+\|u\|_{\bH_{p,d}^{\gamma+1}(\Omega,1)}+\|u\|_{\bH_{p,d+p}^\gamma(\Omega,1)}\lesssim_{\alpha,\gamma,p,d} \|u\|_{\mathring{\cH}_{p,d-p}^{\alpha,\gamma+2}(\Omega,1)}\,.
		\end{align}
		The first term of LHS in \eqref{240112315} is estimated by the definition of $\mathring{\cH}_{p,d-p}^{\alpha,\gamma+2}(\Omega,1)$.
		For the third term of LHS in \eqref{240112315}, note that for arbitrary $\widetilde{\eta}\in C_c^{\infty}\big([0,1)\big)$, by using \eqref{240419720} with  $\eta:=\int_t^1\frac{(s-t)^{-1+\alpha}}{\Gamma(\alpha)}\widetilde{\eta}(s)\dd s$, we have
		\begin{align}\label{2410071000}
			\int_0^1\la u(t),\phi\ra\widetilde{\eta}(t)\dd t=\int_0^1 I_t^\alpha\la \partial_t^\alpha u(t),\phi\ra \widetilde{\eta}(t)\dd t\,.
		\end{align}
		For RHS in \eqref{2410071000}, note that by a similar way with \eqref{240426508}, we have
		\begin{align}\label{2410071031}
			\int_0^1 \Big(I_t^{\alpha}\|\partial_t^\alpha u(t)\|_{H_{p,\theta+2p}^\gamma(\Omega)} \Big)^p\dd t\lesssim_{p,\alpha}  \|\partial_t^\alpha u\|_{\bH_{p,\theta+2p}^{\gamma}(\Omega,1)}^p\,.
		\end{align}
		Hence $I_t^{\alpha}|\la \partial_t^\alpha u(t),\phi\ra|$ is well-defined as an element of $L_p\big((0,1]\big)$ (see \eqref{220526214}).
		Since \eqref{2410071000} holds for all $\widetilde{\eta}\in C_c^{\infty}\big([0,1)\big)$, we obtain that for almost every $t\in(0,1]$, $\la u(t),\phi\ra=I_t^\alpha\la \partial_t^\alpha u(t),\phi\ra$ and
		\begin{align}\label{2410071030}
			\|u(t)\|_{H_{p,\theta+2p}^\gamma(\Omega)}\lesssim_{d,p,\theta,\gamma} I_t^\alpha\|\partial_t^\alpha u(t)\|_{H_{p,\theta+2p}^{\gamma}(\Omega)}\,,
		\end{align}
		(see \eqref{220526214}).
		\eqref{2410071031} and \eqref{2410071030} imply the estimate for the third term of LHS in \eqref{240112315}.
		The second term of LHS in \eqref{240112315} is estimated by the first and third term of LHS in \eqref{240112315} (see \eqref{241031137}).
		Therefore, \eqref{240112315} is proved.

		\textbf{Step 2.1)} Since \textbf{Sol($p,d-p$)} holds, there exists $v\in \mathring{\cH}_{p,d-p}^{\alpha,2}(\Omega,1)$ such that $\partial_t^\alpha v=\cL v+\trho^{\,-1}f^0$ and 
		\begin{align}\label{240119538}
			\|v\|_{\mathring{\cH}_{p,d-p}^{\alpha,2}(\Omega,1)}\lesssim_{d,p,\epsilon_0}\left\|\trho^{-1} f^0\right\|_{\bL_{p,d+p}(\Omega,1)}\simeq_{p,d} \|f^0\|_p\,.
		\end{align}
		By \eqref{240112315} and \eqref{240119538}, we have
		\begin{align}\label{230403637}
			\|v\|_{\bH_{p,d}^1(\Omega,T)}\lesssim \|v\|_{\mathring{\cH}_{p,d-p}^{\alpha,2}(\Omega,1)}\lesssim_{d,p,\epsilon_0} \|f^0\|_p\,.
		\end{align}
		Put 
		$$
		\widetilde{f}:=f^0-\partial_t^\alpha (\trho\,v)+\cL(\trho\,v)=2\Big[\sum_{i= 1}^da_{ij}(t)D_i\big(vD_i\trho )\Big]-v\cL \trho\,,
		$$
		and observe that
		\begin{align*}
			\begin{split}
				\big\|\widetilde{f}\big\|_{\bH_{p,d+p}^{-1}(\Omega,1)}					&\lesssim_{d,p}\|v\|_{\bL_{p,d}(\Omega,1)}\lesssim_{d,p} \|v\|_{\bH_{p,d}^1(\Omega,1)}\lesssim_{d,p,\epsilon_0}\left\|f^0\right\|_p\,,
			\end{split}
		\end{align*}
		where the first inequality follows from \eqref{220526558} and \eqref{2411011115}, while the third one follows from \eqref{230403637}.
		Since \textbf{Sol($p,\theta$)} holds, there exists $w\in \mathring{\cH}_{p,d-p}^{\alpha,1}(\Omega,1)$ such that
		\begin{align}\label{240122813}
			\begin{gathered}
				\partial_t^\alpha w=\cL w+\sum_{i= 1}^dD_if^i+\widetilde{f}\quad,\quad \text{and}\\
				\|w\|_{\mathring{\cH}_{p,d-p}^{\alpha,1}(\Omega,1)}\lesssim_{d,p,\epsilon_0}\sum_{i= 1}^d\|f^i\|_{L_{p,d}(\Omega)}+\big\|\widetilde{f}\big\|_{H_{p,d+p}^{-1}(\Omega)}\lesssim \sum_{i=0}^d\|f^i\|_{p}\,.
			\end{gathered}
		\end{align}
		Put $u=\trho\,v+w$.
		Then $u$ is a solution of equation \eqref{22081310231111} and satisfies
		\begin{alignat}{2}
			&&&\|u_x\|_p+(1+D_\Omega^{-1})\|u\|_p
			\lesssim_{d,p}\|u\|_{\bL_{p,d}(\Omega,1)}+\|u\|_{\bH_{p,d-p}^1(\Omega,1)}\nonumber\\
			&\lesssim_{d,p}\,&&\|w\|_{\bL_{p,d}(\Omega,1)}+\|w\|_{\bH_{p,d-p}^1(\Omega,1)}+\|v\|_{\bL_{p,d+p}(\Omega,1)}+\|v\|_{\bH_{p,d}^1(\Omega,1)}\label{230204149}\\
			&\lesssim_{d,p}\,&& \|w\|_{\mathring{\cH}_{p,d-p}^{\alpha,1}(\Omega,1)}+\|v\|_{\mathring{\cH}_{p,d-p}^{\alpha,2}(\Omega,1)}\lesssim_{d,p,\epsilon_0}\sum_{i\geq 0}\|f^i\|_p\,,\nonumber
		\end{alignat}
		where the last inequality follows from \eqref{230403637} and \eqref{240122813}; note that \eqref{230204149} also implies that $u\in \bL_{p,d}(\Omega,1)\cap \bH_{p,d-p}^1(\Omega,1)$.
		
		This result also completes the proof for the case $D_\Omega=\infty$.
		
		\textbf{Step 2.2)} Next, we consider the case $D_{\Omega}<\infty$.
		Observe that
		\begin{equation}	\label{230404924}
			\begin{alignedat}{2}
				\Big\|f^0+\sum_{i\geq 1}D_if^i\Big\|_{\bH_{p,d+p}^{-1}(\Omega,1)}&\lesssim_{d,p}\,&&\|f^0\|_{\bL_{p,d+p}(\Omega,1)}+\sum_{i\geq 1}\|f^i\|_{\bL_{p,d}(\Omega,1)}\\
				&\leq&& D_{\Omega}\|f^0\|_p+\sum_{i\geq 1}\|f^i\|_p<\infty\,.
			\end{alignedat}
		\end{equation}
		Since \textbf{Sol($p,d-p$)} holds, there exists $\widetilde{u}\in \mathring{\cH}_{p,d-p}^{\alpha,1}(\Omega,1)$ such that
		\begin{align}\label{230404925}
			\partial_t^\alpha \widetilde{u}=\cL \widetilde{u}+f^0+\sum_{i\geq 1}D_if^i\quad\text{and}\quad 	\|\widetilde{u}\|_{\mathring{\cH}_{p,d-p}^{\alpha,1}(\Omega,1)}\lesssim \|f^0+\sum_{i\geq 1}D_if^i\|_{\bH_{p,d+p}^{-1}(\Omega,1)}\,.
		\end{align}
		By \eqref{240112315}, \eqref{230404924}, and \eqref{230404925}, we obtain that
		\begin{equation}\label{230204150}
			\begin{aligned}
				&\|\nabla\widetilde{u}\|_{p}+D_\Omega^{-1}\|\widetilde{u}\|_{p}+\|\widetilde{u}\|_{p}\lesssim_{d,p}\|\widetilde{u}\|_{\bH_{p,d-p}^1(\Omega,1)}+ \|\widetilde{u}\|_{\bH_{p,d+p}^{-1}(\Omega,1)}\\
				\lesssim_{d,p}\,&\|\widetilde{u}\|_{\mathring{\cH}_{p,d-p}^{\alpha,1}(\Omega,1)}\lesssim_{d,p,\epsilon_0} D_{\Omega}\|f^0\|_p+\sum_{i\geq 1}\|f^i\|_p\,.
			\end{aligned}
		\end{equation}
		By \eqref{230204150}, we have $\widetilde{u}\in \bL_{p,d}(\Omega,1)\cap \bH_{p,d-p}^1(\Omega,1)$.
		
		Since $\widetilde{u}$ in Step 2.2 coincides with $u$ in Step 2.1 (by the result in Step 1), \eqref{230204228} follows from \eqref{230204149} and \eqref{230204150}.
	\end{proof}

	\begin{thm}\label{220602322}
		Let $\Omega$ be a bounded domain satisfying $\mathbf{LHMD}(\lambda)$.
		For $T\in(0,\infty)$ and $0<\alpha\leq 1$, consider the equation
		\begin{align}
			\partial_t^\alpha u=\Delta u+f_0+\sum_{i=1}^dD_if_i\quad\text{in}\quad (0,T]\times \Omega\,;\label{241010159}\\
			u(0,\cdot)=u_0\quad \text{and}\quad u|_{(0,T]\times \partial\Omega}\equiv 0\,,\label{2410101591}
		\end{align}
		where $u_0:\Omega\rightarrow \bR$ and $f_i:(0,T]\times \Omega\rightarrow \bR$, $i=0,\,1,\,\ldots,\,d$, are measurable functions.
		Here, $\partial_t^\alpha u$ in \eqref{241010159} is understood in the sense of \eqref{240419514}, while \eqref{2410101591} is interpreted pointwise.
		
		\begin{enumerate}
			\item If
			\begin{align*}
				\begin{gathered}
					I:=\sup_{x\in\Omega}\bigg(|\rho^{-\lambda}u_0|+|\rho^{-\lambda+1}\nabla u_0|\bigg)<\infty\quad\text{and}\\
					F:=\sup_{\substack{0<t\leq T\\x\in\Omega}}\bigg(|\rho^{-\lambda+2}f_0|+\sum_{i=1}^d|\rho^{-\lambda+1}f_i|\bigg)<\infty\,,
				\end{gathered}	
			\end{align*}
			then there exists a solution $u\in C([0,T]\times \overline{\Omega})$ of \eqref{241010159} and \eqref{2410101591}.
			Moreover, for any $\epsilon \in (0,1]$ and $\delta>0$, 
			\begin{align}\label{241012325}
				|u(t,x)|+\rho(x)^{1-\epsilon}\sup_{\substack{0\leq s\leq T\\y\in B(x,\rho(x)/2)}}\frac{|u(t,x)-u(s,y)|}{|x-y|^{1-\epsilon}+|t-s|^{(1-\epsilon)\alpha/2}}\leq N(F+I)\cdot \rho(x)^{\lambda-\delta}\,,
			\end{align}
			where $N$ depends only on $d$, $\alpha$, $\lambda$, $M_\lambda$, the upper bounds of $\mathrm{diam}(\Omega)$ and $T$, and $\epsilon$.
			In particular, for any $\lambda_1\in (0,\lambda)$, we have
			\begin{align}\label{2410110011}
				\sup_{\substack{t,s\in [0,T]\\x,y\in\Omega}}\frac{|u(t,x)-u(s,y)|}{|x-y|^{\lambda_1}+(t-s)^{\lambda_1\alpha/2}}\lesssim F+I\,.
			\end{align}

			\item Let $\psi$ be a bounded superharmonic Harnack function in $\Omega$.
			If there exist constants $\epsilon\in(0,1)$ and $\delta\in(0,\lambda\epsilon]$ such that 
			\begin{align*}
				\begin{gathered}
					I:=\sup_{x\in\Omega}\bigg(|\psi^{-1+\epsilon}\rho^{-\delta}u_0|+|\psi^{-1-\epsilon}\rho^{1-\delta}\nabla u_0|\bigg)<\infty\quad\text{and}\\
					F:=\sup_{\substack{0<t\leq T\\x\in\Omega}}\bigg(|\psi^{-1+\epsilon}\rho^{2-\delta}f_0|+\sum_{i=1}^d|\psi^{-1+\epsilon}\rho^{1-\delta}f_i|\bigg)<\infty\,,
				\end{gathered}	
			\end{align*}
			then there exists a solution $u\in C([0,T]\times \overline{\Omega})$ of \eqref{241010159} and \eqref{2410101591}.
			Moreover, for any $\delta_1\in[0,\delta)$,
			\begin{align}\label{241012331}
				|u(t,x)|\leq N(F+I)\cdot \psi(x)^{1-\epsilon}\rho(x)^{\delta_1}
			\end{align}
			for all $t\in[0,T]$ and $x\in\Omega$, where $N$ depends only on $d$, $\alpha$, $\lambda$, $M_\lambda$, the upper bounds of $\mathrm{diam}(\Omega)$ and $T$, $\mathrm{C}_1(\psi)$, $\epsilon$, $\delta$, $\delta_1$.
		\end{enumerate}
		
	\end{thm}
	
	\begin{proof}
		(1) \textbf{Step 1.} We first consider the case of $\epsilon=\delta\in(0,1]$, where $\epsilon$ satisfies $\frac{d+2/\alpha}{\epsilon}>\frac{2}{\lambda}-1$.
		For any $p\geq \frac{d+2/\alpha}{\epsilon}$ and $d-p\lambda \leq \theta\leq d+\lambda-2$, we have $-2-(p-1)\lambda<\theta-d<-2+\lambda$.
		In addition,
		\begin{align*}
			\begin{gathered}
				\|f\|_{\bH_{p,\theta+2p}^{-1}(\Omega,T)}\lesssim \|f_0\|_{\bL_{p,\theta+2p}(\Omega,T)}+\sum_{i=1}^d\|f_i\|_{\bL_{p,\theta+p}(\Omega,T)}\lesssim \mathrm{diam}(\Omega)^{\frac{\theta-d}{p}+\lambda}\big(T\cdot |\Omega|\big)^{1/p}F\,,\\
				\|u_0\|_{B_{p,\theta+2/\alpha}^{1-2/(\alpha p)}(\Omega)}\lesssim \|u_0\|_{H_{p_0,\theta_0}^{1}(\Omega)}\lesssim \mathrm{diam}(\Omega)^{\frac{\theta_0-d}{p_0}+\lambda}|\Omega|^{1/p_0}I\,, 
			\end{gathered}
		\end{align*}
		where $p_0:=\frac{d}{d+2/\alpha}p$ and $\theta_0-d:=\frac{d}{d+2/\alpha}(\theta-d)$ (see \eqref{241111154} and \eqref{2411011108}).
		Therefore, by Theorem \ref{22.02.18.6}, there exists a solution $u\in \cH_{p,\theta}^{1}(\Omega,T)$ of the equation $\partial_t^\alpha u=\Delta u+f$ with $u(0)=u_0$ such that
		$$
		\|u\|_{\cH_{p,\theta}^{\alpha,1}(\Omega,T)}\lesssim F+I\,.
		$$
		By the global uniqueness (Theorem \ref{220821002901}), these $u$ coincide for all $p$ and $\theta$ in the above.
		
		We first prove \eqref{241012325} in the case $t=s$.
		By taking $p:=\frac{d+2/\alpha}{\epsilon}$ and $\theta:=-p\lambda+d$, we have
		\begin{align}\label{2410110051}
			\sup_{0\leq t\leq T}|u(t)|_{0,1-\epsilon}^{(-\lambda+\epsilon)}\lesssim \sup_{0\leq t\leq T}\|u(t)\|_{B_{p,\theta+2/\alpha}^{1-2/(\alpha p)}}\lesssim \|u\|_{\cH_{p,\theta}^{\alpha,1}(\Omega,T)}\lesssim F+I\,,
		\end{align}
		where the first and second inequalities follow from \eqref{241031713} and Proposition \ref{2204160313}, respectively.
		
		Next, we prove \eqref{241012325} for the case $x=y$. 
		Take large enough $p>1$ such that $p>\frac{d+2/\alpha}{\epsilon}$.
		Put $\beta=\frac{1}{p}+(1-\epsilon)\frac{\alpha}{2}$ and $\theta=-\frac{2}{\alpha}-(\lambda-\epsilon)p$, so that $\theta-d<-2+\lambda$.
		Then by $1-2\beta/\alpha-d/p>0$ and $\theta/p+2\beta/\alpha=1-\lambda$, it follows from  \eqref{241031713} and Proposition \ref{2204160313} that
		\begin{align}\label{241031223}
			\begin{split}		|u(t)-u(s)|^{(1-\lambda)}_{0,0}\leq\,&|u(t)-u(s)|_{0,1-2\beta/\alpha-d/p}^{(\theta/p+2\beta/\alpha)}\lesssim \|u(t)-u(s)\|_{B_{p,\theta+2p\beta/\alpha}^{1-2\beta/\alpha}(\Omega)}\\
				\lesssim\,& |t-s|^{\beta-1/p}\|u\|_{\cH_{p,\theta}^{\alpha,1}(\Omega,T)}\lesssim|t-s|^{(1-\epsilon)\alpha/2}\big(F+I\big)
			\end{split}
		\end{align}
		
		Combining \eqref{2410110051} and \eqref{241031223} yields \eqref{241012325}.
		
		\textbf{Step 2.} For general $\epsilon\in(0,1]$ and $\delta>0$, take $\epsilon_0>0$ small enough such that $\epsilon_0\le\min(\epsilon,\delta)$ and $\frac{d+2/\alpha}{\epsilon_0}>\frac{2}{\lambda}-1$.
		Since $\epsilon_0\leq \delta$, we have
		$$
		|u(t,x)|\lesssim (F+I)\rho(x)^{\lambda-\epsilon_0}\leq \mathrm{diam}(\Omega)^{\delta-\epsilon_0}(F+I)\rho(x)^{\lambda-\delta}.
		$$
		Moreover, for $y\in B\big(x,\rho(x)/2\big)$, since $\frac{\rho(x)}{2}\leq \rho(y)\leq 2\rho(x)$, we have
		\begin{align*}
			&\rho(x)^{1-\epsilon}\frac{|u(t,x)-u(s,y)|}{|x-y|^{1-\epsilon}+|t-s|^{(1-\epsilon)\alpha/2}}\\
			\lesssim\,&  \big(|u(t,x)|+|u(s,y)|\big)^{(\epsilon-\epsilon_0)/(1-\epsilon_0)}\left(\rho(x)^{1-\epsilon_0}\frac{|u(t,x)-u(s,y)|}{|x-y|^{1-\epsilon_0}+|t-s|^{(1-\epsilon_0)\alpha/2}}\right)^{(1-\epsilon)/(1-\epsilon_0)}\\
			\lesssim\,&(F+I)\rho(x)^{\lambda-\epsilon_0}\leq \mathrm{diam}(\Omega)^{\delta-\epsilon_0}(F+I)\rho(x)^{\lambda-\delta}\,.
		\end{align*}
		
		\textbf{Step 3.} It remains to prove \eqref{2410110011}.
		Without loss of generality, assume that $\rho(x)\geq \rho(y)$. 
		In \eqref{241012325}, put $\epsilon=1-\lambda_1$ and $\delta=\lambda-\lambda_1$. 
		Then we obtain that
		\begin{align}\label{241031909}
			\sup_{\substack{0\leq t\leq T\\x\in\Omega}}\rho(x)^{-\lambda_1}|u(t,x)|+\sup_{\substack{t,s\in [0,T]\\x\in\Omega,|x-y|<\rho(x)/2}}\frac{|u(t,x)-u(s,y)|}{|x-y|^{\lambda_1}+(t-s)^{\lambda_1\alpha/2}}\lesssim F+I\,.
		\end{align}
		Therefore it suffices to consider the case $|x-y|>\frac{\rho(x)}{2}$.
		If $|x-y|\geq \frac{\rho(x)}{2}$, then 
		\begin{align*}
			\frac{|u(t,x)-u(s,y)|}{|x-y|^{\lambda_1}+|t-s|^{\alpha \lambda_1/2}}\lesssim_{\lambda_1} \frac{|u(t,x)|+|u(t,y)|}{\rho(x)^{\lambda_1}} \leq \frac{|u(t,x)|}{\rho(x)^{\lambda_1}}+\frac{|u(t,y)|}{\rho(y)^{\lambda_1}}\lesssim F+I\,,
		\end{align*}
		where the last inequality follows from \eqref{241031909}.

		(2) Since $0\leq \delta_1< \delta\leq \lambda \epsilon$, there exists $\lambda_1\in(0,\lambda)$, such that $\delta_1<\lambda_1\epsilon$.
		Take a sufficiently large $p$ such that 
		\begin{align}\label{2410311008}
			\frac{1}{p}<\min\Big(\frac{\lambda_1\epsilon-\delta_1}{d-2+2/\alpha+\lambda_1},\frac{\delta-\delta_1}{d+2/\alpha}\Big)\,,
		\end{align}
		and put
		$$
		\mu:=1-\epsilon+\frac{\delta_1}{\lambda_1}+\frac{d-2+2/\alpha}{\lambda_1 p}\quad,\quad r:=\frac{1-\epsilon}{\mu}\,.
		$$
		Then we have
		$$
		\mu\in(0,1-1/p\big)\,\,,\,\,\,\,r\in[0,1]\,\,,\,\,\,\,\mu r=1-\epsilon\,\,,\,\,\,\, \lambda_1\mu(1-r)=\frac{d-2+2/\alpha}{p}+\delta_1\,.
		$$
		By Theorem \ref{21.11.08.1}, there exists a superharmonic function  $s$ satisfying $s\simeq \rho^{\lambda_1}$.
		Set $\Psi:=\trho^{\,\lambda_1(1-r)}\widetilde{\psi}^{r}$, which is a regularization of the superharmonic function $s^{1-r}\psi^r$ (see  Lemma \ref{21.05.18.1}.(3)).
		Note that
		\begin{align*}
			\begin{gathered}
				\Psi^{\mu}\cH_{p,d-2}^{1}(\Omega,T)=\widetilde{\psi}^{1-\epsilon}\trho^{\lambda_1\mu(1-r)}\cH_{p,d-2}^{1}(\Omega,T)=\widetilde{\psi}^{\,1-\epsilon}\cH_{p,-2/\alpha-\delta_1p}^{1}(\Omega,T)\,,\\
				\Psi^{\mu}\bH_{p,d+2p-2}^{-1}(\Omega,T)=\widetilde{\psi}^{\,1-\epsilon}\bH_{p,2p-2/\alpha-\delta_1p}^{-1}(\Omega,T)\,,\\
				\Psi^{\mu}B_{p,d-2+2/\alpha}^{1-2/(\alpha p)}(\Omega)=\widetilde{\psi}^{1-\epsilon}B_{p,-\delta_1p}^{1-2/(\alpha p)}(\Omega)\subset \widetilde{\psi}^{1-\epsilon}H^1_{p_0,-\delta_1p_0}(\Omega)
			\end{gathered}
		\end{align*}
		(see \eqref{220526558} and \eqref{2411011108}), where $p_0=\frac{d}{d+2/\alpha}p$. \eqref{2409098051} and \eqref{241111154} imply that
		\begin{align*}
			&\big\|f^0+\sum_{i=1}^dD_if^i\big\|_{\Psi^{\mu}\bH_{p,d+2p-2}^{-1}(\Omega,T)}+\big\|u_0\big\|_{\Psi^{\mu}B_{p,d-2+2/\alpha}^{1-2/(\alpha p)}(\Omega)}\\
			\lesssim\,&\big\|f^0+\sum_{i=1}^dD_if^i\big\|_{\widetilde{\psi}^{\,1-\epsilon}\bH_{p,2p-2/\alpha-\delta_1 p}^{-1}(\Omega,T)}+\big\|u_0\big\|_{\widetilde{\psi}^{1-\epsilon}H_{p_0,-\delta_1 p_0}^{1}(\Omega)}\\
			\lesssim\,& \mathrm{diam}(\Omega)^{\delta-\delta_1-(d+2/\alpha)/p}\Big(\big(T\,|\Omega|\big)^{1/p}F+|\Omega|^{1/p_0} I\Big)\,.
		\end{align*}
		Therefore Theorem \ref{22.02.18.6} implies that there exists a solution 
		$$
		u\in \Psi^{\mu}\cH_{p,d-2}^1(\Omega,T)=\widetilde{\psi}^{1-\epsilon}\cH_{p,-2/\alpha-\delta_1p}^1(\Omega,T)
		$$ 
		of equation \eqref{241010159}.
		By the global uniqueness, Theorem \ref{220821002901}, these $u$ coincide for all $p$ and $\mu$ in the above.
		Since $1-\frac{2}{\alpha p}-\frac{d}{p}>0$ (see \eqref{2410311008}), there exists $\beta_t>0$ such that $1-\frac{2}{\alpha p}-\frac{d}{p}-\frac{2\beta_t}{\alpha}>0$.
		\eqref{241031713} and Proposition \ref{2204160313} imply that for any $0\leq s<t\leq T$,
		\begin{align*}
			\frac{\big|\widetilde{\psi}^{-1+\epsilon}\big(u(t)-u(s)\big)\big|_{0,1-2/(\alpha p)-d/p-2\beta_t/\alpha}^{(-\delta_1+2\beta_t/\alpha)}}{|t-s|^{\beta_t}}\,&\lesssim \frac{\|u(t)-u(s)\|_{\widetilde{\psi}^{(1-\epsilon)}B_{p,-\delta_1p+2p\beta_t/\alpha}^{1-2\beta_t/\alpha-2/(\alpha p)}}}{|t-s|^{\beta_t}}\\
			&\lesssim \|u\|_{\widetilde{\psi}^{1-\epsilon}\cH_{p,-2/\alpha-\delta_1 p}^1(\Omega,T)}\,.
		\end{align*}
		This shows that $u\in C([0,T]\times \Omega)$.
		Moreover, by applying \eqref{241031713} and Proposition \ref{2204160313} again, we have
		\begin{align*}
			\sup_{0\leq t\leq T}\big|\widetilde{\psi}^{-1+\epsilon}u(t)\big|_{0,1-2/(\alpha p)-d/p}^{(-\delta_1)}\,&\lesssim \|u(t)\|_{\widetilde{\psi}^{1-\epsilon}B_{p,0}^{1-2/(\alpha p)}(\Omega)} \\
			&\lesssim \|u\|_{\widetilde{\psi}^{1-\epsilon}\cH_{p,-2/\alpha-\delta_1 p}^1(\Omega,T)}\lesssim F+I
		\end{align*}
		Since
		$$
		\sup_{x\in\Omega}|\psi(x)^{-1+\epsilon}\rho(x)^{-\delta_1}u(t,x)|\simeq  |\widetilde{\psi}^{-1+\epsilon}u(t)|_{0,0}^{(-\delta_1)}\lesssim \big|\widetilde{\psi}^{-1+\epsilon}u(t)\big|_{0,1-2/(\alpha p)-d/p}^{(-\delta_1)}\,,
		$$
		\eqref{241012331} is proved.
		Since $\psi$ is bounded, taking $\delta_1>0$ yields $u(t,x)\to 0$ as $(t,x)\to [0,T]\times\partial\Omega$, which means that $u\in C([0,T]\times\overline{\Omega})$ and $u|_{[0,T]\times\partial\Omega}\equiv0$.
	\end{proof}

	\subsection{Domains with thin exterior: Aikawa dimension}\label{0062}\,
	
	Recall the definition of the Aikawa dimension.
	For a set $E\subset\bR^d$, the Aikawa dimension of $E$, denoted by $\dim_{\cA}(E)$, is defined by
	\begin{align}\label{241018630}
		\dim_{\cA}(E)=\inf\Big\{\beta\geq 0\,:\,\sup_{p\in E,\,r>0}\frac{1}{r^\beta}\int_{B_{r}(p)}\frac{1}{d(y,E)^{d-\beta}}\dd y<\infty\Big\}
	\end{align}
	with considering $\frac{1}{0}=\infty$.
	For more information on the Aikawa dimension, see Remark \ref{241110236}.
	
	The following theorem is given in \cite[Theorem 4.17]{Seo202404}:
	\begin{thm}\label{21.10.18.1}
		Suppose there exists $\beta<d-2$ such that
		\begin{align}\label{22.02.08.2}
			\sup_{p\in \Omega^c,\,r>0}\frac{1}{r^\beta}\int_{B_{r}(p)}\frac{1}{d(y,\Omega^c)^{d-\beta}}\dd y\leq A_{\beta}<\infty\,.
		\end{align}
		Then there exists a superharmonic function $\phi$ in $\Omega$ such that
		\begin{align*}
			N^{-1}\rho(x)^{-d+2+\beta}\leq \phi(x)\leq N \rho(x)^{-d+2+\beta}\,,
		\end{align*}
		where $N=N(d,\beta,A_\beta)>0$.
	\end{thm}
	
	\begin{corollary}\label{22.02.19.300}
		Let $\Omega$ satisfy $\beta_0:=\dim_{\cA}(\Omega^c)<d-2$.
		For any $p\in(1,\infty)$ and $\theta\in\bR$ satisfying
		$$
		\beta_0<\theta<(d-2-\beta)p+\beta_0\,,
		$$
		and $\gamma\in\bR$,
		the assertions in Corollary \ref{22.02.19.3} holds for this $\Omega$, $p$, $\theta$, and $\gamma$.
		Here, $N$ in \eqref{241012333} and \eqref{2408244431111} depend only on $d$, $p$, $\gamma$, $\theta$, $\beta_0$, $\{A_\beta\}_{\beta>\beta_0}$.
	\end{corollary}
	\begin{proof}[Proof of Corollary \ref{22.02.19.300}]
		By Lemma \ref{240928354}, $\Omega$ admits the Hardy inequality \eqref{hardy}, and $\mathrm{C}_0(\Omega)$ can be chosen to depend only on $d$, $\beta_0$, and $\{A_\beta\}_{\beta>\beta_0}$. 
		
		Take $\beta\in(\beta_0,d-2)$ such that
		$$
		\beta<\theta<(d-2-\beta)p+\beta\,.
		$$
		By Theorem \ref{21.10.18.1}, there exists a superharmonic function $\phi$ such that $\trho^{\,-d+2+\beta}$ is a regularization of $\phi$, and the constants $\mathrm{C}_2(\trho^{\,-d+2+\beta})$ and $\mathrm{C}_3(\phi,\trho^{\,-d+2+\beta})$ (in Definition \ref{21.10.14.1}.(3)) can be chosen to depend only on $d$, $\beta$, and $A_\beta$.
		Therefore, by Theorem \ref{22.02.18.6} for $\Psi:=\trho^{\,-d+2+\beta}$ (see  Proposition \ref{05.11.1}.(1) and \eqref{220526558}), the proof is complete.
	\end{proof}

	\subsection{Convex domains}\label{convex}\,
	
	We recall the definition of a convex set.
	A set $E\subset\bR^d$ is said to be \emph{convex} if $(1-t)x+ty\in E$ for all $x,y\in E$ and $t\in[0,1]$.
	Note that, as is well known (see \textit{e.g.} \cite[Lemma~5.8]{Seo202404}), an open set $\Omega\subset\bR^d$ is convex if and only if, for every $p\in\partial\Omega$, there exists a unit vector $e_p\in\bR^d$ such that
	\begin{align}\label{21.08.17.1}
		\Omega\subset \{x\,:\,(x-p)\cdot e_p<0\}=:U_p
	\end{align}
	(see Figure \ref{2304201156} below).
	\begin{figure}[h]\centering
		\begin{tikzpicture}[>=Latex]
			\begin{scope}
				\clip (3,-2) -- (3,1.8) -- (-3,1.8) -- (-3,-2);

				\begin{scope}
					\draw[fill=black] (2,-0.3) circle (0.05);
				\end{scope}

				\begin{scope}
					\fill[gray!30] (2,-0.3) -- (1, 1) -- (-1.5,0.7) arc(120:240:1) -- (-1,-1.2) -- (1,-1) -- (2,-0.3); 
				\end{scope}

				\begin{scope}
					\draw (2,-0.3) -- (1, 1) -- (-1.5,0.7) arc(120:240:1) -- (-1,-1.2) -- (1,-1) -- (2,-0.3); 
					\draw[line width=0.8] ($(2,-0.3)+3*(0.195,0.65)$) -- ($(2,-0.3)+{-3}*(0.24,0.8)$);
				\end{scope}
				
			\end{scope}
			
			\begin{scope}
				\draw[->] (2,-0.3) -- +(${0.5}*(1,-0.3)$);
				\draw (2.6,-0.75) node {$e_p$};
				\draw (2.8,1.1) node {$F_p$};
				\draw[white](3,0) circle(0.1);
			\end{scope}
			
		\end{tikzpicture}
		\caption{$e_p$ in \eqref{21.08.17.1}, and $F_p$ in \eqref{230114328}}\label{2304201156}
	\end{figure}

	Recall the definitions of $\mathrm{M}(\nu_1,\nu_2)$ and $\cM_T(\nu_1,\nu_2)$ at the beginning of Section \ref{0030}.
	
	\begin{thm}\label{22.02.19.5} Let $\Omega$ be a convex domain.
		For any $(\alpha^{ij})_{d\times d}\in\bigcup\limits_{0<\nu\leq 1}\mathrm{M}(\nu^2,1)$,
		$$
		\sum_{i,j=1}^d\alpha^{ij}D_{ij}\rho\leq 0
		$$
		in the sense of distributions.
	\end{thm}

	We temporarily assume Theorem \ref{22.02.19.5} in order to prove Corollary \ref{2208131026}.
	
	\begin{corollary}\label{2208131026}
		Let $\Omega\subset \bR^d$ be a convex domain, and consider the equation 
		\begin{align}\label{2409245551111}
			\partial_t^\alpha u=\cL u+f\quad\text{in}\,\, (0,T]\,,
		\end{align}
		where $0<\alpha\leq 1$, $0<T\leq \infty$, and $\cL\in\cM_T(\nu,\nu^{-1})$ with $0<\nu\leq 1$.
		Then for any $p\in(1,\infty)$ and $\theta\in\bR$ satisfying
		\begin{align*}
			-p-1<\theta-d<-1\,,
		\end{align*}
		and $\gamma\in\bR$,	the following assertion holds:			
		
		\begin{enumerate}
			\item 
			For any $f\in \bH^{\gamma}_{p,\theta+2p}(\Omega,T)$, equation \eqref{2409245551111}
			has a unique solution $u$ in $\mathring{\cH}^{\alpha,\gamma+2}_{p,\theta}(\Omega,T)$.
			Moreover, we have
			\begin{align}\label{240824442}
				\|u\|_{\mathring{\cH}^{\alpha,\gamma+2}_{p,\theta}(\Omega,T)}\leq N\| f\|_{\bH^{\gamma}_{p,\theta+2p}(\Omega,T)}\,,
			\end{align}
			where $N$ depends only on $d$, $p$, $\gamma$, $\theta$, $\nu$.
			
			\item If $\frac{1}{p}<\alpha$, then for any $u_0\in B^{\gamma+2-2/(p\alpha)}_{p,\theta+2/\alpha}(\Omega)$ and $f\in  \bH^{\gamma}_{p,\theta+2p}(\Omega,T)$, equation \eqref{2409245551111} with $u(0,\cdot)=u_0$
			has a unique solution $u$ in $\cH^{\alpha,\gamma+2}_{p,\theta}(\Omega,T)$.
			Moreover, we have
			\begin{align}\label{24082444311111}
				\|u\|_{\cH^{\alpha,\gamma+2}_{p,\theta}(\Omega,T)}\leq N\left(\| u_0\|_{B^{\gamma+2-2/(p\alpha)}_{p,\theta+2/\alpha}(\Omega)}+\| f\|_{\bH^{\gamma}_{p,\theta+2p}(\Omega,T)}\right)\,,
			\end{align}
			where $N$ depends only on $d$, $p$, $\gamma$, $\theta$, $\nu$.
			
		\end{enumerate}
		
		In particular, the domain $\Omega$ is not necessarily bounded, and the constants $N$ in \eqref{240824442} and \eqref{24082444311111} are independent of $\Omega$.
	\end{corollary}

	\begin{proof}[Proof of Corollary \ref{2208131026}]
		Since $\Omega$ is convex, it follows from Lemma \ref{240928354}, Example \ref{240928314}.(1), and Proposition \ref{240928313} that $\Omega$ admits the Hardy inequality \eqref{hardy}, where $\mathrm{C}_0(\Omega)$ can be chosen to depend only on $d$.
		
		Note that $\tilde\rho$ is a regularization of $\rho$, and the constants $\mathrm{C}_2(\tilde\rho)$ and $\mathrm{C}_3(\rho,\tilde\rho)$ (in Definition \ref{21.10.14.1}) can be chosen to depend only on $d$.
		It follows from Proposition \ref{05.11.1}.(2) that for any $\mu\in (-1/p,1-1/p)$, $\mu$ is in $ \cI(\rho,\nu^2,p)$.
		In addition, the constant $\mathrm{C}_4$ in \eqref{21.07.12.1} can be chosen to depend only on $\mu$, $p$, and $\nu$.
		Put
		$$
		\mu=-\frac{\theta-d+2}{p}\in\Big(-\frac{1}{p},1-\frac{1}{p}\,\Big)
		$$
		and apply Theorem \ref{22.02.18.6} with $\Psi:=\trho$ (see  Proposition \ref{05.11.1}.(1) and \eqref{220526558}).
		This completes the proof.
	\end{proof}

	\begin{proof}[Proof of Theorem \ref{22.02.19.5}]
		For each $p\in\partial\Omega$, define $W_p(x) := (p - x)\cdot e_p$, where $e_p$ is a unit vector satisfying \eqref{21.08.17.1}.
		We first claim that
		\begin{align}\label{22.02.17.10}
			\rho(x)=\inf_{p\in\partial\Omega}W_p(x)\quad\text{for all}\,\,\,x\in\Omega\,.
		\end{align}
		For a fixed $x\in\Omega$, we have
		$$
		\inf_{p\in\partial\Omega}W_p(x)= \inf_{p\in\partial\Omega}d(x,F_p)\geq \rho(x)\,,
		$$ 
		where 
		\begin{align}\label{230114328}
			F_p:=\{y\in\bR^d:(y-p)\cdot e_p=0\}\subset \Omega^c
		\end{align}
		(see Figure \ref{2304201156} above).
		For the inverse inequality, take $p_x\in\partial\Omega$ such that $|x-p_x|=\rho(x)$.
		Since 
		$$
		B\big(x,\rho(x)\big)\subset \Omega\quad\text{and}\quad p_x\in\partial B\big(x,\rho(x)\big)\,,
		$$
		we obtain that $e_{p_x}=\frac{p_x-x}{|p_x-x|}$.
		Therefore
		$$
		\inf_{p\in\partial\Omega}W_p(x)\leq W_{p_x}(x)=|p_x-x|=\rho(x).
		$$
		Consequently, \eqref{22.02.17.10} is proved.
		
		Let $\mathrm{A}=(\alpha^{ij})_{d\times d}\in \mathrm{M}(\nu^2,1)$, $\nu\in(0,1]$, and take $\mathrm{B}\in \mathrm{M}(\nu,1)$ such that $\mathrm{B}^2=\mathrm{A}$.
		For any $p\in\partial\Omega$, 
		$$
		\Delta \big(W_{p}(\mathrm{B}\,\cdot\,)\big)\equiv 0\quad\text{on}\,\,\mathrm{B}^{-1}\Omega\,.
		$$ 
		By \eqref{22.02.17.10} and Lemma \ref{21.05.18.1}.(2), we obtain that $\rho(\mathrm{B}\,\cdot\,)$ is an infimum of classical superharmonic functions, and therefore $\rho(\mathrm{B}\,\cdot\,)$ is a superharmonic function.
		Consequently, we have
		$$
		\langle\alpha^{ij}D_{ij}\rho,\zeta\rangle=\det(\mathrm{A})^{1/2}\big\langle\Delta (\rho(\mathrm{B}\,\cdot\,)),\zeta(\mathrm{B}\,\cdot\,)\big\rangle\leq 0
		$$
		for any $\zeta\in C_c^{\infty}(\Omega)$ with $\zeta\geq 0$.
	\end{proof}

	\subsection{Exterior Reifenberg condition}\label{ERD}\,
	
	The vanishing Reifenberg condition was introduced by Reifenberg \cite{Reifcondition} and has since been extensively studied in the literature (see, \textit{e.g.}, \cite{Relliptic, CKL, KenigToro3, TT}).
	The following definition is given in \cite{Relliptic, KenigToro3}:
	For $\delta\in(0,1)$ and $R>0$, a domain $\Omega\subset \bR^d$ is said to satisfy the $(\delta,R)$-\textit{Reifenberg condition} if for every $p\in\partial\Omega$ and $r\in(0,R]$, there exists a unit vector $e_{p,r}\in\bR^d$ such that
	\begin{align}\label{22.02.26.41}
		\begin{split}
			&\Omega\cap B_r(p)\subset \{x\in B_r(p)\,:\,(x-p)\cdot e_{p,r}<\delta r\}\quad\text{and}\\
			&\Omega\cap B_r(p)\supset \{x\in B_r(p)\,:\,(x-p)\cdot e_{p,r}>-\delta r\}\,.
		\end{split}
	\end{align}
	In addition, $\Omega$ is said to satisfy the \textit{vanishing Reifenberg condition} if for any $\delta\in(0,1)$, there exists $R_{\delta}>0$ such that $\Omega$ satisfies the $(\delta,R_{\delta})$-Reifenberg condition.
	Note that the vanishing Reifenberg condition is weaker than having $C^1$ boundary; see Example \ref{220910305}.
	
	In this subsection, we deal with the totally vanishing exterior Reifenberg condition, which was introduced in~\cite{Seo202404}.
	We obtain a result similar to Corollary \ref{2208131026} for domains satisfying the totally vanishing exterior Reifenberg condition (see Corollary \ref{22.07.17.109}).
	
	\begin{defn}[Exterior Reifenberg condition]\label{2209151117}\,\,
		
		\begin{enumerate}
			
			\item
			$\mathbf{ER}_{\Omega}$ denotes the set of all $(\delta,R)\in[0,1]\times\bR_+$ satisfying the following: for each $p\in\partial\Omega$, and each connected component $\Omega_{p,R}^{(i)}$ of $\Omega\cap B(p,R)$, there exists a unit vector $e_{p,R}^{(i)}\in\bR^d$ such that
			\begin{align*}
				\Omega_{p,R}^{(i)}\subset \{x\in B_R(p)\,:\,(x-p)\cdot e_{p,R}^{(i)}<\delta R\}\,.
			\end{align*}
			We denote by $\delta(R) := \delta_{\Omega}(R)$ the infimum of all $\delta$ such that $(\delta, R) \in \mathbf{ER}_{\Omega}$.
			
			\item For $\delta\in[0,1]$, $\Omega$ is said to satisfy the \textit{totally $\delta$-exterior Reifenberg condition} (abbreviate to `$\langle\mathrm{TER}\rangle_\delta$'), if there exist $0<R_{0,\delta}\leq R_{\infty,\delta}<\infty$ such that
			\begin{align}\label{220916111}
				\delta_{\Omega}(R)\leq \delta\quad\text{whenever}\quad R\leq R_{0,\delta}\,\,\,\text{or}\,\,\,R\geq R_{\infty,\delta}\,.
			\end{align}
			
			\item $\Omega$ is said to satisfy the \textit{totally vanishing exterior Reifenberg condition} (abbreviate to `$\langle\mathrm{TVER}\rangle$'), if
			$\Omega$ satisfies $\langle\mathrm{TER}\rangle_\delta$ condition for all $\delta\in(0,1]$. In other words,
			$$
			\lim_{R\rightarrow 0}\delta_{\Omega}(R)=\lim_{R\rightarrow \infty}\delta_{\Omega}(R)=0\,.
			$$
		\end{enumerate}
	\end{defn}
	For a comparison between the vanishing Reifenberg condition and $\langle\mathrm{TVER}\rangle$, see Figure \ref{230212856} and Example \ref{220910305} below.
	
	\begin{remark}\label{241109258}
		As shown in \cite[Lemma 5.12]{Seo202404}, for any $R>0$, $\big(\delta(R),R\big)\in\mathbf{ER}_\Omega$.
	\end{remark}
	
	In this subsection, we provide results on domains satisfying $\langle\mathrm{TER}\rangle_\delta$ for sufficiently small $\delta>0$.
	However, our main interest is the condition $\langle\mathrm{TVER}\rangle$.

	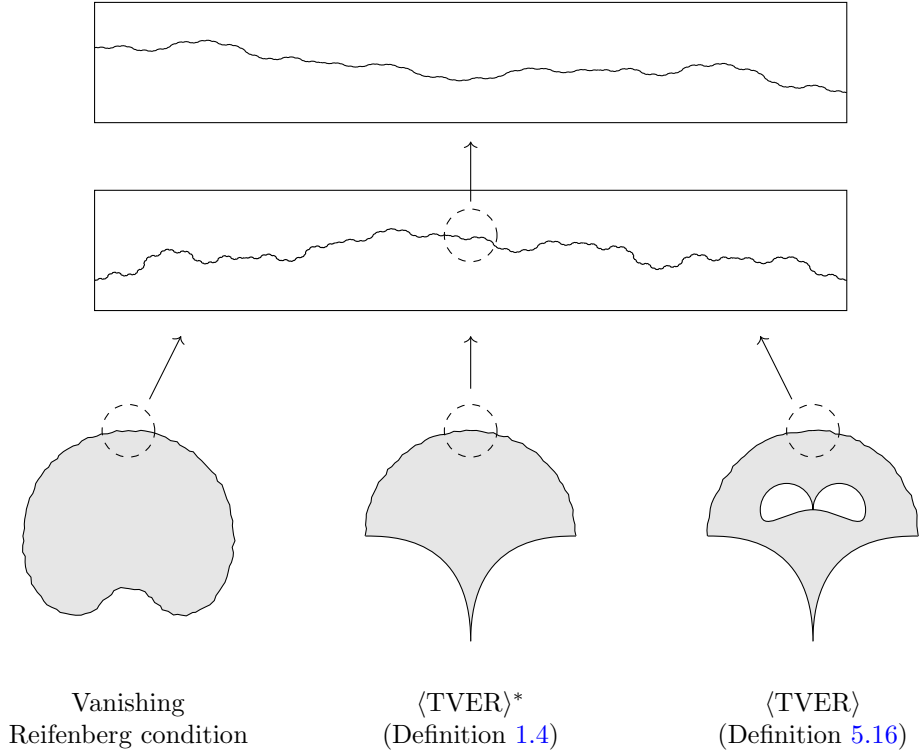
\begin{figure}[h]\centering
		\begin{tikzpicture}
			\begin{scope}[shift={(0,4.5)}]
				\begin{scope}[shift={(0,2)}]
					\draw (-5,-1) rectangle (5,0.6);
					\clip (-5,-1) rectangle (5,0.6);
					\draw[decoration={Koch, Koch angle=15, Koch order=4}] 
					decorate {(-5,-0) -- (-2,-0.2)};
					\draw[decoration={Koch, Koch angle=12, Koch order=4}] 
					decorate {(1.7,-0.3)--(-2,-0.2)};
					\draw[decoration={Koch, Koch angle=16.5, Koch order=4}] 
					decorate {(1.7,-0.3) -- (5,-0.6)};
					
				\end{scope}
				
				\begin{scope}[shift={(0,-0.5)}]
					\draw[dashed] (0,0) circle (0.35);
					\draw[->] (0,0.45) -- (0,1.25); 
					\draw (-5,-1) rectangle (5,0.6);
					\clip (-5,-1) rectangle (5,0.6);
					\draw[decoration={Koch, Koch angle=30, Koch order=4}] 
					decorate {(-5,-0.6) -- (-3,-0.3)};
					\draw[decoration={Koch, Koch angle=27.5, Koch order=4}] 
					decorate {(-1.8,-0.1)--(-3,-0.3)};
					\draw[decoration={Koch, Koch angle=21.25, Koch order=4}] 
					decorate {(-1.8,-0.1) -- (-0.3,0)};
					\draw[decoration={Koch, Koch angle=22.5, Koch order=4}] 
					decorate {(1.5,-0.1)--(-0.3,0)};
					\draw[decoration={Koch, Koch angle=28.75, Koch order=4}] 
					decorate {(3.5,-0.3)--(1.5,-0.1)};
					\draw[decoration={Koch, Koch angle=24.75, Koch order=4}] 
					decorate {(3.5,-0.3) -- (5,-0.6)};
					
					
				\end{scope}
				
			\end{scope}
			
			\begin{scope}[scale=0.7, shift={(-6.5,0)}]
				\draw[decorate, decoration={random steps,segment length=2, amplitude=0.5}, fill=gray!20]
				(2,0) arc(0:180:2) .. controls +(0,-1) and +(-0.5,0) .. (-1,-1.5)..controls +(0.5,0) and +(-0.5,0) .. (0,-1).. controls +(0.5,0) and +(-0.5,0) .. (1,-1.5)..controls +(0.5,0) and +(0,-1)..(2,0);
				
				\begin{scope}
					\draw[dashed] (0,2) circle (0.5);
					\draw[->] (0.4,2.6) -- (1,3.8);
						%
						%
						%
				\end{scope}
				
				\node[align=center] at (0,-3.5) {Vanishing\\Reifenberg condition};
			\end{scope}

			\begin{scope}[scale=0.7]
				\begin{scope}
					\draw[decorate, decoration={random steps,segment length=2, amplitude=0.5}, fill=gray!20]
					(2,0) arc(0:180:2);
					\fill[gray!20] (-2,0)--(0,-2) -- (2,0)--(0,2);
					\fill[white]
					(0,-2) ..controls +(0,2) and +(-0.5,0) .. (2,0)--(2,-2);
					\draw (0,-2) ..controls +(0,2) and +(-0.5,0) .. (2,0);
					\fill[white]
					(0,-2) ..controls +(0,2) and +(0.5,0) .. (-2,0)--(-2,-2);
					\draw (0,-2) (0,-2) ..controls +(0,2) and +(0.5,0) .. (-2,0);
				\end{scope}
				\draw[dashed] (0,2) circle (0.5);
				\draw[->] (0,2.8) -- (0,3.8); 
				
				\node[align=center] at (0,-3.5) {$\la \mathrm{TVER}\ra^*$\\(Definition \ref{221013228})};
				
			\end{scope}

			\begin{scope}[scale=0.7, shift={(6.5,0)}]
				\begin{scope}
					\draw[decorate, decoration={random steps,segment length=2, amplitude=0.5}, fill=gray!20]
					(2,0) arc(0:180:2);
					\fill[gray!20] (-2,0)--(0,-2) -- (2,0)--(0,2);
					\fill[white]
					(0,-2) ..controls +(0,2) and +(-0.5,0) .. (2,0)--(2,-2);
					\draw (0,-2) ..controls +(0,2) and +(-0.5,0) .. (2,0);
					\fill[white]
					(0,-2) ..controls +(0,2) and +(0.5,0) .. (-2,0)--(-2,-2);
					\draw (0,-2) (0,-2) ..controls +(0,2) and +(0.5,0) .. (-2,0);

					\draw[fill=white, shift={(0,0.5)}] (0,0) arc(0:180:0.5) .. controls +(0,-0.5) and +(-0.5,0) .. (0,0) .. controls +(0.5,0) and +(0,-0.5) .. (1,0) arc(0:180:0.5);
				\end{scope}
				\draw[dashed] (0,2) circle (0.5);
				\draw[->] (-0.4,2.6) -- (-1,3.8); 
				
				\node[align=center] at (0,-3.5) {$\la \mathrm{TVER}\ra$\\(Definition \ref{2209151117})};
				
			\end{scope}
		\end{tikzpicture}
		
		\caption{Totally vanishing exterior Reifenberg condition}\label{230212856}
	\end{figure}

	\begin{example}[Example 5.13 in \cite{Seo202404}]\label{220910305}
		\,\,
		
		\begin{enumerate}
			\item If $\Omega$ is bounded, then $\delta(R)\leq \frac{\mathrm{diam}(\Omega)}{R}$, and hence $\lim_{R\to\infty}\delta(R)=0$.
			
			\item
			If $\Omega$ satisfies the $(\delta,R_1)$-Reifenberg condition,
			then $\delta(R)\leq \delta$ for all $R\leq R_1$.
			In particular, if $\Omega$ satisfies the vanishing Reifenberg condition, then $\lim_{R\rightarrow 0}\delta(R)=0$.
			
			\item Suppose that $\Omega$ is bounded, and for each $p\in\partial\Omega$ there exist $R>0$ and $f\in\lambda_{\ast}(\bR^{d-1})$ such that
			$$
			\Omega\cap B(p,R)=\big\{y=(y',y_n)\in\bR^{d-1}\times \bR\,:\,|y|<R\,\,\,\text{and}\,\,\,y_n>f(y')\big\}\,,
			$$
			where $(y', y_n) = (y_1, \ldots, y_n)$ is an orthonormal coordinate system with origin at $p$, and $\lambda_\ast(\bR^{d-1})$ denotes the little Zygmund class, \textit{i.e.}, the set of all $f \in C(\bR^{d-1})$ such that
			$$
			\lim_{h\rightarrow 0}\sup_{x\in\bR^{d-1}}\frac{|f(x+h)-2f(x)+f(x-h)|}{|h|}=0\,.
			$$
			Then $\Omega$ satisfies the vanishing Reifenberg condition, and hence $\langle\mathrm{TVER}\rangle$.
			
			\item If $\Omega$ satisfies the exterior $R_0$-ball condition, then $\delta(R)\le R/(2R_0)$, and hence $\lim_{R\to 0}\delta(R)=0$.
			
			\item If $\Omega$ is the intersection of finitely many domains satisfying the totally vanishing Reifenberg condition, then $\Omega$ also satisfies $\langle\mathrm{TVER}\rangle$.
		\end{enumerate}
	\end{example}
	
	A sufficient condition for $\lim_{R\to\infty}\delta_{\Omega}(R)=0$ is $\delta_{\Omega}(R)\lesssim 1/R$, and this holds if $\Omega$ is bounded.
	It was shown in \cite[Proposition~5.14]{Seo202404} that, for a general domain $\Omega$,
	\begin{align}\label{241025520}
		\delta_{\Omega}(R)\lesssim 1/R\quad\text{if and only if}\quad \sup_{p\in\partial\Omega}d\big(p,\partial(\Omega_{\mathrm{c.h.}})\big)<\infty\,,
	\end{align}
	where $\Omega_{\mathrm{c.h.}}$ is the convex hull of $\Omega$, \textit{i.e.}, 
	$$
	\Omega_{\mathrm{c.h.}}=\{tx+(1-t)y\,:\,x,\,y\in\Omega\,,\,\,t\in[0,1]\}\,.
	$$
	
	We now state the main result of this subsection.
	We temporarily assume Theorem~\ref{22.02.19.6} (its proof will be given at the end of this subsection) in order to prove Corollary~\ref{22.07.17.109}.
	
	\begin{thm}\label{22.02.19.6}
		For any $\nu\in(0,1]$ and $\epsilon\in(0,1)$, there exists $\delta_1>0$ depending only on $d,\,\epsilon,\,\nu$ such that, if 
		$\Omega$ satisfies $\langle\mathrm{TER}\rangle_{\delta_1}$,
		then there exists a measurable function $\phi:\Omega\rightarrow \bR$ such that the following hold:
		\begin{enumerate}
			\item For any $(\alpha^{ij})_{d\times d}\in\mathrm{M}(\nu^2,1)$, $\sum_{i,j=1}^d\alpha^{ij}D_{ij}\phi\leq 0$ in the sense of distributions.
			
			\item There exists $N=N(d,\nu,\epsilon,R_{0,\delta}/R_{\infty,\delta})>0$ such that
			$$
			N^{-1}\rho(x)^{1-\epsilon}\leq \phi(x)\leq N\rho(x)^{1-\epsilon}\quad\text{for all}\,\,\,x\in\Omega\,,
			$$
			where $R_{0,\delta}$ and $R_{\infty,\delta}$ are the constants in \eqref{220916111}.
		\end{enumerate}
	\end{thm}

	\begin{corollary}\label{22.07.17.109}
		Consider the equation 
		\begin{align*}
			\partial_t^\alpha u=\cL u+f\quad\text{in}\,\, (0,T]\,,
		\end{align*}
		where $0<\alpha\leq 1$, $0<T\leq \infty$, and $\cL\in\cM_T(\nu,\nu^{-1})$ with $0<\nu\leq 1$.
		Then, for any $p\in(1,\infty)$ and $\theta\in\bR$ satisfying
		\begin{align*}
			-p-1<\theta-d<-1\,,
		\end{align*}
		there exists $\delta>0$, depending only on $d$, $p$, $\theta$, and $\nu$, such that the following holds: 
		\begin{itemize}
			\item[]  If $\Omega$ satisfies $\langle\mathrm{TER}\rangle_\delta$, then the assertions in Corollary~\ref{2208131026} hold for this $\Omega$, $p$, $\theta$, and any $\gamma>0$.
			In addition, $N$ in \eqref{240824442} and \eqref{24082444311111} depend only on $d$, $p$, $\gamma$, $\theta$, $\nu$, $R_{0,\delta}/R_{\infty,\delta}$.
			Here, $R_{0,\delta}$ and $R_{\infty,\delta}$ are the constants in \eqref{220916111}.
		\end{itemize}
	\end{corollary}
	\begin{proof}[Proof of Corollary \ref{22.07.17.109}]
		Take $\epsilon\in(0,1)$ such that
		$$
		-p-1+(p-1)\epsilon<\theta-d<-1-\epsilon\,.
		$$
		For this $\epsilon$ and any $\nu\in(0,1]$, let $\delta_1>0$ denote the constant in Theorem~\ref{22.02.19.6}.
		It was shown in \cite[Theorem~5.17, Lemma~4.10]{Seo202404} that there exists $\delta_2>0$, depending only on $d$, such that, if $\Omega$ satisfies $\langle\mathrm{TER}\rangle_\delta$, then $\Omega$ admits the Hardy inequality~\eqref{hardy}.
		Here, $\mathrm{C}_0(\Omega)$ depends only on $d$, $\delta$, and $R_{0,\delta}/R_{\infty,\delta}$.
		
		Put $\delta=\delta_1\wedge \delta_2>0$.
		Then $\Omega$ admits the Hardy inequality.
		Let $\phi$ be the function in Theorem \ref{22.02.19.6}.
		By Proposition \ref{05.11.1}.(2), we obtain that for
		$$
		\mu:=-\frac{\theta-d+2}{p(1-\epsilon)}\in\Big(-\frac{1}{p},1-\frac{1}{p}\Big)\,,
		$$
		we have $\mu\in\cI(\phi,\nu^2,p)$.
		In addition, for this $\mu$, $\mathrm{C}_4$ in \eqref{21.07.12.1} can be chosen to depend only on $\mu$, $\nu$, and $p$.
		Note that  $\trho^{\,1-\epsilon}$ is a regularization of $\phi$, and 
		the constants $\mathrm{C}_2(\trho^{\,1-\epsilon})$ and $\mathrm{C}_3(\phi,\trho^{\,1-\epsilon})$ (in Definition \ref{21.10.14.1}) can be chosen to depend only on $d$, $\epsilon$, $\nu$ and $R_{0,\delta}/R_{\infty,\delta}$.
		By applying Theorem~\ref{22.02.18.6} with $\Psi := \trho^{,1-\epsilon}$ (see Proposition \ref{05.11.1}.(1) and \eqref{220526558}), this completes the proof.
	\end{proof}
	
	To prove Theorem \ref{22.02.19.6}, we first state the following lemma:
	
	\begin{lemma}\label{21.08.24.1}
		Suppose that $(\delta,R)\in\mathbf{ER}_{\Omega}$.
		For any $\nu\in(0,1)$ and $p\in\partial\Omega$, there exists a continuous function $w_{p,R}:\Omega\rightarrow (0,1]$ satisfying the following:
		\begin{enumerate}
			\item For any $\mathrm{B}\in\mathrm{M}(\nu,1)$, $w_{p,R}(\mathrm{B}\,\cdot\,)$ is a classical superharmonic function in $\mathrm{B}^{-1}\Omega$.
			
			\item $w_{p,R}=1$ on $\{x\in\Omega\,:\,|x-p|>(1-\delta)R\}$.
			\item $w_{p,R}\leq M\delta$ in $\Omega\cap B(p,\delta R)$.
		\end{enumerate}
		Here, $M$ depends only on $\nu$ and $d$, and in particular, it is independent of $\delta$.
	\end{lemma}
	\begin{proof}
		The proof of this lemma is almost the same as that of \cite[Lemma 5.20]{Seo202404}, by setting $F_0(t)=1-t^{2-d/\nu^2}$ instead of \cite[(5.21)]{Seo202404}.
		We leave the complete proof to the reader.
	\end{proof}

	\begin{proof}[Proof of Theorem \ref{22.02.19.6}]
		It suffices to consider the case $\nu\in(0,1)$.
		Let $M>0$ be the constant in Lemma \ref{21.08.24.1}.
		For a fixed $\epsilon\in (0,1)$, choose $\delta\in(0,1)$ sufficiently small such that $M\delta < \delta^{1-\epsilon}$, and then choose $\eta\in(0,1)$ sufficiently small such that
		$$
		(1-\eta)M\delta+\eta\leq \delta^{1-\epsilon}\,.
		$$
		We assume that $\Omega$ satisfies \eqref{220916111} for this $\delta$.
		By scaling and using the fact that $(\delta(R),R)\in\mathbf{ER}_\Omega$ (see Remark \ref{241109258}), we may assume without loss of generality that $R_{\infty,\delta}=1$, and that $(\delta,R)\in\mathbf{ER}_{\Omega}$ whenever $R\leq R_{0,\delta}(\leq 1)$ or $R\geq 1$.

		\textbf{Step 1.}
		Put
		$$
		k_0=\min\big\{k\in\bN\,:\,\delta^k\leq R_{0,\delta}\big\}\quad\text{and}\quad \cI=\{k\in\bZ\,:\,k\leq 0\,\,\,\text{or}\,\,\, k\geq k_0\}\,,
		$$
		so that $(\delta,\delta^k)\in\mathbf{ER}_\Omega$ for every $k\in\cI$.
		For each $p\in\partial\Omega$ and $k\in\cI$, put
		\begin{align*}
			\phi_{p,k}=\delta^{k(1-\epsilon)}\Big((1-\eta)\, w_{p,\delta^k}+\eta\Big)\,,
		\end{align*}
		where $w_{p,\delta^k}$ is the function $w_{p,R}$ in Lemma \ref{21.08.24.1} with $R=\delta^k$.
		Note that
		\begin{alignat*}{2}
			&\qquad \phi_{p,k}(x)\leq \delta^{(k+1)(1-\epsilon)}&&\text{on}\quad \Omega\cap \overline{B}(p,\delta^{k+1})\,;\\
			&\qquad \phi_{p,k}(x)=\delta^{k(1-\epsilon)} &&\text{on}\quad \Omega\cap \partial B(p,\delta^k)\,;\\
			&\eta\cdot \delta^{k(1-\epsilon)}\leq \phi_{p,k}\leq \delta^{k(1-\epsilon)}\qquad&&\text{on}\quad \Omega\cap B(p,\delta^k)\,.
		\end{alignat*}
		Put
		\begin{alignat*}{2}
			&\phi_p^{(1)}(x):=\inf\{\phi_{p,k}(x)\,:\,k\geq k_0\,\,,\,\,|x-p|< \delta^k\}\quad&&\text{for}\quad |x-p|< \delta^{k_0}\,;\\
			&\phi_p^{(2)}(x):=\inf\{\phi_{p,k}(x)\,:\,k\leq 0\,\,,\,\,|x-p|< \delta^k\}\quad&&\text{for}\quad |x-p|>\delta\,.
		\end{alignat*}
		In this step, we first claim that, for any $\mathrm{B}\in\mathrm{M}(\nu,1)$, $\phi_p^{(1)}(\mathrm{B}\,\cdot\,)$ and $\phi_p^{(2)}(\mathrm{B}\,\cdot\,)$ are classical superharmonic functions in
		$$
		\{\mathrm{B}^{-1}x\,:\,x\in\Omega\cap B(p,\delta^{k_0})\}\quad\text{and}\quad \{\mathrm{B}^{-1}x\,:\,x\in \Omega\setminus \overline{B}(p,\delta)\}\,,
		$$
		respectively.
		The second claim is that for each $i=1,\,2$, $\phi_p^{(i)}(x)$ satisfies
		\begin{align}\label{2209161228}
			\eta |x-p|^{1-\epsilon}\leq \phi_p^{(i)}(x)\leq \delta^{-1+\epsilon}|x-p|^{1-\epsilon}
		\end{align}
		on its domain.
		By similarity, we only provide the proof for $\phi^{(1)}_p$.
		
		For the first claim, it suffices to prove that, for each $\mathrm{B}\in\mathrm{M}(\nu,1)$ and $k_1\in\bZ$ with $k_1\geq k_0$, the function $\phi_p^{(1)}(\mathrm{B}\,\cdot\,)$ is a classical superharmonic function in 
		$$
		U:=\{\mathrm{B}^{-1}x\in\Omega\,:\,\delta^{k_1+2}<|x-p|<\delta^{k_1}\}\,.
		$$
		For $x\in U$, put
		\begin{align*}
			v_{p,k_1}(x)=
			\begin{cases}
				\phi_{p,k_1}(x) &\text{if}\quad \delta^{k_1+1}\leq |x-p|< \delta^{k_1}\\
				\phi_{p,k_1}(x)\wedge \phi_{p,k_1+1}(x)&\text{if}\quad \delta^{k_1+2}<|x-p|< \delta^{k_1+1}\,.
			\end{cases}
		\end{align*}
		Since $\phi_{p,k_1}\leq \phi_{p,k_1+1}$ on $\Omega\cap \partial B(p,\delta^{k_1+1})$, Lemma \ref{21.05.18.1}.(4) implies that $v_{p,k_1}(\mathrm{B}\,\cdot\,)$ is a classical superharmonic function in $U$.
		Observe that
		$$
		\phi_p^{(1)}(x)=v_{p,k_1}(x)\wedge \inf\{\phi_{p,k}(x)\,:\,k_0\leq k\leq k_1-1\}
		$$
		on $U$.
		If $\eta\,\delta^{k(1-\epsilon)}\geq \delta^{k_1(1-\epsilon)}$ then
		$$
		v_{p,k_1}(x)\leq \phi_{p,k_1}(x)\leq \delta^{k_1(1-\epsilon)}\leq \eta\,\delta^{k(1-\epsilon)}\leq \phi_{p,k}(x)\,.
		$$
		Therefore we have
		\begin{align*}
			\phi_p^{(1)}(\mathrm{B}\,\cdot\,)=v_{p,k_1}(\mathrm{B}\,\cdot\,)\wedge \inf\{\phi_{p,k}(\mathrm{B}\,\cdot\,)\,:\, k_0\leq k\leq k_1\quad\text{and}\quad \eta\, \delta^{k(1-\epsilon)}\leq \delta^{k_1(1-\epsilon)}\}\,.
		\end{align*}
		This implies that on $U$, $\phi_p^{(1)}(\mathrm{B}\,\cdot\,)$ is the minimum of finitely many classical superharmonic functions.
		It follows from Lemma \ref{21.05.18.1}.(1) that $\phi_p^{(1)}(\mathrm{B}\,\cdot\,)$ is a classical superharmonic function in $U$.
		
		For the second claim, let $x\in\Omega$ satisfy $\delta^{k_1+1}\leq |x-p|< \delta^{k_1}$ for $k_1\in\bZ$ with $k_1\geq k_0$.
		Since $\phi^{(1)}_{p,k_1}(x)\leq \delta^{k_1(1-\epsilon)}$, and 
		$\phi^{(1)}_{p,k}(x)\geq \eta \delta^{k(1-\epsilon)}\geq \eta \delta^{k_1(1-\epsilon)}$ for all $k_0\leq k\leq k_1$,
		we obtain that $\eta\,\delta^{k_1(1-\epsilon)}\leq \phi_p^{(1)}(x)\leq \delta^{k_1(1-\epsilon)}$.
		This implies \eqref{2209161228}.

		\textbf{Step 2.}
		Observe that
		\begin{equation}\label{230105945}
			\begin{alignedat}{2}
				&\phi_p^{(1)}(x)\leq \phi_{p,k_0}(x)\leq \delta^{(k_0+1)(1-\epsilon)}\quad&&\text{if}\,\,\,|x-p|=\delta^{k_0+1}\,\,;\\
				&\phi_p^{(1)}(x)=\phi_{p,k_0}(x)=\delta^{k_0(1-\epsilon)}\quad &&\text{if}\,\,\,|x-p|=\delta^{k_0}\,.
			\end{alignedat}
		\end{equation}
		Put $\gamma=-\nu^{-2}d+2<0$ and take $\alpha_1,\,\beta_1\in\bR$ such that the function
		$$
		f(t):=\alpha_1-\beta_1t^{\gamma}
		$$
		satisfies
		\begin{align}\label{23010594511}
			f(\delta^{k_0+1})=\delta^{(k_0+1)(1-\epsilon)}\quad\text{and}\quad f(\delta^{k_0})=\delta^{k_0(1-\epsilon)}\,.
		\end{align}
		Since $f(\delta^{k_0+1})<f(\delta^{k_0})$, we have $\beta_1>0$. 
		Because $\gamma=-\nu^{-2}d+2<0$ and $\beta_1>0$, we have the following:
		\begin{itemize}
			\item For any $(\alpha^{ij})_{d\times d}\in\mathrm{M}(\nu^2,1)$, 
			$$
			\sum_{i,j=1}^d\alpha_{ij}D_{ij}\Big(f\big(|\cdot-p|\big)\Big)\leq 0\quad \text{in}\quad \bR^d\setminus \{p\}\,.
			$$
			
			\item $f(t)$ increases as $t\rightarrow \infty$.
			In particular, $f(t)\geq \delta^{(k_0+1)(1-\epsilon)}$ for all $t\geq \delta^{k_0+1}$.
		\end{itemize}
		
		Take $\alpha_2>0,\,\beta_2\in\bR$ such that
		\begin{align}\label{2301051008}
			\alpha_2\eta\delta^{1-\epsilon}+\beta_2=f(\delta)\quad\text{and}\quad \alpha_2\delta^{-1+\epsilon}+\beta_2=f(1)\,,
		\end{align}
		and put $\widetilde{\phi}_p^{(2)}:=\alpha_2\phi_p^{(2)}+\beta_2$.
		We define
		\begin{align}\label{230105954}
			\phi_p=
			\begin{cases}
				\vspace{1mm}
				\,\,\, \phi_p^{(1)}&\quad \text{on}\,\,\,\big\{x\in\Omega\,:\,|x-p|\leq \delta^{k_0+1}\big\}\\
				\vspace{1mm}
				\,\,\,\phi_p^{(1)}\wedge f(|\,\cdot\,-p|)&\quad \text{on}\,\,\,\big\{x\in\Omega\,:\,\delta^{k_0+1}<|x-p|< \delta^{k_0}\big\}\\
				\vspace{1mm}
				\,\,\,f(|\,\cdot\,-p|)&\quad \text{on}\,\,\,\big\{x\in\Omega\,:\,\delta^{k_0}\leq |x-p|\leq \delta\big\}\\
				\vspace{1mm}
				\,\,\,f(|\,\cdot\,-p|)\wedge \widetilde{\phi}_p^{(2)} &\quad \text{on}\,\,\,\big\{x\in\Omega\,:\,\delta<|x-p|< 1\big\}\\
				\vspace{1mm}
				\,\,\,\widetilde{\phi}_p^{(2)}&\quad \text{on}\,\,\,\big\{x\in\Omega\,:\,|x-p|\geq 1\big\}\,.
			\end{cases}
		\end{align}
		
		Observe that 
		\begin{align*}
			\text{$\phi_p^{(1)}(x)\leq f(|x-p|)$\quad if $|x-p|=\delta^{k_0+1}$}\quad\text{and}\quad \text{$\phi_p^{(1)}(x)\geq f(|x-p|)$\quad if $|x-p|=\delta^{k_0}$}
		\end{align*}
		(see \eqref{230105945} and \eqref{23010594511}), and 
		\begin{align*}
			\text{$f(|x-p|)\leq \widetilde{\phi}_p^{(2)}$\quad if $|x-p|=\delta$}\quad\text{and}\quad \text{$f(|x-p|)\geq \widetilde{\phi}_p^{(2)}$\quad if $|x-p|=1$}
		\end{align*}
		(see \eqref{2209161228} and \eqref{2301051008}).
		By Lemma \ref{21.05.18.1}.(4), it follows that for any $\mathrm{B}\in\mathrm{M}(\nu,1)$, $\widetilde{\phi}_p(\mathrm{B}\,\cdot\,)$ is a classical superharmonic function in $\mathrm{B}^{-1}\Omega$.

		\textbf{Step 3.}
		We claim that for every $x\in\Omega$,
		\begin{align}\label{2209161229}
			N^{-1}|x-p|^{1-\epsilon}\leq \phi_p(x)\leq N|x-p|^{1-\epsilon}\,,
		\end{align}
		where $N=N(d,\epsilon,\nu,R_{0,\delta})>0$.
		Recall the definition of $\phi_p$ in \eqref{230105954}.
		
		\textbf{Step 3.1)} By \eqref{2209161228}, 
		\begin{align}\label{230105957}
			\eta|x-p|^{1-\epsilon}\leq \phi_p^{(1)}(x)\leq \delta^{-1+\epsilon}|x-p|^{1-\epsilon}\quad\text{on}\quad \big\{x\in\Omega\,:\,|x-p|<\delta^{k_0}\big\}\,.
		\end{align}
		
		\textbf{Step 3.2)} Since $f(t)=\alpha_1-\beta_1t^{\gamma}$, $\beta_1>0$, and $\gamma<-d+2\leq 0$,
		we have
		\begin{align}\label{2301059571}
			\delta^{(k_0+1)(1-\epsilon)}=f(\delta^{k_0+1})\leq f\big(|x-p|\big)\leq f(1)\quad\text{if}\quad \delta^{k_0+1}<|x-p|< 1\,.
		\end{align}
		
		\textbf{Step 3.3)} Note that $\widetilde{\phi}_p^{(2)}=\alpha_2\phi_p^{(2)}+\beta_2$ and $\alpha_2>0$. 
		Take $K\geq 1$ such that 
		$$
		\alpha_2\eta K^{1-\epsilon}\geq 2|\beta_2|\,.
		$$
		For $x\in \Omega$ satisfying $\delta<|x-p|<K$, it follows from \eqref{2209161228}, \eqref{2301051008}, and \eqref{2301059571} that
		\begin{equation}\label{2301048512}
			\begin{aligned}
				\delta^{(k_0+1)(1-\epsilon)}\leq f(\delta)=\alpha_2\eta\delta^{1-\epsilon}+\beta_2\leq \widetilde{\phi}_p^{(2)}(x)\leq \alpha_2\delta^{-1+\epsilon}K^{1-\epsilon}+\beta_2\,.
			\end{aligned}
		\end{equation}
		If $|x-p|\geq K$, then 
		$$
		2|\beta_2|\leq \alpha_2\eta K^{1-\epsilon}\leq \alpha_2\eta|x-p|^{1-\epsilon}\,.
		$$
		By \eqref{2209161228}, we have
		\begin{align}\label{2301048513}
			\frac{\eta\alpha_2}{2}|x-p|^{1-\epsilon}\leq \widetilde{\phi}_p^{(2)}(x)\leq \alpha_2\big(\delta^{-1+\epsilon}+\frac{\eta}{2}\,\big)|x-p|^{1-\epsilon}\,.
		\end{align}
		
		Since $k_0,\eta,\,\alpha_1\,\beta_1,\,\alpha_2,\,\beta_2,\,K$ depend only on $d,\,\nu,\,\epsilon,\,\delta,\,R_{0,\delta}$, 
		\eqref{230105954}--\eqref{2301048513} imply \eqref{2209161229}.

		\textbf{Step 4.}
		Put $\phi(x):=\inf_{p\in\partial\Omega}\phi_{p}(x)$.
		Then 
		$$
		N^{-1}\rho(x)\leq \phi(x)\leq N\rho(x)\,,
		$$
		where $N$ is the same constant as in \eqref{2209161229}.
		For any fixed $\mathrm{B}\in \mathrm{M}(\nu,1)$, because $\phi_p(\mathrm{B}\,\cdot\,)$ is a classical superharmonic function in $\mathrm{B}^{-1}\Omega$, Lemma \ref{21.05.18.1}.(2) yields that $\phi(\mathrm{B}\,\cdot\,)$ is superharmonic on $\mathrm{B}^{-1}\Omega$.
	\end{proof}

	\subsection{Conic domains}\label{0074}\,
	
	$\bS^{d-1}$ denotes the $(d-1)$-dimensional unit sphere in $\bR^d$, that is, $\{x\in\bR^d\,:\,|x|=1\}$.
	We denote by $A_{\bS}$ the surface measure on $\bS^{d-1}$.
	Note that for any nonnegative Borel function $F$ on $\bR^{d}\setminus\{0\}$,
	$$
	\int_{\bR^{d}\setminus\{0\}} F(x)\dd x=\int_0^{\infty}\left(\int_{\bS^{d-1}}F(r\sigma)\dd A_{\bS}(\sigma)\right)r^{d-1}\dd r\,.
	$$
	Let $\cM$ be a connected and relatively open subset of $\bS^{d-1}$, and define
	\begin{align*}
		\Omega=\big\{x\in\bR^d\setminus\{0\}\,:\,\frac{x}{|x|}\in\cM\big\}
	\end{align*}
	which is the conic domain generated by $\cM$ (see Figure \ref{230222651} below). 
	\begin{figure}[ht]
		\begin{tikzpicture}[> = Latex, scale=0.8]
			
			\begin{scope}
				
				\begin{scope}[shift={(-6,0)}, scale=0.8]
					\fill[gray!20] (0,0)--(-1.65,-2.0)--(-2.0,-2.0)--(2.0,-2.0)--(2.0,2.0)--(-2.0,2.0)--(-1.65,2.0)--(0,0);
					
					\clip (-2.2,-2.25)--(2.4,-2.25)--(2.4,2.4)--(-2.4,2.4);	
					\draw[->] (-2.2,0) -- (2.4,0);
					\draw[->] (0,-2.2) -- (0,2.4);
					\draw (0,0) -- (-1.69125,2.05);
					\draw (0,0) -- (-1.69125,-2.05);
					\begin{scope}
					\end{scope}
				\end{scope}

				\begin{scope}[shift={(0,-.2)}]
					\clip (-2.2,-2)--(2.4,-2)--(2.4,2.4)--(-2.4,2.4);	
					\begin{scope}[shift={(-1.3,-1.4)}, scale=1.2]
						
						\begin{scope}[scale=0.9]
							\clip (0,0) circle (3.6) ;
							\clip (-3,-3) -- (3,-3) -- (3,3) -- (-3,3);
							\draw[dashed] (0,0) -- (4.4,4.4);
							\draw[dashed] (0,0) -- (2.8,5.2);
							\draw (0,0) -- (4.4,2.2);
							\draw (0,0) -- (1.2,4.8);
						\end{scope}

						\begin{scope}[scale=1.6]
							\draw[fill=gray!25]
							(1.1,0.55) .. controls (1.4,0.7) and (1.2,1) .. (1.1,1.1) .. controls (1.0,1.2) and (0.8,1.25) .. (0.7,1.3) .. controls (0.5,1.4) and (0.33,1.32) .. (0.3,1.2) .. controls (0.23,0.92) and (0.5,0.85) .. (0.6,0.7) .. controls (0.7,0.55) and (1,0.5) .. (1.1,0.55);
						\end{scope}

						\begin{scope}[scale=1.6]
							\draw (1.1,0.55) .. controls (1.4,0.7) and (1.2,1) .. (1.1,1.1) .. controls (1.0,1.2) and (0.8,1.25) .. (0.7,1.3) .. controls (0.5,1.4) and (0.33,1.32) .. (0.3,1.2);
						\end{scope}
						
						\begin{scope}[scale=1.6]
							\draw (0.3,1.2) .. controls (0.23,0.92) and (0.5,0.85) .. (0.6,0.7) .. controls (0.7,0.55) and (1,0.5) .. (1.1,0.55);
						\end{scope}
						
						\begin{scope}[scale=1.6]
							\clip (1.1,0.55) .. controls (1.4,0.7) and (1.2,1) .. (1.1,1.1) .. controls (1.0,1.2) and (0.8,1.25) .. (0.7,1.3) .. controls (0.5,1.4) and (0.33,1.32) .. (0.3,1.2) .. controls (0.23,0.92) and (0.5,0.85) .. (0.6,0.7) .. controls (0.7,0.55) and (1,0.5) .. (1.1,0.55);

							%
							\draw[gray] (0,1.7) arc(90:0:0.54 and 1.95);
							\draw[gray] (0,1.7) arc(90:0:0.795 and 1.92);
							\draw[gray] (0,1.7) arc(90:0:1.03 and 1.85);
							\draw[gray] (0,1.7) arc(90:0:1.23 and 1.787);
							\draw[gray] (0,1.7) arc(90:0:1.4 and 1.71);
							
							\draw[gray] (-0.15,1.14) arc(-90:0:1.7*0.565 and 0.6*0.565);
							\draw[gray] (-0.15,1.00) arc(-90:0:1.7*0.66 and 0.6*0.66);
							\draw[gray] (-0.15,0.84) arc(-90:0:1.7*0.745 and 0.6*0.745);
							\draw[gray] (-0.15,0.675) arc(-90:0:1.7*0.82 and 0.6*0.82);
							\draw[gray] (-0.15,0.5) arc(-90:0:1.7*0.89 and 0.6*0.89);
						\end{scope}
						
						\draw[thin] (0,0) -- (1.8,2.1);
					\end{scope}
				\end{scope}

				\begin{scope}[shift={(4.7,-1.7)},scale=1.2]
					\clip (-.2,-.2) -- (2.5,-.2) -- (2.5,2.7) -- (-.2,2.7);
					\clip (0,0) circle (3.0) ;
					\begin{scope}
						
						\draw (0,0) -- (4.4,2.2); 
						\draw[dashed] (0,0) -- (5,4);
						\draw[dashed] (0,0) -- (3,5);
						\draw (0,0) -- (1.3,4.7);
						\draw (0,0) -- (1.70,2.20);
					\end{scope}
					
					
					\begin{scope}[scale=1.6]
						\draw[line width=0.3mm] (0.51, 0.66) arc(27:54.5:0.575 and 1.42) arc(101:87:1.63 and 1.2) arc(60:40:1.42 and 1.39) arc(10.5:-10:1.115 and 0.9) arc(53:74:1.83 and 0.6);
					\end{scope}

					\begin{scope}[scale=1.6]
						\clip (0.51, 0.66) arc(27:54.5:0.575 and 1.42) arc(101:87:1.63 and 1.2) arc(60:40:1.42 and 1.39) arc(10.5:-10:1.115 and 0.9) arc(53:74:1.83 and 0.6);

						\fill[gray!25] (0.51, 0.66) arc(27:54.5:0.575 and 1.42) arc(101:87:1.63 and 1.2) arc(60:40:1.42 and 1.39) arc(10.5:-10:1.115 and 0.9) arc(53:74:1.83 and 0.6);
						%
						
						\draw[gray] (0,1.7) arc(90:0:0.54 and 1.95);
						\draw[gray] (0,1.7) arc(90:0:0.795 and 1.92);
						\draw[gray] (0,1.7) arc(90:0:1.03 and 1.85);
						\draw[gray] (0,1.7) arc(90:0:1.23 and 1.787);
						\draw[gray] (0,1.7) arc(90:0:1.4 and 1.71);
						
						\draw[gray] (-0.15,1.14) arc(-90:0:1.7*0.565 and 0.6*0.565);
						\draw[gray] (-0.15,1.00) arc(-90:0:1.7*0.66 and 0.6*0.66);
						\draw[gray] (-0.15,0.84) arc(-90:0:1.7*0.745 and 0.6*0.745);
						\draw[gray] (-0.15,0.675) arc(-90:0:1.7*0.82 and 0.6*0.82);
						\draw[gray] (-0.15,0.5) arc(-90:0:1.7*0.89 and 0.6*0.89);
					\end{scope}
					
					\begin{scope}[scale=1.6]	
						
					\end{scope}

					\draw (0,0) -- (1.70,2.20);
				\end{scope}


			\end{scope}
		\end{tikzpicture}
		\caption{Conic domains}\label{230222651}
	\end{figure}
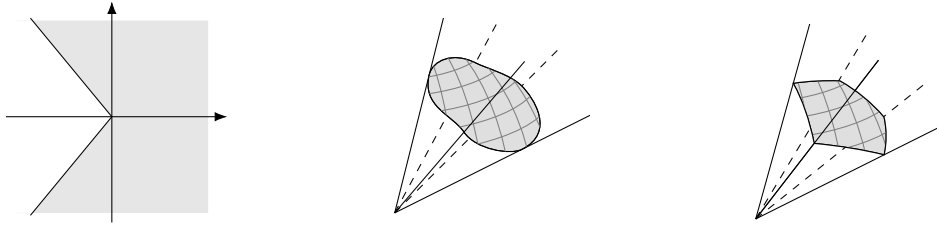
	We denote
	$$
	B_R^{\Omega}:=\Omega\cap B_R(0)\quad\text{and}\quad Q_R^{\Omega}:=(1-R^2,1]\times B_R^{\Omega}\,.
	$$
	
	In this subsection, for a certain class of sets $\cM$ including Lipschitz cones (see Assumption \ref{2207111151}), we prove that if $u$ satisfies
	\begin{align*}
		\begin{cases}
			u_t=\Delta u\quad&\text{in}\quad Q_1^{\Omega}\,;\\
			u=0\quad&\text{on}\quad (0,1]\times\big((\partial \Omega)\cap B_1(0)\big)\,,
		\end{cases}
	\end{align*}
	then for any $\lambda\in(0,\lambda_0)$ and  $0<R<1$,
	\begin{align}\label{2301111236}
		|u(t,x)|\lesssim_{\cM,\lambda,R} |x|^{\lambda}\sup_{Q_1^{\Omega}}|u|\quad\text{whenever}\quad (t,x)\in Q_R^{\Omega}
	\end{align}
	(see Remark \ref{2301111204}), where $\lambda_0$ is the constant defined in \eqref{220818848}.
	
	\begin{remark}
		As shown in \cite{Kozlov}, estimate \eqref{2301111236} is closely related to heat kernel estimates.
		In \cite[Lemma 3.9]{Kozlov}, Kozlov and Nazarov used an estimate of the form \eqref{2301111236} to derive bounds for the kernel of parabolic equations in $C^{1,1}$ cones.
	\end{remark}
	
	Before stating the main result of this subsection (Theorem \ref{2208221223}), we introduce the notions of the spherical gradient and the spherical Laplacian, formulated without differential geometric terminology.
	For a function $f$ on $\cM$, we denote $F_f(x):=f\big(\frac{x}{|x|}\big)$.
	We also denote
	\begin{align*}
		C^{\infty}(\cM)&:=\text{the set of all $f:\cM\rightarrow\bR$ for which $F_f\in C^{\infty}(\Omega)$}\,;\\
		C_c^{\infty}(\cM)&:=\{f\in C^{\infty}(\cM)\,:\,\mathrm{supp}(f)\subset \cM\}\,.
	\end{align*}
	The spherical gradient and spherical Laplacian of $f\in C^{\infty}(\cM)$, denoted by $\nabla_\bS f$ and $\Delta_\bS f$, are defined by
	\begin{align*}
		\nabla_\bS f=\nabla F_f|_{\cM}\quad\text{and}\quad \Delta_\bS f=\Delta F_f|_{\cM}\,.
	\end{align*}
	Direct calculation shows that:
	\begin{itemize}
		\item For any $f\in C_c^{\infty}(\cM)$ and $g\in C^{\infty}(\cM)$\,,
		\begin{align*}
			\int_{\cM}\big(\nabla_\bS f,\nabla_\bS g\big)_{\mathbb{R}^d}\dd A_\bS=-\int_{\cM}(\Delta_\bS f) g\,\dd A_\bS\,,
		\end{align*}
		where $(\,\cdot\,,\,\cdot\,)_{\mathbb{R}^d}$ is the inner product on $\bR^d$.
		
		\item For any $F\in C^{\infty}(\Omega)$,
		\begin{align}\label{220818255}
			|\nabla F|^2=|D_r F|^2+\frac{1}{r^2}|\nabla_\bS F|^2\,.
		\end{align}
		
		\item For a function $F\in C^{\infty}(\Omega)$,
		\begin{align}\label{220716632}
			\Delta F=D_{rr}F+\frac{d-1}{r}D_r F+\frac{1}{r^2}\Delta_\bS F\,.
		\end{align}
	\end{itemize}
	In \eqref{220818255} and \eqref{220716632}, $F$ is also viewed as a function defined on $\bR_+\times\cM$, given by $(r,\sigma)\mapsto F(r\sigma)$.
	We leave it to the reader to verify that $\nabla_{\bS}$ (resp. $\Delta_{\bS}$) coincides with the gradient (resp. Laplace–Beltrami operator) associated with the standard differential structure on $\bS^{d-1}$; see \cite{JJ} for details.

	We make certain assumptions on $\cM$ in order to apply Theorem \ref{220602322}.(2).

	\begin{assumption}\label{2207111151}
		Define  $\partial_\bS \cM:=\overline{\cM}\setminus\cM$.
		\begin{enumerate}
			\item $\cM$ is a connected (relatively) open set in $\bS^{d-1}$ with $\overline{\cM}\neq \bS^{d-1}$.
			\item
			\begin{align}\label{220820634}
				\inf_{\substack{p\in\partial_\bS\cM \\ r\in(0,1]}}\frac{A_\bS\big(\{\sigma\in \bS^{d-1}\setminus\cM\,:\,|\sigma-p|<r\}\big)}{r^{d-1}}>0\,,
			\end{align}
			
			\item Let $w_0(\sigma)$ be the first (positive) Dirichlet eigenfunction of the spherical Laplacian $\Delta_\bS$ on $\cM$ (see Proposition \ref{230413525}.(1)).
			There exist constants $A,\,N>0$ such that
			\begin{align}\label{2208181123}
				w_0(\sigma)\geq N^{-1}d(\sigma,\partial_\bS\cM)^A\,.
			\end{align}
		\end{enumerate}
	\end{assumption}

	By $\mathring{W}_2^1(\cM)$, we denote the closure of $C_c^{\infty}(\cM)$ in
	$$
	W_2^1(\cM):=\{f\in\cD'(\cM)\,:\,\|f\|_{L_2(\cM)}+\|\nabla_{\bS}f\|_{L_2(\cM)}<\infty\}\,.
	$$
	
	\begin{prop}\label{230413525}\,
		
		\begin{enumerate}
			\item If Assumption \ref{2207111151}.(1) holds, then 
			\begin{align}\label{230107830}
				\Lambda_0:=\inf_{\substack{w\in C_c^{\infty}(\cM)\\w\not\equiv 0}}\frac{\int_{\cM}|\nabla_\bS w|^2\dd A_\bS}{\int_{\cM}|w|^2\dd A_\bS}>0\,,
			\end{align}
			and there exists a unique $w_0$ in $C^{\infty}(\cM)\cap \mathring{W}_2^1(\cM)$ such that
			\begin{align}\label{230413250}
				w_0>0\quad,\quad \int_{\cM}|w_0|^2\dd A_\bS=1\quad,\quad\Delta_\bS w_0+\Lambda_0w_0=0\,.
			\end{align}
			Moreover, $w_0$ is bounded on $\cM$, and the function
			\begin{align}\label{220818845}
				W_0(x):=|x|^{\lambda_0}w_0(x/|x|)
			\end{align}
			is a positive harmonic function in $\Omega$, where
			\begin{align}\label{220818848}
				\lambda_0:=-\frac{d-2}{2}+\sqrt{\Lambda_0+\Big(\frac{d-2}{2}\Big)^2}>0\,.
			\end{align}

			\item If Assumption \ref{2207111151}.(2) holds, then 
			\begin{align}\label{2410311138}
				\inf_{p\in\partial\Omega,\, r>0}\frac{\big|\Omega^c\cap B_r(p)\big|}{\big|B_r(p)\big|}>0\,.
			\end{align}
			
			\item Let $-e_d\notin \overline{\cM}$ and let $\phi_d$ be the stereographic projection from $\bS^{d-1}\setminus\{-e_d\}$ to $\bR^{d-1}$ defined by 
			\begin{align}\label{2208221047}
				\phi_d(\sigma_1,\ldots,\sigma_{d-1},\sigma_d)=\Big(\frac{\sigma_1}{1+\sigma_d},\ldots,\frac{\sigma_{d-1}}{1+\sigma_d}\Big)\,.
			\end{align}
			If $\phi_d(\cM)$ is a John domain in $\bR^{d-1}$ (see Remark \ref{220819318} for its definition), then Assumption \ref{2207111151}.(3) holds.
		\end{enumerate}
	\end{prop}
	\begin{proof}
		(1) \eqref{230107830} follows from \cite[Theorems 10.13, 10.18, 10.22]{AGriog}.
		It is proved in \cite[Theorem 10.11, Corollary 10.12]{AGriog} that there exists a unique $w_0\in C^{\infty}(\cM)\cap \mathring{W}_2^1(\cM)$ satisfying \eqref{230413250}.
		
		To prove the boundedness of $w_0$, without loss of generality, we assume that $-e_d:=(0,\ldots,0,-1)\notin\overline{\cM}$.
		We denote by $\phi_d$ the stereographic projection from $\bS^{d-1}\setminus\{-e_d\}$ onto $\bR^{d-1}$ defined by \eqref{2208221047}.
		Then $\phi_d(\cM)$ is a bounded domain in $\bR^{d-1}$.
		Consider the function $\widetilde{w}_0:=w_0\circ \phi_d^{-1}:\phi_d(\cM)\rightarrow \bR$.
		Then $\widetilde{w}_0$ belongs to $\mathring{W}_2^1\big(\phi_d(\cM)\big)$ and satisfies
		\begin{align}\label{2411111144}
			\sum_{i,j=1}^{d-1}a^{ij}D_{ij}\widetilde{w}_0+\sum_{i=1}^{d-1}b^iD_i\widetilde{w}_0+\Lambda_0\widetilde{w}_0=0\quad\text{in}\quad \phi_d(\cM)\subset \bR^{d-1}\,,
		\end{align}
		where $a^{ij},\,b^{i}\in C^{\infty}(\bR^{d-1})$ ($i,\,j=1,\,\ldots,\,d-1$) are smooth functions in $\bR^{d-1}$ such that there exists $\nu>0$ satisfying
		$$
		\nu|\xi|^2\leq \sum_{i,j=1}^{d-1}a^{ij}(x)\xi_i\xi_j\leq \nu^{-1}|\xi|^2\qquad\forall\,\,\xi=(\xi_1,\ldots,\xi_{d-1})\in\bR^{d-1}\,\,,\,\,x\in\phi_d(\cM)\,.
		$$
		The boundedness of $\widetilde{w}_0$ follows from classical results for elliptic equations (see, \textit{e.g.}, \cite[Theorem 3.13.1]{lady2}), and hence $w_0$ is bounded.
		
		It follows directly from \eqref{220716632} that the function $W_0$ in \eqref{220818845} is harmonic in $\Omega$.

		(2) 
		For any $p\in\partial_\bS\cM$ and $r\in(0,1)$, we have
		\begin{align*}
			&\big\{s\sigma\in\bR^{d}\,:\,s\in(1-r/2,1+r/2),\,\sigma\in \bS^{d-1}\cap B_{r/2}(p)\big\}\\
			\subset\,&B_r(p)\subset\big\{s\sigma\in\bR^{d}\,:\,s\in(1-r,1+r),\,\sigma\in \bS^{d-1}\cap B_{2r}(p)\big\}\,.
		\end{align*}
		Therefore \eqref{220820634} holds if and only if
		\begin{align}\label{2208201149}
			\inf_{\substack{p\in\partial_\bS\cM\\ r\in(0,1]}}\frac{\big|\Omega^c\cap B_r(p)\big|}{\big| B_r(p)\big|}>0\,.
		\end{align}
		If $r\geq 2$, then $B_r(p)\supset B_{r/2}(0)$.
		Hence, 
		\begin{align}\label{2208206491}
			\inf_{p\in\partial_\bS\cM,\, r\geq 2}\frac{\big|\Omega^c\cap B_r(p)\big|}{\big| B_r(p)\big|}\geq2^{-d}\inf_{r\geq 1}\frac{\big|\Omega^c\cap B_{r}(0)\big|}{\big| B_r(0)\big|}=2^{-d}\frac{A_{\bS}\left(\bS^{d-1}\setminus\cM\right)}{A_\bS(\bS^{d-1})}>0\,.
		\end{align}
		Consequently, \eqref{2410311138} is implied by \eqref{2208201149} and \eqref{2208206491}.
		
		(3) We denote $U_{\cM}:=\phi_d(\cM)$.
		It follows from Example \ref{21.05.18.2}.(2) that $W_0$ (in \eqref{220818845}) is a Harnack function in $\Omega$.
		Since $W_0$ is a continuous Harnack function, and $\phi_d$ (resp. $\phi_d^{-1}$) is Lipschitz continuous on $\cM$ (resp. $U_\cM$), we obtain that $\widetilde{w}_0:=w_0\circ \phi_d^{-1}$ is a continuous Harnack function in $U_{\cM}$ (see \eqref{241012350}).
		In addition, $d(\sigma,\partial_\bS\cM)\simeq d\big(\phi_d(\sigma),\partial U_{\cM}\big)$.
		By Remark \ref{220819318}, since $\Omega$ is a John domain, there exists a constant $A>0$ such that $\widetilde{w}_0(x')\gtrsim d(x',\partial U_{\cM})^A$ for all $x'\in \phi_d(\cM)$.
		Thus, \eqref{2208181123} is proved.
	\end{proof}

	\begin{thm}\label{2208221223}
		Let $\cM\subset \bS^{d-1}$ ($d\ge 2$) satisfy Assumption \ref{2207111151}, and suppose that $u\in C^{\infty}(Q_1^{\Omega})$ satisfies
		\begin{align*}
			&u_t=\Delta u\quad\text{in}\quad Q_1^{\Omega}\quad;\quad \lim_{(t,x)\rightarrow (t_0,x_0)}u(t,x)=0\,\,\,\,\text{whenever}\,\,\,\, 0<t_0\leq 1\,\,,\,\,x_0\in(\partial\Omega)\cap B_1\,.
		\end{align*}
		Then for any $\epsilon\in(0,1)$ and $R \in(0,1)$,
		\begin{align*}
			|u(t,x)|\leq N\Big(\sup_{Q_1^{\Omega}}|u|\Big)W_0(x)^{1-\epsilon}\qquad \forall\,\, (t,x)\in Q_R^{\Omega}\,,
		\end{align*}
		where $W_0$ is the function defined in \eqref{220818845} and $N=N(\cM,\epsilon,R)>0$.
	\end{thm}
	Recall that $\mathring{W}_2^1\big(B_1^{\Omega}\big)$ is the closure of $C_c^{\infty}\big(B_1^{\Omega}\big)$ in $W_2^1\big(B_1^{\Omega}\big)$.

	\begin{remark}\label{2301111204}
		Theorem \ref{2208221223} implies that if $u$ satisfies the assumptions of the theorem and $\lambda\in(0,\lambda_0)$, where $\lambda_0$ is given in \eqref{220818848}, then
		\begin{align}\label{2208221229}
			|u(t,x)|\leq N(\cM,\lambda,R)\Big(\sup_{Q_1^{\Omega}}|u|\Big)|x|^{\lambda}\quad\text{on}\,\,\,Q_{R}^{\Omega}\,.
		\end{align}
		This holds because $w_0$ in \eqref{220818845} is bounded.
		
		It is important to note that for $\lambda>\lambda_0$, estimate \eqref{2208221229} does not hold in general. 
		Indeed, observe that $u(t,x):=W_0(x)$ satisfies the assumptions of Theorem \ref{2208221223}. 
		Given the definition of $W_0(x)$ in \eqref{220818845}, there is no constant $N$ such that \eqref{2208221229} holds with $u(t,x)=W_0(x)$ and $\lambda>\lambda_0$.
		To complete this claim, we need to show that $w_0$ in Proposition \ref{230413525}.(1) satisfies $w_0\in C(\overline{\cM})$ and $w_0|_{\partial_\bS \cM}\equiv 0$.
		To do this, it suffices to prove that the function $\widetilde{w}_0$ in the proof of Proposition \ref{230413525}.(1) belongs to $C\big(\overline{\phi_d(\cM)}\big)$ and $\widetilde{w}_0|_{\partial\phi_d(\cM)}=0$, assuming, without loss of generality, that $-e_d\notin \overline{\cM}$.
		Assumption \ref{2207111151} implies that $\phi_d(\cM)$ satisfies the volume density condition \eqref{Acondition} in $\bR^{d-1}$.
		Since $\widetilde{w}_0\in \mathring{W}_2^1\big(\phi_d(\cM)\big)$ satisfies equation \eqref{2411111144}, it follows from classical results for elliptic equations (see, \textit{e.g.}, \cite[Theorem 3.14.1]{lady2}) that $\widetilde{w}_0\in C\big(\overline{\phi_d(\cM)}\big)$ and $\widetilde{w}_0|_{\partial\phi_d(\cM)}=0$.
	\end{remark}

	\begin{proof}[Proof of Theorem \ref{2208221223}]
		\textbf{Step 1.}
		Put $K=\max(A,\lambda_0)$ where $A$ and $\lambda_0$ are the constants in \eqref{2208181123} and \eqref{220818848}, respectively.
		A direct calculation (see, \textit{e.g.}, \cite[Lemma 3.4.(1)]{ConicPDE}) shows that for any $\sigma\in\cM$,
		\begin{align*}
			d(\sigma,\partial\Omega)\leq d(\sigma,\partial_\bS\cM)\leq 2d(\sigma,\partial\Omega)\,.
		\end{align*}
		Therefore, we obtain that for any $x\in\Omega\cap B_1(0)$,
		$$
		d(x,\partial B_1^\Omega)^K\leq d(x,\partial\Omega)^{K}\simeq  |x|^{K}d\big(x/|x|,\partial_\bS\cM\big)^{K}\leq |x|^{\lambda_0}d(x/|x|,\partial_\bS\cM)^A\lesssim W_0(x)\,.
		$$
		Put $\epsilon_i=\epsilon+\frac{i}{2K}$ for $i\in\bN_0$, and take $i_0\in\bN$ such that $\epsilon_{i_0}< 1\leq \epsilon_{i_0+1}$.
		Since $W_0$ is bounded on $B_1^\Omega$ (see Proposition \ref{230413525}.(1)), we have
		$$
		\sup_{Q_{1}^\Omega}|W_0^{-1+\epsilon_{i_0+1}}u|\lesssim_{\Omega,\epsilon}\sup_{Q_{1}^{\Omega}}|u|\,.
		$$
		Hence, it remains to prove that for any $i\in\{0,1,\ldots,i_0\}$ and $0<R_1<R_2\le 1$,
		\begin{align}\label{2208210059}
			\sup_{Q_{R_1}^\Omega}|W_0^{-1+\epsilon_{i}}u|\lesssim N(\Omega,\epsilon,R,r)\sup_{Q_{R_2}^\Omega}|W_0^{-1+\epsilon_{i+1}}u|\,.
		\end{align}
		
		\textbf{Step 2.}
		Take $\eta\in C^{\infty}\big(\bR\times \bR^d\big)$ such that 
		\begin{align*}
			\eta(t,x)=
			\begin{cases}
				\vspace{0.5mm}
				1\quad&\text{if}\quad t>1-R_1^2\quad\text{and}\quad|x|<R_1\,;\\
				0\quad&\text{if}\quad t<1-R_2^2\quad\text{or}\quad|x|>R_2\,.
			\end{cases}
		\end{align*}
		Put 
		\begin{align}\label{2304141114}
			v=u\eta\,\,,\,\,\,\,f^0:=\big(\partial_t\eta+\Delta \eta\big)u\,\,,\,\,\,\,f^i:=-2u D_i\eta\quad (i=1,\,\ldots,\,d)\,,
		\end{align}
		so that $v\in C\big(\overline{Q_1^{\Omega}}\big)\cap C^{\infty}(Q_1^{\Omega})$ satisfies 
		\begin{align*}
			\partial_tv=\Delta v+f_0+\sum_{i=1}^d D_if^i\quad\text{in}\quad Q_1^{\Omega}\quad ;\quad v\equiv 0\quad\text{on}\quad  \overline{Q_1^{\Omega}}\setminus Q_1^{\Omega}\,.
		\end{align*}
		We recall that by Propositions \ref{230413525}.(2) and \ref{240928313} and \eqref{Acondition}, $B_1^\Omega$ satisfies $\mathbf{LHMD}(\lambda)$ for some $\lambda\in(0,1)$.
		Therefore, the results in Section \ref{fatex} are applicable to $B_1^\Omega$.
		
		\textbf{Step 2.1)} We first show that $v\in\cH_{2,d-2}^1(B_1^\Omega,1)$.
		Since
		\begin{align*}
			\Big\|f^0+\sum_{i=1}^dD_if^i\Big\|_{\bH_{2,d+2}^{-1}(B_1^{\Omega},1)}\lesssim \left\|f^0\right\|_{\bL_{2,d+2}(B_1^{\Omega},1)}+\sum_{i=1}^d\left\|f^i\right\|_{\bL_{2,d}(B_1^{\Omega},1)}\lesssim \sup_{Q_1^{\Omega}}|u|
		\end{align*}
		(see \eqref{241111154} and \eqref{2409098051}),
		there exists $\widetilde{v}\in \cH_{2,d-2}^1(B_1^\Omega,1)$ such that 
		\begin{align}\label{24111111223}
			\partial_t\widetilde{v}=\Delta \widetilde{v}+f^0+\sum_{i=1}^d D_if^i\quad\text{and}\quad \widetilde{v}(0,\,\cdot\,)=0\,.
		\end{align}
		Since $f^0,f^i\in C^{\infty}(Q_1^{\Omega})$, we obtain that $\widetilde{v}\in C^{\infty}(Q_1^{\Omega})$.
		Moreover, because $f^0,\,f^i\in C(\overline{Q_1^\Omega})$, Theorems \ref{220602322} and \ref{220821002901} yield that $\widetilde{v}\in C(\overline{Q_1^\Omega})$ and that $v=0$ on $\overline{Q_1^\Omega}\setminus Q_1^\Omega$.
		The maximum principle implies that $v\equiv \widetilde{v}$, and therefore $v\in \cH_{2,d-2}^1(B_1^{\Omega},1)$.
		
		\textbf{Step 2.2)}
		To prove \eqref{2208210059}, suppose that RHS in \eqref{2208210059} is finite.
		Recall that $B_1^\Omega$ satisfies $\mathbf{LHMD}(\lambda)$ for some $\lambda\in(0,\frac{1}{2})$.
		In addition, $v\in\cH_{2,d-2}^1(\Omega,1)$ (in \eqref{2304141114}) is a solution of equation \eqref{24111111223}.
		Since $W_0$ is a regular Harnack function (see Example \ref{21.05.18.2}.(2)), and 
		\begin{align*}
			&\sup_{Q_{R_2}^\Omega}\Big(\big|W_0^{-1+\epsilon_i}d(\cdot,\partial B_1^\Omega)^{2-\lambda \epsilon}f\big|+\big|W_0^{-1+\epsilon_i}d(\cdot,\partial B_1^\Omega)^{1-\lambda \epsilon}f^i\big|\Big)\\
			\lesssim_N\,&\sup_{Q_{R_2}^{\Omega}}\big|W_0^{-1+\epsilon_i}d(\cdot,\partial B_1^\Omega)^{1-\lambda\epsilon}u\big|\\
			\lesssim_N\,&\sup_{Q_{R_2}^{\Omega}}\big|W_0^{-1+\epsilon_{i}+1/(2K)}u\big|=\sup_{Q_{R_2}^{\Omega}}\big|W_0^{-1+\epsilon_{i+1}}u\big|\,,
		\end{align*}
		where $N=N(R_1,R_2,\cM,\epsilon)$.
		Note that $1-\lambda \epsilon\geq \frac{1}{2}$, so $d(\cdot,\partial B_1^\Omega)^{1-\lambda \epsilon}\leq d(\cdot,\partial B_1^\Omega)^{1/2}$.
		Theorem \ref{220602322}.(2) (with Theorem \ref{220821002901}) implies that $v$ satisfies
		\begin{align}\label{230413551} 
			\begin{aligned}
				\sup_{(0,1]\times\Omega}\big|W_0^{-1+\epsilon_i}v| \,&\lesssim\sup_{Q_{R_2}^\Omega}\Big(\big|W_0^{-1+\epsilon_i}d(\cdot,\partial B_1^\Omega)^{2-\lambda \epsilon}f\big|+\big|W_0^{-1+\epsilon_i}d(\cdot,\partial B_1^\Omega)^{1-\lambda \epsilon}f^i\big|\Big)\\
				&\lesssim\sup_{Q_{R_2}^{\Omega}}\big|W_0^{-1+\epsilon_{i+1}}u\big|\,.	
			\end{aligned}				
		\end{align}
		Since $v\equiv u$ in $Q_{R_1}^{\Omega}$, \eqref{230413551} implies \eqref{2208210059}.
		Hence, the claim is proved.
	\end{proof}

	\appendix
	
	\mysection{Weighted Besov space}
	
	\begin{prop}\label{220418435}
		Let $\Phi$ be a regular Harnack function, and let $p\in(1,\infty)$ and $\theta\in\bR$.
		For any $k\in\bN_0$ and $s\in (0,1)$,
		\begin{align}\label{220609106}
			\begin{aligned}
				\|\Phi f\|_{B_{p,\theta}^{k+s}}^p \simeq_N\,&\sum_{i=0}^k\int_{\Omega}|\rho^i D_x^if|^p\Phi^p\rho^{\theta-d}\dd x\\
				&\,\,+\int_{\Omega}\bigg(\int_{|x-y|<\frac{\rho(x)}{2}}\frac{|D^{k}_xf(x)-D^{k}_xf(y)|^p}{|x-y|^{d+s p}}dy\bigg)\Phi(x)^p\rho(x)^{(k+s)p +\theta-d}\dd x
			\end{aligned}
		\end{align}
		where  $N=N(d,p,k,s,\mathrm{C}_2(\Phi))$.
	\end{prop}
	\begin{proof}
		\textbf{Step 1.} We first show that
		\begin{align}\label{220609102}
			\|f\|_{B_{p,\theta}^{s}(\Omega)}^p\simeq\,&\|f\|_{L_{p,\theta}(\Omega)}+\int_{\Omega}\int_{|x-y|\leq \frac{\rho(x)}{2}}\frac{|(\trho^{(\theta-d)/p+s}f)(x)-(\trho^{(\theta-d)/p+s}f)(y)|^p}{|x-y|^{d+s p}}\dd x\dd y\,.
		\end{align}
		We recall the following equivalent norm for Besov spaces:
		\begin{align}\label{2205181108}
			\|f\|_{B_{p}^{s}}^p\simeq_{d,p,s} \|f\|_{L_p}^p+\iint_{\bR^d\times \bR^d}\frac{|f(x)-f(y)|^p}{|x-y|^{d+s p}}\dd x\dd y
		\end{align}
		(see, \textit{e.g.}, \cite[Theorem 2.5.7/(i)]{triebel2}).
		Recall that for $\xi\in C_c^{\infty}(\bR_+)$, we denote $\xi_{(n)}(x):=\xi(\ee^{-n}\trho(x))$.
		Applying \eqref{2205181108}, we obtain
		\begin{alignat*}{2}
			\|f\|_{B_{p,\theta}^{s}(\Omega)}^p&\simeq_N&& \sum_{k\in\bZ}\ee^{n\theta}\big\|\big(\zeta_{0,(n)}f\big)(\ee^n\,\cdot\,)\big\|_p^p\\
			&\quad&&+\sum_{k\in\bZ}\ee^{n\theta}\iint_{\bR^d\times \bR^d}\frac{\big|\big(\zeta_{0,(n)}f\big)(\ee^nx)-\big(\zeta_{0,(n)}f\big)(\ee^ny)\big|^p}{|x-y|^{d+s p}}\dd x\dd y\\
			&=:&&I_1+I_2\,.
		\end{alignat*}
		\eqref{2409098051} implies that $I_1\simeq_{d,p,\theta} \|f\|_{L_{p,\theta}}^p$.
			By a change of variables, setting $F=\trho^{\,(\theta-d)/p+s}f$ and $\eta(t)=t^{-(\theta-d)/p-s}\zeta_0(t)$, we obtain
		\begin{alignat*}{3}
			I_2\,&=&&\sum_{k\in\bZ}\iint_{\bR^d\times \bR^d}\frac{|\zeta_{0,(n)}(x)f(x)-\zeta_{0,(n)}(y)f(y)|^p}{|x-y|^{d+s p}}\ee^{n(\theta-d+s p)}\dd x\dd y\\
			&=&&\sum_{k\in\bZ}\iint_{\bR^d\times \bR^d}\frac{|\eta_{(n)}(x)F(x)-\eta_{(n)}(y)F(y)|^p}{|x-y|^{d+s p}}\dd x\dd y\\
			&\lesssim_p&& \sum_{k\in\bZ}\iint_{\left\{|x-y|\geq \rho(x)/2\right\}}\frac{|\eta_{(n)}(x)F(x)|^p+|\eta_{(n)}(y)F(y)|^p}{|x-y|^{d+s p}}\dd x\dd y\\
			&&& +\sum_{k\in\bZ}\iint_{\left\{|x-y|\leq \rho(x)/2\right\}}\frac{|\eta_{(n)}(x)-\eta_{(n)}(y)|^p}{|x-y|^{d+s p}}|F(x)|^p\dd x\dd y\\
			&&&+\sum_{k\in\bZ}\iint_{\left\{|x-y|\leq \rho(x)/2\right\}}|\eta_{(n)}(y)|^p\frac{|F(x)-F(y)|^p}{|x-y|^{d+s p}}\dd x\dd y\\
			&=:&&I_{2,1}+I_{2,2}+I_{2,3}\,.
		\end{alignat*}
		We now estimate each term $I_{2,1}$, $I_{2,2}$, and $I_{2,3}$.
		Observe that
		\begin{align}\label{2206091204}
			N^{-1}\leq \sum_{n\in\bZ}|\eta(\ee^{-n} t)|^p\leq N \quad\text{and}\quad\sum_{n\in\bZ}\ee^{-np}|\eta'(\ee^{-n} t)|^p\leq N t^{-p}\,,
		\end{align}
		for any $t>0$, where $N=N(d,p,\theta,s)$.
		It follows from \eqref{2206091204} that
		\begin{align}\label{2206091221}
			I_{2,1}\simeq_N \int_{\Omega}\int_{y:|x-y|\geq \frac{\rho(x)}{2}}\frac{|F(x)|^p+|F(y)|^p}{|x-y|^{d+s p}}\dd y\dd x\simeq_N \int_{\Omega}|f(x)|^p \rho(x)^{\theta-d}\dd x\,,
		\end{align}
		where $N=N(d,p,\theta,s)$; for the last inequality, note that if $|x-y|\geq \frac{\rho(x)}{2}$, then $|x-y|\geq \frac{\rho(y)}{3}$.
		To estimate $I_{2,2}$, observe that for $x,\,y\in\Omega$ with $|x-y|<\frac{\rho(x)}{2}$, 
		\begin{align}\label{220609621}
			\begin{split}
				&\sum_{n\in\bZ}|\eta_{(n)}(x)-\eta_{(n)}(y)|^p\lesssim_N \sum_n|x-y|^p\ee^{-np}\Big(\int_0^1|\eta'(\ee^{-n}\trho(x_r))|\dd r\Big)^p\\
				\leq\,\,\,\,& |x-y|^p\int_0^1\sum_{n}\ee^{-np}|\eta'(\ee^{-n}\trho(x_r))|^p \dd r \lesssim_N |x-y|^p\int_0^1\trho(x_r)^{-p}\dd r\,,
			\end{split}
		\end{align}
		where $x_r=(1-r)x+ry$ and $N=N(d,p,\theta,s)$.
		Here, the first inequality follows from the fact that $|\nabla \trho|$ is bounded in $\Omega$, and the last inequality follows from \eqref{2206091204}.
		Since $\rho(x_r)\geq \frac{\rho(x)}{2}$, we have
		$$
		\sum_n|\eta_{(n)}(x)-\eta_{(n)}(y)|^p\lesssim_ N\,|x-y|^p\rho(x)^{-p}\,,
		$$
		where $N=N(d,p,\theta,s)$.
		Hence, we obtain
		\begin{align}\label{2206091223}
			I_{2,2}\,&\lesssim \int_{\Omega}\int_{y:|x-y|\leq \frac{\rho(x)}{2}}\frac{|F(x)|^p\rho(x)^{-p}}{|x-y|^{d-(1-s)p}}\dd y\dd x\lesssim_{d,s,p} \int|f(x)|^p\rho(x)^{\theta-d}\dd x.
		\end{align}
		By \eqref{2206091221}-\eqref{2206091223} and that $I_{2,3}\lesssim I_2+I_{2,2}\lesssim \|f\|_{B_{p,\theta}^{s}}^p$, 
		we have $\|f\|_{B_{p,\theta}^{s}(\Omega)}^p\simeq \|f\|_{L_{p,\theta}(\Omega)}+I_{2,3}$.
		By applying \eqref{2206091204} to $I_{2,3}$, we conclude that \eqref{220609102} holds.
		
		\textbf{Step 2.}
		We now prove \eqref{220609106} for the case $k=0$.
		Denote $F:=\trho^{(\theta-d)/p+s}f$.
		Since $\Phi\cdot \trho^{(\theta-d)/p+s}$ is a regular Harnack function, if $|x-y|<\rho(x)/2$, then
		\begin{align}
			&\Big|\big|\Phi(x)F(x)-\Phi(y)F(y)\big|-\Phi(x)\trho(x)^{(\theta-d)/p+s}\big|f(x)-f(y)\big|\Big|\nonumber\\
			\leq \,&\big|\Phi(x)\trho(x)^{(\theta-d)/p+s}-\Phi(y)\trho(y)^{(\theta-d)/p+s}\big||f(y)|\label{2206091243}\\
			\leq \,&N |x-y|\cdot\Phi(y)\rho(y)^{-1}|F(y)|\nonumber
		\end{align}
		where $N=N(d,\mathrm{C}_2(\Phi))$.
		By combining \eqref{220609102} (applied to $\Psi F$ instead of $f$), \eqref{2206091243}, together with the fact that
		\begin{align*}
			\int_{\Omega}\int_{y:|x-y|<\rho(y)}\frac{\big(|x-y|\cdot\Phi(y)\rho(y)^{-1}|F(y)|\big)^p}{|x-y|^{d+s p}}\dd y\dd x\lesssim\int_{\Omega}|f(y)|^p\Phi(y)^p\rho(y)^{\theta-d}\dd y\,,
		\end{align*}
		we conclude that \eqref{220609106} holds for $k=0$.
		
		\textbf{Step 3.} Let $k\geq 1$.
		It follows from \eqref{240911813} that
		$$
		\|\Phi f\|_{B_{p,\theta}^{k+s}(\Omega)}\simeq \sum_{i=0}^{k-1}\|\Phi D_x^if\|_{B_{p,\theta+ip}^{s}(\Omega)}+\|\Phi D_x^kf\|_{B_{p,\theta+kp}^{s}(\Omega)}\,.
		$$
		By \eqref{2411011107} and \eqref{240911813}, we have
		\begin{align*}
			\sum_{i=0}^{k-1}\|\Phi D_x^if\|_{L_{p,\theta+ip}(\Omega)}\lesssim\,& \sum_{i=0}^{k-1}\|\Phi D_x^if\|_{B_{p,\theta+ip}^{s}(\Omega)}\lesssim \sum_{i=0}^{k-1}\|\Phi D_x^if\|_{H_{p,\theta+ip}^{1}(\Omega)}\\
			\simeq\,&\sum_{i=0}^{k}\|\Phi D_x^if\|_{L_{p,\theta+ip}(\Omega)}\lesssim \sum_{i=0}^{k-1}\|\Phi D_x^if\|_{L_{p,\theta+ip}(\Omega)}+\|\Psi D_x^k f\|_{B_{p,\theta+kp}^{s}(\Omega)}\,.
		\end{align*}
		Hence, we have
		\begin{align*}
			\|\Phi f\|_{B_{p,\theta}^{k+s}(\Omega)}\simeq \sum_{i=0}^k\|\Phi D_x^if\|_{B_{p,\theta+ip}^{s}(\Omega)} \simeq \sum_{i=0}^{k-1}\|\Phi D_x^if\|_{L_{p,\theta+ip}(\Omega)}+\|\Psi D_x^k f\|_{B_{p,\theta+kp}^{s}(\Omega)}\,.
		\end{align*}
		By Lemma \ref{240911814} and Step 2 (\eqref{220609106} for $k=0$), the proof is complete.
	\end{proof}

	\begin{lemma}[Extension lemma]\label{240824401}
		Let $\alpha\in(0,1]$ and $p\in(1,\infty)$ be such that $\alpha>1/p$.
		For any $u_0\in \Psi B_{p,\theta+2/\alpha}^{\gamma+2-2/(\alpha p)}(\Omega)$, there exists $u\in \Psi \cH_{p,\theta}^{\gamma+2}(\Omega,\infty)$ such that $u(0)=u_0$ and 
		\begin{align}\label{240417615}
			\|u\|_{\Psi \cH_{p,\theta}^{\gamma+2}(\Omega,\infty)}\leq N\|u_0\|_{\Psi B_{p,\theta+2/\alpha}^{\gamma+2-2/(\alpha p)}(\Omega)}\,,
		\end{align}
		where $N=N(d,\alpha,p,\gamma,\theta)$.
	\end{lemma}
	
	\begin{proof}
		Note that
		$$
		\Psi B_{p,\theta+2/\alpha}^{\gamma+2-2/(\alpha p)}(\Omega)=\big(\Psi H_{p,\theta}^{\gamma+2}(\Omega),\Psi H_{p,\theta+2p}^{\gamma}(\Omega)\big)_{1/(\alpha p),p}
		$$
		(see Lemma \ref{241108522}),
		where $\big(X_0,X_1\big)_{1/(\alpha p),p}$ is the real interpolation space introduced in \cite{triebel}.
		It was proved in \cite{trace2023} that for any interpolation couple $(X_0,X_1)$ and any $u_0\in (X_0,X_1)_{1/(\alpha p),p}$, there exist $u\in L_p(\bR_+;X_0)$ and $f\in L_p(\bR_+;X_1)$ such that $u(0)=u_0$,
		$$
		\frac{1}{\Gamma(1-\alpha)}\int_0^t(t-s)^{-\alpha}\big(u(s)-u_0\big)\dd s=\int_0^tf(s)\dd s\,,
		$$
		and 
		\begin{align}\label{240417614}
			\|u\|_{L_p(\bR_+;X_0)}+\|f\|_{L_p(\bR_+;X_1)}\lesssim_{\alpha,p}\|u_0\|_{(X_0,X_1)_{1/(\alpha p),p}}\,.
		\end{align}
		By putting $X_0:=\Psi H_{p,\theta}^{\gamma+2}(\Omega)$ and $X_1:=\Psi H_{p,\theta+2p}^{\gamma}(\Omega)$, we obtain $u\in \Psi \cH_{p,\theta}^{\gamma+2}(\Omega,\infty)$ such that $u(0)\equiv u_0$ and $\partial_t^\alpha u=f$. 
		Moreover, \eqref{240417614} yields \eqref{240417615}.
	\end{proof}
	
	\mysection{Auxiliary result for the fractional heat equation}
	Throughout this section, we fix constants $T$ with $0<T<\infty$.
	
	\begin{lemma}\label{240426835}
		Let $\alpha\in(0,1)$ and let $\Omega$ be a bounded smooth domain.
		For any $f\in C_c^{\infty}\big((0,T]\times \Omega\big)$, there exists a unique function $$
		u\in C^{\infty}\big((0,T]\times \Omega\big)\cap C\big([0,T]\times \overline{\Omega}\big)
		$$
		such that $\mathrm{supp}(u)\subset (0,T]\times \overline{\Omega}$ and
		$$
		\partial_t^\alpha u=\Delta u+f\quad\text{in}\quad (0,T]\times \Omega\quad ;\quad u(t,x)=0\quad \text{if}\,\,\,\,t=0\,\,\,\,\text{or}\,\,\,\,x\in\partial\Omega\,.
		$$
		Moreover, such a function $u$ satisfies, for any $p\in(1,\infty)$, 
		$$
		\int_0^T\int_\Omega|u(t,x)|^{p-1}|D_x^2u(t,x)|\dd x\dd t<\infty\,.
		$$
	\end{lemma}
	
	\begin{proof}
		The uniqueness of such $u$ follows from the maximum principle (see \cite[Theorem 2]{LUCHKO2009218}).
		Therefore, it remains to prove the existence of such $u$.

		Recall that $\Omega$ is a bounded smooth domain and $f\in C_c^{\infty}\big((0,T]\times \Omega\big)$. 
		The following result follows from \cite[Theorem 2.10]{HKP2021}:
		\begin{itemize}
			\item[] For any $k\in\bN_0$, $p\in(1,\infty)$, and $\theta\in(-p-1,-1)$, there exists a unique $u:(0,T]\times \Omega\rightarrow \bR$ such that 
			\eqref{2206301251} holds for all $\zeta\in C_c^{\infty}(\Omega)$ and $\eta\in C_c^{\infty}\big([0,T)\big)$, and
			\begin{align}\label{240429814}
				\sum_{i=0}^{k+2}\int_0^T\int_\Omega |\rho^iD_x^i u|^p\rho^{\theta-d}\dd x\dd t<\infty\,.
			\end{align}
		\end{itemize}
		Note that the solution spaces $\mathfrak{H}_{p,\theta,0}^{\alpha,\gamma+2}(\Omega,T)$ in \cite{HKP2021} coincide with the space $\mathring{\cH}_{p,\theta-p}^{\alpha,\gamma+2}(\Omega,T)$ defined in Definition~\ref{240419512}.(2) (see Lemma \ref{240911115}.(1)).
		Since $\Omega$ is bounded, it follows that for any $k,k'\in\bN_0$ with $k'\ge k$, $1<p\le p'<\infty$, and $\theta,\theta'\in\bR$ satisfying $\frac{\theta'-d}{p'}\le\frac{\theta-d}{p}$, we have
		\begin{align}\label{241108903}
			\sum_{i=0}^{k+2}\left(\int_0^T\int_\Omega |\rho^iD_x^i u|^p\rho^{\theta-d}\dd x\dd t\right)^{1/p}\lesssim \sum_{i=0}^{k'+2}\left(\int_0^T\int_\Omega |\rho^iD_x^i u|^{p'}\rho^{\theta'-d}\dd x\dd t\right)^{1/p'}\,.
		\end{align}
		In addition, $f\in C_c^{\infty}\big((0,T]\times \Omega)$, so that for any $p\in(1,\infty)$, $\theta\in (-p-1,-1)$, and $k\in\bN$,
		\begin{align}\label{241108904}
			\int_0^T\int_\Omega |\rho^kD_x^k f|^p\rho^{\theta+2p-d}\dd x\dd t<\infty\,.
		\end{align}
		Combining \eqref{241108903}, \eqref{241108904}, and \cite[Theorem 2.10]{HKP2021}, 
		we conclude that for $k=2$, $p\in(1,\infty)$, and $\theta=d-p$, the solutions satisfying \eqref{240429814} coincide. 
		Using this together with \eqref{241108903} and again invoking \cite[Theorem 2.10]{HKP2021}, we further deduce that for all $k\in\bN$, $p\in(1,\infty)$, and $\theta\in(d-p-1,d-1)$, the solutions satisfying \eqref{240429814} coincide. 
		We denote this unique solution by $Rf$.
		Observe the following results:
		\begin{itemize}
			\item[-] Since \eqref{240429814} holds for all $k\in\bN$, it is also obtained by \eqref{241031713} that
			\begin{align*}
				\sum_{i=0}^k\int_0^T\sup_{x\in\Omega}|\rho^{i-1+d/p} D_x^iu(t,x)|^p\dd t\lesssim \sum_{i=0}^{k+1}\int_0^T\int_\Omega |\rho^iD_x^i u|^p\rho^{-p}\dd x\dd t\,,
			\end{align*}
			for any $p>d$, and $k\in\bN$.
			
			\item[-] Since $f\in C_c^{\infty}\big((0,T]\times \Omega\big)$, there exists $\epsilon>0$ such that $f(t,\cdot)\equiv0$ for all $0<t\le\epsilon$.
			By the uniqueness of $Rf$ on the time interval $[0,\epsilon]$, we have $Rf(t,\cdot)=0$ for all $0<t\le\epsilon$.
		\end{itemize}

		Recall that, since $f\in C_c^{\infty}\big((0,T]\times \Omega\big)$, we have $f(t,x)=I_t^n \big(\partial_t^nf\big)(t,x)$.
		Denote $u_n:=R\big(\partial_t^nf\big)$.
		One can observe that for any $n\in\bN$, $I^n_t u_n$ is also solves \eqref{2206301251}.
		Therefore, by the uniqueness of solutions to \eqref{2206301251}, we obtain that $u(t,x)=I_t^nu_n(t,x)$.
		This implies that for any $m,\,n\in\bN_0$, $\partial_t^nD_x^m u=I_t\big(D_x^m u_{n+1}\big)$. 
		In addition, for any $p>d$,
		\begin{align*}
			&\sup_{x\in \Omega}|\rho(x)^{m-1+(d-1)/p}\partial_t^nD_x^mu(t,x)|^p\\
			\leq\,&T^{p-1}\int_0^T\sup_{x\in\Omega}|\rho(x)^{m-1+(d-1)/p}D_x^mu_{m+1}(t,x)|^p\dd t<\infty\,.
		\end{align*}
		It then follows that for any $t\in(0,T]$, $u(t,x)\rightarrow 0$ as $x\in\partial\Omega$.
		Moreover, by \eqref{240429814}, we have
		$$
		\int_0^T\int_\Omega |u|^{p-1}|D_x^2u|\dd x\dd t\leq \int_0^T\int_\Omega |u|^p\rho^{-2}\dd x\dd t\cdot \int_0^T\int_\Omega |D_x^2 u|^p\rho^{2p-2}\dd x\dd t<\infty\,.
		$$			 
	\end{proof}
	
	\begin{prop}\label{240827811}
		Let $0<\frac{1}{p}\leq \beta<\alpha\leq 1$ and $\gamma\in\bR$. 
		For any $u\in\cH_p^{\gamma+2}(\bR^d,T)$, the following estimates hold for all $0\le s<t\le T$ and $A>0$:
		\begin{align}\label{2408261012}
			\begin{aligned}
				&\big\|u(t,\cdot)-u(s,\cdot)\big\|_{B_p^{\gamma+2-2\beta/\alpha}(\bR^d)}\\
				\leq\,& N|t-s|^{\beta-1/p} A^{-\beta}\Big(\|u\|_{\bH_{p}^{\gamma+2}(\bR^d,T)}+A^{\alpha}\|\partial_t^\alpha u\|_{\bH_{p}^{\gamma}(\bR^d,T)}+A^{1/p}\|u(0,\cdot)\|_{B_p^{\gamma+2-2/(p\alpha)}(\bR^d)}\Big)\,,
			\end{aligned}
		\end{align}
		and
		\begin{align}\label{2408261013}
			\big\|u(t,\cdot)-u(s,\cdot)\big\|_{H_p^{\gamma}(\bR^d)}
			\leq N|t-s|^{\alpha-1/p}\|\partial_t^\alpha u\|_{\bH_{p}^{\gamma}(\bR^d,T)}\,.
		\end{align}
		Here, $N=N(d,p,\gamma,\alpha,\beta)$.
	\end{prop}
	
	\begin{proof}
		Considering $w:=(1-\Delta)^{\gamma/2}u$, it suffices to prove the case $\gamma=0$.
		Without loss of generality, we assume $A=1$, since the dependence on $A$ (\textit{i.e.}, the exponent of $A$) can be deduced from the scaling in time.
		
		Suppose for the moment that \eqref{2408261012} holds for $\beta=0$ and that \eqref{2408261013} holds.
		Note that 
		\begin{align}\label{2408271245}
			\big[L_p(\bR^d),B_p^{2-2/(p \alpha)}(\bR^d)\big]_{\theta,p}=B_p^{\theta(2-2/(p \alpha))}\,.
		\end{align}
		Indeed, 
		\begin{align*}
			&L_p(\bR^d)\supset B_p^{0}(\bR^d)\quad\text{and}\quad H_p^{2-2/(p \alpha)}(\bR^d)\supset 	B_p^{2-2/(p \alpha)}(\bR^d)\quad\text{if $p\leq 2$}\,,\\
			&L_p(\bR^d)\subset B_p^{0}(\bR^d)\quad\text{and}\quad H_p^{2-2/(p \alpha)}(\bR^d)\subset 	B_p^{2-2/(p \alpha)}(\bR^d)\quad\text{if $p\geq 2$}\,,
		\end{align*}
		and
		$$
		\big[L_p(\bR^d),H_p^{2-2/(p \alpha)}(\bR^d)\big]_{\theta,p}=\big[B_p^0(\bR^d),B_p^{2-2/(p \alpha)}(\bR^d)\big]_{\theta,p}=B_p^{\theta(2-2/(p \alpha))}\,.
		$$			
		Here we use the standard interpolation and embedding properties of $H_p^\gamma(\bR^d)$ and $B_p^\gamma(\bR^d)$; see \cite[Sections 2.3 and 2.4]{triebel2}.
		Assuming \eqref{2408261012} for $\beta=\frac{1}{p}$ and \eqref{2408261013}, we obtain
		\begin{align*}
			\|u(t)-u(s)\|_{B_p^{\theta(2-2/(\alpha p))}(\bR^d)}\,&\lesssim\|u(t)-u(s)\|_{L_p(\bR^d)}^{1-\theta}\|u(t)-u(s)\|_{B_p^{2-2/(\alpha p)}(\bR^d)}^\theta\\
			&\lesssim (t-s)^{(1-\theta)(\alpha-1/p)}\|u\|_{\cH_p^2(\bR^d,T)}\,,
		\end{align*}
		where the first inequality follows from \eqref{2408271245} together with \cite[Theorem 1.3.3.(g)]{triebel}.
		By taking $\theta=\frac{\alpha-\beta}{\alpha-1/p}\in(0,1)$, we obtain that \eqref{2408261012} also holds for $\frac{1}{p}<\beta<\alpha$.
		Therefore, it suffices to prove \eqref{2408261012} for $\beta=0$ and \eqref{2408261013}.
		
		\textbf{Proof of \eqref{2408261012} for $\beta=\frac{1}{p}$:}
		In this step, we employ results from \cite{pruss2012evolutionary, zacher2003quasilinear,zacher2005maximal}.
		For this, we note that for any positive multiplier $m\in C^\infty(\bR^d\setminus\{0\})$ satisfying $|D_\xi^km(\xi)|\lesssim_k |\xi|^km(\xi)$, the operator $f\mapsto \big(m(\xi)\widehat{f}(\xi)\big)^{\vee}$ is $\mathcal{R}$-sectorial on $L_p(\bR^d)$ with $\mathcal{R}$-angle $0$ (see \cite[Example 10.1.3, Theorem G.5.2]{hytonen2018analysis}).

		Consider the equation
		$$
		\partial_t^\alpha v=\Delta v-v\quad\quad \text{in}\quad \bR_+\times \bR^d\quad;\quad v(0,\cdot)=u(0,\cdot)\in B_p^{2-2/(p\alpha)}(\bR^d)\,,
		$$
		where $\partial_t^\alpha v$ is understood in the sense of \eqref{240419514}.
		By \cite[Theorem 3.1]{pruss2012evolutionary}, there exists a solution $v\in C\big([0,\infty);B_p^{2-2/(p\alpha)}(\bR^d)\big)$ such that 
		\begin{align}\label{241108740}
			\|v\|_{C([0,\infty);B_p^{2-2/(p\alpha)}(\bR^d))}\lesssim \|u(0,\cdot)\|_{B_p^{2-2/(p\alpha)}(\bR^d)}\leq \|u\|_{\cH_p^{\alpha,2}(\bR^d,T)}\,.
		\end{align}
		Consider the equation
		\begin{align*}
			\partial_t^\alpha w=\Delta w-w+\big(\partial_t^\alpha u-\Delta u+u\big)1_{(0,T]\times \bR^d}\quad \text{in}\quad \bR_+\times \bR^d\quad;\quad w(0,\cdot)=0\,,
		\end{align*}
		in the sense of \eqref{240419514}.
		By combining \cite[Theorem 3.1.1 and Remark 3.1.1.(v)]{zacher2003quasilinear} and \cite[Theorem 3.6]{zacher2005maximal}, we obtain the existence of a solution $w\in C\big([0,\infty);B_p^{2-2\beta/\alpha}(\bR^d)\big)$ satisfying
		\begin{align}\label{241108741}
			\begin{aligned}
				&\big\|w\big\|_{C([0,\infty);B_p^{2-2/(p \alpha)}(\bR^d))}
				\lesssim \big\|\partial_t^\alpha u-\Delta u+u\big\|_{L_p\left((0,T]\times \bR^d\right)}\lesssim \|u\|_{\cH_{p}^{\alpha,2}(\bR^d,T)}\,.
			\end{aligned}
		\end{align}

		The proof will be complete once we show that $u=v+w$.
		Note that $u\in L_p\big((0,T)\times \bR^d\big)$ and $v+w\in C\big([0,T];B_{p}^{2-2/(p\alpha)}(\bR^d)\big)\subset L_p\big((0,T)\times \bR^d\big)$.
		Put $\widetilde{u}=(1-\Delta)^{-1}\big(u-(v+w)\big)\in \bH_p^2(\bR^d,T)$.
		Since 
		\begin{align}\label{241013339}
			\partial_t^\alpha \widetilde{u}=\Delta \widetilde{u}-\widetilde{u}\quad\text{in}\quad (0,T]\times \bR^d\quad;\quad \widetilde{u}(0,\,\cdot\,)\equiv 0
		\end{align}
		in the sense of \eqref{240419514}, we obtain that $\widetilde{u}\in \mathring{\cH}_p^{\alpha,2}(\bR^d,T)$ (see Lemma \ref{240426524}).
		As established in \cite[Theorem 2.4]{DONG2019289}, the solution to equation \eqref{241013339} in $\mathring{\cH}_p^{\alpha,2}(\bR^d,T)$ is unique.
		(Note that $\bH^{\alpha,2}_{p,0}(\bR^d_T)$ in \cite{DONG2019289} coincides with $\mathring{\cH}_p^{\alpha,2}(\bR^d,T)$ by the same argument as in the proof of Lemma \ref{240911115}.(1).)
		Therefore $\widetilde{u}(t)\equiv 0$ for almost all $t\in[0,T]$.
		Since $u(t)$ is a $\cD'(\bR^d)$-valued continuous function for $t\in[0,T]$ (see the result in Step 1), and $u=v+w$ for almost all $t\in[0,T]$, we conclude that $u(t,\cdot)\equiv v(t,\cdot)+w(t,\cdot)$ for all $t\in [0,T]$. 
		Combining \eqref{241108740} and \eqref{241108741}, we complete the proof.

		\textbf{Proof of \eqref{2408261013}:}
		When $\alpha=1$, this estimate follows from a direct calculation.
		It remains to prove the case $0<\alpha<1$.
		It follows from \eqref{240419514} and the Minkowski inequality that 
		\begin{align}
			\begin{aligned}\label{2408271240}
				&\|u(t)-u(s)\|_p\\
				\lesssim_{\alpha,p} &\|\partial_t^\alpha u\|_{L_p((0,T)\times \bR^d)}\bigg(\int_s^t(t-r)^{-(1-\alpha)p'}\dd r+\int_0^s\Big|(s-r)^{-1+\alpha}-(t-r)^{-1+\alpha}\Big|^{p'}\dd r\bigg)^{1/p'}\,,
			\end{aligned}
		\end{align}
		where $p':=\frac{p}{p-1}$.
		Since $(1-\alpha)p'<1$, we have
		\begin{align}
			\bigg(\int_s^t(t-r)^{-(1-\alpha)p'}\dd r\bigg)^{1/p'}=N(\alpha,p)(t-s)^{\alpha-1/p}\,.
		\end{align}
		For the last term in \eqref{2408271240}, we note that if $s>0$, then for any $r,\,A>0$, 
		$$
		r^{-s}-(A+r)^{-s}\lesssim_s \left(\int_0^{A}(r+\ell)^{-s-1}\dd \ell\right)\wedge r^{-s}\lesssim_s (A\wedge r)r^{-1-s}\simeq_s \frac{Ar^{-s}}{A+r}\,.
		$$
		This implies that 
		\begin{align}
			&\int_0^s\Big|r^{-1+\alpha}-(t-s+r)^{-1+\alpha}\Big|^{p'}\dd r\lesssim (t-s)^{p'}\int_0^s\frac{r^{-(1-\alpha)p'}}{(t-s+r)^{p'}}\dd r\,.
		\end{align}
		Since $(1-\alpha) p'<1$ and $(2-\alpha)p'>1$, we have 
		\begin{align}\label{2408271241}
			\int_0^\infty\frac{r^{-(1-\alpha)p'}}{(t-s+r)^{p'}}\dd r= N(\alpha,p) (t-s)^{1-(2-\alpha)p'}\,.
		\end{align}
		Combining \eqref{2408271240}-\eqref{2408271241}, we obtain \eqref{2408261012} for $\beta=\frac{1}{p}$.
	\end{proof}

		\section*{Conflict of Interest Statement}
		The author declares that there is no conflict of interest.

%

	\end{document}